\documentclass[12pt]{amsart}

\usepackage{stmaryrd}

\usepackage{graphicx}
\usepackage{amsmath}
\usepackage{mathtools}
\usepackage{amssymb}
\usepackage{amsfonts}
\usepackage{amsthm}
\usepackage{cancel}
\usepackage[all]{xy}
\usepackage[neveradjust]{paralist}
\usepackage{enumitem}
\usepackage{multicol}
\usepackage{mathrsfs}
\usepackage{mathdots}
\usepackage[pagebackref,colorlinks=true,linkcolor=blue,citecolor=blue]{hyperref}
\usepackage{tikz-cd}
\usepackage{tikz}
\usepackage{blkarray}
\usepackage{scalerel,stackengine}

\UseRawInputEncoding

\stackMath
\newcommand\reallywidehat[1]{%
\savestack{\tmpbox}{\stretchto{%
  \scaleto{%
    \scalerel*[\widthof{\ensuremath{#1}}]{\kern-.6pt\bigwedge\kern-.6pt}%
    {\rule[-\textheight/2]{1ex}{\textheight}}
  }{\textheight}%
}{0.5ex}}%
\stackon[1pt]{#1}{\tmpbox}%
}

\tikzset{
  symbol/.style={
    draw=none,
    every to/.append style={
      edge node={node [sloped, allow upside down, auto=false]{$#1$}}}
  }
}

\usepackage{color}
\usepackage{blkarray}

\newcommand{\Z}{\mathbb{Z}}

\newcommand{\R}{\mathbb{R}}
\newcommand{\BC}{\mathbb{C}}

\newcommand{\inv}{^{-1}}

\newcommand{\SL}{\mathrm{SL}}
\newcommand{\GL}{\mathrm{GL}}
\newcommand{\SO}{\mathrm{SO}}
\newcommand{\OO}{\mathrm{O}}
\newcommand{\Sp}{\mathrm{Sp}}
\newcommand{\St}{\mathrm{St}}
\newcommand{\ord}{\mathrm{ord}}

\newcommand{\Rep}{\underline{\mathrm{Rep}}}

\newcommand{\Jord}{\mathrm{Jord}}

\DeclareMathOperator{\sgn}{sgn}

\newcommand{\half}[1]{\frac{#1}{2}}

\renewcommand{\implies}{\Rightarrow}

\newcommand{\comment}[1]{}
\newcommand{\EE}{\mathcal{E}}
\newcommand{\FF}{\mathcal{F}}

\newcommand{\MM}{\mathcal{M}}

\newcommand{\Moe}{\mathrm{M\textnormal{\oe}}}
\newcommand{\Ato}{\mathrm{Ato}}
\newcommand{\Xu}{\mathrm{Xu}}

\newcommand{\com}{\mathrm{com}}
\newcommand{\NV}{\mathrm{NV}}
\newcommand{\AZ}{\mathsf{AZ}}

\newcommand{\cO}{\mathcal{O}}

\newcommand{\ABV}{\operatorname{ABV}}
\newcommand{\unip}{\operatorname{unip}}
\newcommand{\Lus}{\operatorname{Lus}}
\newcommand{\LLC}{\operatorname{LLC}}

\newcommand{\WF}{\operatorname{WF}}

\newcommand{\Frob}{\operatorname{Frob}}

\newcommand{\weak}{\operatorname{weak}}

\newcommand{\mfr}[1]{\mathfrak{#1}}

\newcommand{\good}{\mathrm{good}}
\newcommand{\pure}{\mathrm{pure}}
\newcommand{\temp}{\mathrm{temp}}

\newtheorem{thm}{Theorem}[section]
\newtheorem{cor}[thm]{Corollary}
\newtheorem{lemma}[thm]{Lemma}
\newtheorem{prop}[thm]{Proposition}
\newtheorem {conj}[thm]{Conjecture}
\newtheorem {assu}[thm]{Working Hypothesis}
\newtheorem {ques/conj}[thm]{Question/Conjecture}

\newtheorem{defn}[thm]{Definition}
\newtheorem{remark}[thm]{Remark}

\newtheorem{algo}[thm]{Algorithm}
\newtheorem{conv}[thm]{Convention}

\DeclareMathOperator{\supp}{supp}

\DeclareMathOperator{\Hom}{Hom}

\numberwithin{equation}{section}

\let\oldbullet\bullet
\renewcommand{\bullet}{{\vcenter{\hbox{\tiny$\oldbullet$}}}}

\usepackage[textsize=footnotesize,textwidth=3.2cm]{todonotes}

\begin{document}
\renewcommand{\theequation}{\arabic{equation}}
\numberwithin{equation}{section}

\title[Intersection of Arthur packets for even orthogonal groups]
{The intersection of local Arthur packets for even orthogonal groups and applications}

\author{Alexander Hazeltine
}
\address{Department of Mathematics\\
Columbia University\\
New York, NY 10027, USA}
\email{alh2252@columbia.edu}

\author{Baiying Liu}
\address{Department of Mathematics\\
Purdue University\\
West Lafayette, IN, 47907, USA}
\email{liu2053@purdue.edu}

\author{Chi-Heng Lo}
\address{Department of Mathematics\\
National University of Singapore\\
119076, Singapore}
\email{{ch\_lo@nus.edu.sg}}


\date{\today}


\thanks{The research of the first-named author is supported by an AMS-Simons Travel Grant. The research of the
second-named author is partially supported by the NSF Grants DMS-1848058 and the Simons Foundation: Travel Support for Mathematicians.}

\begin{abstract}
We give algorithms to explicitly compute local Arthur packets and their intersections for even orthogonal groups $G_n$ over $p$-adic fields using extended multi-segments. In addition, we provide an algorithm which determines whether a representation of $G_n$ lies in a local Arthur packet. As applications, these results are used to identify distinguished parameters, prove the closure ordering conjecture, the enhanced Shahidi conjecture, the weak local Arthur packet conjecture, and the good parity unitary dual conjecture for $G_n$. 
\end{abstract}

\maketitle

\tableofcontents

\section{Introduction}

Let $F$ be a non-Archimedean local field of characteristic zero. In the fundamental work \cite{Art13}, Arthur established the endoscopic classification of representations of quasi-split symplectic and orthogonal groups over $F$. The main objects underlying the classification are  \emph{local Arthur packets} $\Pi_{\psi}$, which are finite multi-sets of irreducible unitary representations of these groups attached to local Arthur parameters $\psi$ (see \S \ref{sec: lap for symplectic}). Local Arthur packets are important in number theory as they describe the local components of square-integrable automorphic representations (\cite[Theorem 1.5.2]{Art13}). They are also important in representation theory since the representations in each local Arthur packet are unitary and are expected to form the building blocks for the unitary dual of these classical groups (\cite[Conjecture 1.2]{HJLLZ25}).

For either purpose, knowing that $\Pi_{\psi}$ exists is not enough. We desire explicit computations of $\Pi_{\psi}$ in terms of the Langlands classification, that is, an algorithm which produces from $\psi$ the Langlands data  (see \S\ref{sec Langlands classification})
\[ \pi=L(\Delta_{\rho_1}[x_1,y_1],\dots,\Delta_{\rho_r}[x_r,y_r];\pi_{temp})\]
for every $\pi \in \Pi_{\psi}$. We also want a converse algorithm which, starting from the Langlands data  of an irreducible representation $\pi$, decides whether $\pi$ is of Arthur type, i.e., $\pi\in\Pi_\psi$ for some $\psi$. Note that local Arthur packets may intersect, and so we also want to compute the set
\[ \Psi(\pi):=\{ \psi \ | \ \pi \in \Pi_{\psi}\}.\]
Knowing the combinatorial structure of $\Psi(\pi)$ is useful in certain applications, e.g., for studying the upper bound of the wavefront set of $\pi$ (\cite{HLLS24}).

The explicit parameterization of local Arthur packets for classical groups, including non-quasi-split pure inner forms, was initiated by M{\oe}glin. In a series of papers (\cite{Moe06a, Moe06b, Moe09a, Moe10, Moe11}) M{\oe}glin gave a representation theoretic construction of $\Pi_{\psi}$ and proved that it is multiplicity free. More precisely, M{\oe}glin introduced a combinatorial datum $\MM$, which we call a M{\oe}glin parameter (Definition \ref{def Moeglin parameter}), and attached a representation $\pi_{\Moe}(\MM)$ to $\MM$, which is either irreducible or zero. M{\oe}glin showed that the nonzero ones exhaust $\Pi_{\psi}$ (Theorem \ref{thm Moeglin construction}). Furthermore, Xu gave an explicit combinatorial algorithm which determines whether $\pi_\Moe(\MM)$ is nonzero (\cite{Xu21b}). Consequently, these results can be used to compute the size $|\Pi_\psi|$. However, M{\oe}glin's description does not give an algorithm to determine the Langlands data  for $\pi_{\Moe}(\MM)$. In particular, it does not give a combinatorial criterion to decide when $\pi_\Moe(\MM_1)=\pi_\Moe(\MM_2)$ for two different M{\oe}glin parameters.

This problem was solved by Atobe (\cite{Ato22a}) for split symplectic and odd special orthogonal groups, by reformulating M{\oe}glin's construction in terms of \emph{extended multi-segments} (see Definition \ref{def multi-segment}).
For an extended multi-segment $\EE$ of these groups, Atobe attached a local Arthur parameter $\psi_{\EE}$ and a representation $\pi_{\Ato}(\EE)$. The representation $\pi_{\Ato}(\EE)$ is either irreducible or zero, and there is an explicit algorithm for computing the Langlands data  of $\pi_{\Ato}(\EE)$ using the formula for highest derivatives given in \cite{AM20}. One of the main results of \cite{Ato22a} is the following equality (for good parity $\psi$; see \S\ref{sec Local Arthur packets for orthogonal groups}):
\begin{align}\label{eq Ato A-packet intro}
    \bigoplus_{\pi \in \Pi_{\psi}} \pi = \bigoplus_{\psi_{\EE}=\psi} \pi_{\Ato}(\EE)
\end{align}
(\cite[Theorem 1.2]{Ato22a}; see Theorem \ref{thm Sp}(2)). Furthermore, Atobe compared the two parametrizations by giving an explicit map $\MM \mapsto \EE_{\MM}$ (Definition \ref{defn Moeglin to Atobe dictionary}) and proved that $\pi_{\Ato}(\EE_{\MM})=\pi_{\Moe}(\MM)$ (\cite[Theorem 6.6]{Ato22a}; see Theorem \ref{thm Sp}(4)). 

Both the notion of an extended multi-segment and the assignment $\EE \mapsto \psi_{\EE}$ make perfect sense for pure inner forms of quasi-split even orthogonal groups. The construction of $\pi_{\Ato}(\EE)$, however, does not carry over since the formula for highest derivatives is currently unavailable for even orthogonal groups.

In this paper we circumvent this obstruction by using the local theta correspondence. Let $G_n$ denote a pure inner form of a quasi-split even orthogonal group $\OO_{2n}$. Given an extended multi-segment $\EE$ of $G_n$ and an odd integer $\alpha \gg 0$, we attach an extended multi-segment $\EE_{\alpha}$ of the split symplectic group $\Sp_{2n+\alpha-1}(F)$ (Definition \ref{defn EE_alpha}). Then we define $\pi_{\theta}(\EE)$ to be the unique representation of $G_n$ whose theta lift is $\pi_{\Ato}(\EE_{\alpha})$ (Definition \ref{def orthog rep of segment}). We show that this is well-defined (independent of the choice of $\alpha$). Here is our first theorem.
\begin{thm}\label{thm intro}
    Let $\psi$ be a local Arthur parameter of $G_n$ of good parity. The following holds.
    \begin{enumerate}
        \item (Theorem \ref{thm Atobe reformulation orthog}) We have 
    \[ \bigoplus_{\pi \in \Pi_{\psi}} \pi = \bigoplus_{\psi_{\EE}=\psi} \pi_{\theta}(\EE).\]
    \item (Theorem \ref{thm non-vanishing}) We determine a combinatorial algorithm which decides when $\pi_{\theta}(\EE)\neq 0$.
    \item (Theorem \ref{cor Moeglin to Atobe dictionary orthog}) Using the same map $\MM \mapsto \EE_{\MM}$ as in \cite[Theorem 6.6]{Ato22a}, we have $\pi_{\theta}(\EE_{\MM})=\pi_{\Moe}(\MM)$.
    \end{enumerate}
\end{thm}

The proof of the above theorem is based on M{\oe}glin's result on the Adams conjecture. More precisely, given any parameter $\MM$ of $G_n$ and $\alpha \gg 0$,
M{\oe}glin explicitly constructed a parameter $\MM_{\alpha}$ of $\Sp_{2n+\alpha-1}(F)$ such that $\theta_{-\alpha}(\pi_{\Moe}(\MM))= \pi_{\Moe}(\MM_{\alpha})$ (\cite[Theorem 5.1]{Moe11c}; see Theorem \ref{thm Moeglin Adams}). Then we apply Atobe's map $\MM_{\alpha} \mapsto \EE_{\MM_{\alpha}}$ to obtain an extended multi-segment of $\Sp_{2n+\alpha-1}(F)$ (\cite[Theorem 6.6]{Ato22a}; see Theorem \ref{thm Sp}(4)). Finally, we define the map $\EE \mapsto \EE_{\alpha}$ (Definition \ref{defn EE_alpha}) so that the following diagram commutes.
\[
\begin{tikzcd}
    \MM \ar[rr,mapsto] \ar[d,mapsto] && \ar[d,mapsto]\MM_{\alpha}\\
    \EE_{\MM}\ar[r,mapsto,"\textrm{Defn \ref{defn EE_alpha}}"] & (\EE_{\MM})_{\alpha}\ar[r,equal,"\eqref{eqn Moeglin to Atobe dictionary}"]& \EE_{(\MM_\alpha)} 
\end{tikzcd}
\]
Here the upper horizontal arrow is defined by M{\oe}glin and the vertical arrows follow Atobe's recipe.

Next we want to determine $\Psi(\pi).$ We begin by recalling its history for split symplectic and odd special orthogonal groups. For these groups, in \cite{HLL22}, we showed that the following list of operators on $\EE$ determines $\Psi(\pi)$:
\begin{enumerate}
    \item[(i)] $R_k$, the row exchange of extended segments (Definition \ref{def row exchange});
    \item[(ii)] $ui_{i,j}$, the union-intersection of extended segments (Definition \ref{def ui});
    \item[(iii)] $dual \circ ui_{i,j} \circ dual$, where $dual$ is the Aubert-Zelevinsky dual on extended multi-segments (Definition \ref{dual segment});
    \item[(iv)] $dual_k^{\pm}$, the partial dual, which applies only when $a_k+b_k$ is odd (Definition \ref{def partial dual}).
\end{enumerate}
More precisely, we have  $\pi_\Ato(\EE_1)=\pi_\Ato(\EE_2)$ if and only if $\EE_1=T_1\circ\dots \circ T_r(\EE_2)$, where each $T_i$ is one of the above operators or an inverse (\cite[Theorem 1.4]{HLL22}). A logically equivalent form of this result is proved in \cite{Ato23} independently. These operators apply to extended multi-segments of even orthogonal groups as well and we prove the analogous result by reducing to the symplectic case and using Howe duality (Theorem \ref{thm Howe duality}).
\begin{thm}[Theorem \ref{thm O}]\label{thm intro O}
    Let $\EE_1$ and $\EE_2$ be extended multi-segments of $G_n$ with $\pi_\theta(\EE_1) \neq 0$. Then  $\pi_\theta(\EE_1)=\pi_\theta(\EE_2)$ if and only if $\EE_1=T_1\circ\dots \circ T_r(\EE_2)$, where each $T_i$ is one of the above operators or an inverse.
\end{thm}

Now, we turn towards improving the algorithm to compute the Langlands data  of $\pi_\theta(\EE)$. We recall that by definition, to compute $\pi_{\theta}(\EE)$, we first compute the Langlands data of $\pi_\Ato(\EE_\alpha)$, and then use the explicit computation of the theta correspondence given in \cite{AG17a, BH21} (see Algorithm \ref{algo orthog Langlands Classification}). It is desirable to have a more direct algorithm. Using the intersection theory we give an intrinsic combinatorial algorithm which attaches to $\EE$ a representation $\pi_{\com}(\EE)$ by computing the Langlands data directly from $\EE$ (Algorithm \ref{algo pi_com}). We prove the following theorem.

\begin{thm}\label{thm intro 2}
    Let $\EE$ be an extended multi-segment of $G_n$. The following holds.
    \begin{enumerate}
        \item (Theorem \ref{thm theta=com}) We have $\pi_{\theta}(\EE)=\pi_{\com}(\EE)$.
        \item (Theorem \ref{thm in L-packet}) The representation $\pi_{\com}(\EE)$ lies in the $L$-packet attached to $\psi_{\EE}$ if and only if a combinatorial condition $(L)$ for $\EE$ holds (Definition \ref{def (L)}).
        \item (Algorithm \ref{algo Arthur type pi^-}) Given the Langlands data  of a representation $\pi$, we give an algorithm to determine the set 
        \[ \{\EE \ | \ \pi_{\com}(\EE)=\pi\}.\]
        In particular, by Part (1) and Theorem \ref{thm intro}(1), it determines $\Psi(\pi)$ and whether $\pi$ is of Arthur type or not.
    \end{enumerate}
\end{thm}

We remark that the assignment $\EE \mapsto \pi_{\com}(\EE)$ can be defined in greater generality, including symplectic, odd special orthogonal, metaplectic and unitary groups. Part (2) and the algorithm in Part (3) of the theorem then come from combinatorial arguments following the definition of $\pi_{\com}(\EE)$.
For split symplectic and odd special orthogonal groups, we have $\pi_{\com}(\EE)= \pi_{\Ato}(\EE)$ (\cite[Algorithm 6.3]{HJLLZ25}). Thus, together with Atobe's result \eqref{eq Ato A-packet intro}, the algorithm in Part (3) determines $\Psi(\pi)$ (\cite[Algorithm 6.11]{HJLLZ25}). We expect that the above conclusion also applies to pure inner forms of odd special orthogonal groups and metaplectic groups. See \S \ref{sec meta} for how the methods of this paper can be applied to the last two cases. We also expect that the conclusion holds for unitary groups, which is a work in progress of the third named author.

Now we discuss the proof of Theorem \ref{thm intro 2}(1). We make use of a result of M{\oe}glin which we sketch here. For a parameter $\MM$ of $G_n$ and an index $j$, M{\oe}glin defined another parameter $ (add_{j}^{\Moe})^{-1}(\MM)$ of $G_m$ with $m<n$
(see \S \ref{subsec Moeglin operators}). If both $\pi:=\pi_{\Moe}(\MM)$ and $\pi^-:=\pi_{\Moe}((add_{j}^{\Moe})^{-1}(\MM))$ are nonzero, then M{\oe}glin showed that $\pi$ and $\pi^-$ are related by a specific parabolic induction (\cite[Theorem A]{Moe11b}; see Theorem \ref{thm Moeglin + embed}). Under certain conditions, the Langlands data of $\pi$ can therefore be read from that of $\pi^-$. We translate M{\oe}glin's result into the language of extended multi-segments using Theorem \ref{thm intro}(3). This involves an operator $add_j^{-1}$ on extended multi-segments (Definition \ref{def shift and add}(2)). We show that the following diagram commutes.

\[
\begin{tikzcd}
  \pi_{\Moe} ( \MM)\ar[rr, mapsto] \ar[d, equal] & &\pi_{\Moe}((add_{j}^{\Moe})^{-1}(\MM)) \ar[d,equal]\\
    \pi_{\theta}(\EE_{\MM})\ar[r, mapsto] & \pi_{\theta}(add_j^{-1}(\EE_{\MM})) \ar[r,equal,"\textrm{Thm. \ref{thm add = +}}"]& \pi_{\theta}(\EE_{(add_{j}^{\Moe})^{-1}(\MM)}).
\end{tikzcd}
\]
Our key insight is that if $\pi_\theta(\EE)\neq 0$, then by replacing $\EE$ with a certain extended multi-segment $\EE^{|max|}$ which satisfies $\pi_\theta(\EE)=\pi_\theta(\EE^{|max|})$ (and $\psi_{\EE^{|max|}}=\psi^{max}(\pi_{\theta}(\EE))$ from application (1) below), we can guarantee that $\pi_{\theta}(add_j^{-1}(\EE^{|max|}))\neq 0$ (Proposition \ref{prop max triangle}). From M{\oe}glin's result, the above diagram, and the non-vanishing result, we obtain a combinatorial algorithm to compute the Langlands data  of $\pi_{\theta}(\EE)$, which is precisely the algorithm for $\pi_{\com}(\EE)$.

With the intersection theory, we prove the following applications. These results are the analogues of results which are known to hold for split symplectic and odd special orthogonal groups.
\begin{enumerate}
    \item[(1)] \emph{Distinguished members of $\Psi(\pi)$.}
    
    \noindent When $\Psi(\pi)$ is non-empty, we show that there exist distinguished members $ \psi^{max}(\pi),$ $\psi^{min}(\pi) \in \Psi(\pi)$ characterized as follows. We introduce five different orderings on $\Psi(\pi)$: the operator ordering $\geq_O$ (Definition \ref{def operators on parameters}(2)), the closure ordering $\geq_C$ (Definition \ref{def C ordering Arthur}), the two partition orderings  $\geq_A$ and $\geq_D$ (Definitions \ref{A ordering} and \ref{def Dordering}), and the ordering $\geq_N$ (Definition \ref{M ordering}) coming from the poles of normalized intertwining operators. We show that for any $X \in \{O, C, A,D,N\}$, $\psi^{max}(\pi)$ and $\psi^{min}(\pi)$ are the unique members of $\Psi(\pi)$ such that the following inequality holds for any  $\psi \in \Psi(\pi)$ (Theorems \ref{thm psi max and min}, \ref{thm Jiang's partition}, \ref{thm D order}, \ref{thm N ordering}, \ref{thm O implies C})
    \[ \psi^{max}(\pi) \geq_X \psi \geq_{X} \psi^{min}(\pi).\]
    
    \item[(2)] \emph{Closure ordering conjecture $($\cite[Conjecture 3.1]{Xu24}$)$.} 

    \noindent The closure ordering $\geq_C$ is defined on $L$-parameters (Definition \ref{def C ordering}). We prove that if $\pi\in\Pi(G_n)$ is of Arthur type, then 
    \[ \phi_{\pi} \geq_C \phi_{\psi^{max}(\pi)},\]
    where $\phi_{\pi}$ is the $L$-parameter of $\pi$, and $\phi_{\psi^{max}(\pi)}$ is the $L$-parameter associated to $\psi^{max}(\pi)$ (Theorem \ref{thm phi pi > phi psi}). Together with the previous application, we see that among $\{\phi_{\psi}\}_{\psi \in \Psi(\pi)}$, the $L$-parameter $\phi_{\psi^{max}(\pi)}$ is ``closest" to the $L$-parameter $\phi_{\pi}$. This supports the claim that $\psi^{max}(\pi)$ is a canonical candidate for ``the'' local Arthur parameter of $\pi$.
    A further consequence is that no local Arthur packet of a quasi-split even orthogonal group properly contains another (Theorem \ref{thm noncontainment}).
    \item[(3)] \emph{The enhanced Shahidi conjecture} (\cite[Conjecture 1.5]{LS22}). 
    
    \noindent For quasi-split even orthogonal groups, we show that a local Arthur packet is tempered if and only if it has a generic member (Theorem \ref{thm enhanced Shahidi}).
    \item[(4)] \emph{Weak local Arthur packet conjecture}.
    
    \noindent  Given a nilpotent orbit $\mathcal{O}$ of the dual group, the weak local Arthur packet $\Pi_{\mathcal{O}}^{\textrm{weak}}$ is defined in \cite{CMO24} using wavefront sets (see Definition \ref{def weal LAP}). They conjectured that $\Pi_{\mathcal{O}}^{\textrm{weak}}$ is a union of local Arthur packets. We determine a short list of desiderata for local Arthur packets (Working Hypothesis \ref{assu Arthur ABV}) under which this conjecture, and its analogue for ABV packets, hold for a general group. In particular, it holds for pure inner forms of split even orthogonal groups for $p \gg 0$ (Corollary \ref{cor weak LAP orthog}). 
    \item[(5)] \emph{Unitarity and Arthur type}. 
    
    \noindent In \cite[Conjecture 1.2]{HJLLZ25}, we conjectured that a representation of a classical group of good parity is unitary if and only if it is of Arthur type. This conjecture is proved in \cite{AM25} for split symplectic and odd special orthogonal groups. We generalize one direction to pure inner forms of quasi-split even orthogonal groups (Theorem \ref{thm unitary O}). Namely, good parity unitary representations are of Arthur type.
\end{enumerate}

The above applications (1)-(4) are proved for split symplectic and odd special orthogonal groups in the previous work of the authors (\cite{HLL22, HLL24, HLLZ25, LL24}). For even orthogonal groups, the above applications follow by similar arguments, based on the intersection theory of local Arthur packets we developed in this paper. For application (5), we use the local theta correspondence and our algorithm for determining Arthur type reduces the problem to symplectic groups, which is solved in \cite{AM25}. Recently, we learned from Chen and Zou that they also proved application (5), including metaplectic groups and non-quasi-split special odd orthogonal groups, using global theta lifting (\cite{CZ26}).

Finally, we remark on local Arthur packets for even special orthogonal groups $G_n^{\circ}$. For a local Arthur parameter $\psi$ of $G_n^{\circ}$, the local Arthur packet $\Pi_{\psi}(G_n^{\circ})$ is a collection of outer conjugacy classes of irreducible representations of $G_n^{\circ}$. If we regard $\psi$ as a local Arthur parameter of $G_n$, then the packet $\Pi_{\psi}(G_n^{\circ})$ is given by
\[ [\pi] \in \Pi_{\psi}(G_n^{\circ}) \Longleftrightarrow \textrm{Ind}_{G_n^\circ}^{G_n} \pi \subseteq \Pi_{\psi}(G_n). \]
Then our results on local Arthur packets for $G_n$ can be transported to $G_n^{\circ}$ directly.

The rest of the paper is organized as follows.  In \S \ref{sec: notation}, we set up notation and recall the Witt towers of Hermitian spaces, the Langlands classification, the Aubert-Zelevinsky involution and the local theta correspondence.  In \S \ref{sec: lap}, we recall local Arthur packets, the local Langlands correspondence for even orthogonal groups, M{\oe}glin's parameterization $\MM \mapsto \pi_{\Moe}(\MM)$, and M{\oe}glin's result on the Adams conjecture.  In \S \ref{sec: extended multi-segments}, we recall extended multi-segments and the operators acting on them, the non-vanishing criterion, and the algorithm which attaches $\pi_{\com}(\EE)$ to $\EE$. We also record the consequences of this algorithm, the map $\MM \mapsto \EE_{\MM}$, and the results for symplectic groups on which the rest of the paper relies.
In \S \ref{sec intersections for O(2n)}, we define $\pi_{\theta}(\EE)$ and prove Theorems \ref{thm intro} and \ref{thm intro O}.  In \S \ref{sec pi comb = pi theta}, we compare the operator $add_j^{-1}$ with M{\oe}glin's operator $(add_j^{\Moe})^{-1}$ and prove Theorem \ref{thm intro 2}(1).  The remaining sections are devoted to the applications: the distinguished members $\psi^{max}(\pi)$ and $\psi^{min}(\pi)$ and the orderings $\geq_A, \geq_D, \geq_N$ in \S \ref{sec distinguished members}, the closure ordering conjecture in \S \ref{sec closure order conj}, the enhanced Shahidi conjecture in \S \ref{sec enhanced Shahidi}, the weak local Arthur packet conjecture in \S \ref{sec weak local Arthur packet}, and show one direction of the good parity unitary dual conjecture in \S \ref{sec unitary rep}.  Finally, in \S \ref{sec meta}, we discuss the applicability of the methods of this paper to metaplectic groups and to pure inner forms of split odd special orthogonal groups. Additionally, in Appendix \ref{appendix: Bosnjak Stadler} we translate the main results into a modified version of extended multi-segments introduced by Bo{\v s}njak and Stadler (\cite{BS25}).

\subsection*{Acknowledgements} 
The authors are grateful for the constant support from Freydoon Shahidi. We also thank Rui Chen for helpful discussion.

\section{Notation and preliminaries}\label{sec: notation}

Let $F$ be a non-Archimedean local field of characteristic $0$ and fix $\epsilon\in\{\pm1\}$. We let
$W_n$ be a $-\epsilon$-Hermitian space of even dimension $2n$ over $F$ and $V_m$ be an $\epsilon$-Hermitian space of even dimension $2m$ over $F.$ The isometry groups of $W_n$ and $V_m$ are denoted by $G_n=G(W_n)$ and $H_m=G(V_m)$, respectively. When $\epsilon=1,$ $G_n$ is a split symplectic group and $H_m$ is an even orthogonal group. Conversely, when $\epsilon=-1,$ $G_n$ is an even orthogonal group and $H_m$ is a split symplectic group.

We often switch between the cases $\epsilon=1$ and $\epsilon=-1.$ Rather than stating which $\epsilon$ we are considering, we instead write $G_n=\Sp(W_n)$ and $H_m=\OO(V_m)$ to indicate the case of $\epsilon=1.$ Similarly, we write $G_n=\OO(W_n)$ and  $H_m=\Sp(V_m)$  to indicate the case of $\epsilon=-1.$

\subsection{Towers of Hermitian spaces}\label{subsec Hermitian spaces}
Any $-\epsilon$-Hermitian space $W_n$ has a Witt decomposition
\begin{equation}\label{eqn Witt decomp}
    W_n=W_{n_0}+W_{r,r},
\end{equation}
where $n=n_0+r$, $W_{n_0}$ is anisotropic and $W_{r,r}$ denotes the split $-\epsilon$-Hermitian space of dimension $2r$. The isomorphism class of $W_n$ uniquely determines the Witt index $r$ and the space $W_{n_0}$. Fix an anisotropic $-\epsilon$-Hermitian space $W_{n_0}.$ Then we associate a Witt tower to $W_{n_0}$ as follows:
\begin{equation}\label{eqn Witt tower}
    \mathcal{W}=\{W_{n_0}+W_{r,r} \ | \ r\geq 0\}.
\end{equation}
Similarly, for $V_m$, we have a Witt tower
\begin{equation}\label{eqn Witt tower 2}
    \mathcal{V}=\{V_{m_0}+V_{r,r} \ | \ r\geq 0\}.
\end{equation}
When $\epsilon=1,$ $W_n$ is symplectic and we have $\dim W_{n_0}=0$. This is not necessarily the case when $\epsilon=-1$ which we outline below.

For the rest of this subsection, we let $\epsilon=-1.$ Then $W_n$ is a quadratic space over $F$ and such spaces are classified by their dimension, discriminant, and Hasse invariant. Let $d,c\in F^\times$. Suppose first that $d\not\in (F^\times)^2$ and let
$$
W_{(d,c)}=F[X]/(X^2-d)
$$
be the 2-dimensional quadratic space over $F$ with bilinear form
$$
(\alpha,\beta)=c\cdot \mathrm{tr}(\alpha\overline{\beta}),
$$
where if $\beta=a+bX,$ then $\overline{\beta}=a-bX.$ We say $W_n$ is associated with $W_{(d,c)}$ if $W_{n_0}\cong W_{(d,c)}.$ Since $d\not\in (F^\times)^2$, there exists a unique orthogonal space $W_n^-$ such that $\mathrm{dim}(W_n)=2n=\mathrm{dim}(W_n^-)$ and  $\mathrm{disc}(W_n)=d=\mathrm{disc}(W_n^-)$, but $W_n^-\not\cong W_n.$ In this case, $W_n^-$ is associated to $W_{(d,c')}$ where $c'\not\in cN_{E/F}(E^\times)$. Here $E=F(\sqrt{d})$ and $N_{E/F}$ is the norm map. Note that this  indeed determines $W_n^-$ since $W_{(d,c)}\cong W_{(d',c')}$ if and only if $d\equiv d' \ \mathrm{mod} (F^\times)^2$ and $c\equiv c' \ \mathrm{mod} N_{E/F}(E^\times)$, where $E=F(\sqrt{d})=F(\sqrt{d'}).$

When $\mathrm{disc}(W_n)=d\in (F^\times)^2$, either $W_n\cong W_{n,n}$ or $W_n\cong D+ W_{n-2,n-2},$ where $D$ is the unique quaternion algebra over $F.$

Now fix $d\not\in (F^\times)^2$ and assume that $W_n$ is associated to $W_{(d,c)}.$ Let $W_n^-$ be the orthogonal space defined above. From $W_n$ and $W_n^-$, we obtain two Witt towers
\begin{align}
\label{eqn tower+} \mathcal{W}^+ &:=\{W_{(d,c)}+W_{r,r} \ | \ r\geq 0\},    \\
\label{eqn tower-}  \mathcal{W}^- &:=\{W_{(d,c')}+W_{r,r} \ | \ r\geq 0\}.
\end{align}
For $W \in \mathcal{W}^+$ or $\mathcal{W}^-$, the isometry group of $W$ is a (non-split) quasi-split even orthogonal group $G_n=\OO(W)$.

If $d\in (F^\times)^2$, then we set $W_n^+= W_{n,n}$ and we consider the  Witt towers
\begin{align*}
 \mathcal{W}^+ &:=\{W_{r,r} \ | \ r\geq 0\}, \\
\mathcal{W}^- &:=\{D+ W_{r,r} \ | \ r\geq 0\}.
\end{align*}
For any $W\in \mathcal{W}^+$, the isometry group of $W$ is a split even orthogonal group $G_n=\OO(W)$, while for $W\in \mathcal{W}^-$, the isometry group of $W$ is a non-quasi-split even orthogonal group $G_n=\OO(W)$.

In either case, $d\in (F^\times)^2$ or otherwise, we have that $\OO(W^+)$ is quasi-split for $W^+\in \mathcal{W}^+$ while for $W^-\in \mathcal{W}^-$ with $\dim W^-=\dim W^+,$ we have that $\OO(W^-)$ is the unique nontrivial pure inner form of $\OO(W^+)$.

\begin{remark}\label{rmk symplectic to orthogonal towers}
    When $\epsilon=1,$ $W_n$ is symplectic and $V_m$ is quadratic. In this situation, we write $\mathcal{V}^+$ and $\mathcal{V}^-$ for the corresponding towers obtained from $V_m$ as above. 
\end{remark}

\subsection{Langlands classification}\label{sec Langlands classification}

The Langlands classification is one of the fundamental tools in the representation theory of reductive algebraic groups over $p$-adic fields (\cite{Sil78, BW00, Kon03}). We recall the form we need for the classical groups considered in this paper over  $F$.

Let $n$ be a positive integer. We consider the general linear group $\GL_n(F)$ and $G_n=G(W_n)$. We let $\Pi(G_n)$ be the set of equivalence classes of irreducible admissible representations of $G_n$. Fix a Borel subgroup of $\GL_n(F)$ and consider a standard parabolic subgroup $P$ with Levi subgroup $M\cong \GL_{n_1}(F)\times\cdots\times \GL_{n_r}(F).$ For $\tau_i\in\Pi(\GL_{n_i}(F)),$ the normalized parabolic induction is denoted by 
\[
\tau_1\times\cdots\times\tau_r := \mathrm{Ind}_{P}^{\GL_n(F)}(\tau_1\otimes\cdots\otimes\tau_r).
\]
Let $\rho$ be an irreducible unitary supercuspidal representation of $\GL_n(F).$ 
A (Zelevinsky) segment $[x,y]_\rho$ is the set of (not necessarily unitary) supercuspidal representations of 
the form
$$
[x,y]_\rho=\{\rho\vert\cdot\rvert^x, \rho\vert\cdot\rvert^{x-1},\dots,\rho\vert\cdot\rvert^y\},
$$
where $x,y\in\mathbb{R}$ such that $x-y$ is a non-negative integer. Here $\vert\cdot\rvert$ denotes the composition of the normalized $p$-adic absolute value of $F$ with the determinant of $\GL_n(F)$. The unique irreducible subrepresentation of $\rho\vert\cdot\rvert^x\times\cdots\times\rho\vert\cdot\rvert^y$ is called the (shifted) Steinberg representation attached to $[x,y]_\rho$ and is denoted by $\Delta_\rho[x,y].$ 
For convenience, we set $\Delta_\rho[x,x+1]$, with $y=x+1,$ to be the trivial representation of $\GL_0(F).$ When the context is unambiguous, we refer to both $[x,y]_\rho$ and $\Delta_\rho[x,y]$ simply as segments.

The Langlands classification for $\GL_n(F)$ classifies the admissible dual $\Pi(\GL_n(F))$. Specifically, any irreducible admissible representation $\tau$ of $\GL_n(F)$ can be realized as the unique irreducible subrepresentation of a parabolic induction of the form
\[\Delta_{\rho_1}[x_1,y_1]\times\cdots\times\Delta_{\rho_r}[x_r,y_r],\]
where $\rho_i$ is an irreducible unitary supercuspidal representation of $\GL_{n_i}(F),$ $[x_i,y_i]_{\rho_i}$ is a segment, and $x_1+y_1\leq\cdots\leq x_r+y_r.$ In this situation, we write
$$
\tau=L(\Delta_{\rho_1}[x_1,y_1],\dots,\Delta_{\rho_r}[x_r,y_r]).
$$
Let $(x_{i,j})_{1\leq i\leq s, 1\leq j \leq t}$ be real numbers such that $x_{i,j}=x_{1,1}-i+j.$ A (shifted) generalized Speh representation is an irreducible representation of  the form
\begin{equation}\label{generalized Speh representation}
\begin{pmatrix}
x_{1,1} & \cdots & x_{1,t} \\
\vdots & \ddots & \vdots \\
x_{s,1} & \cdots & x_{s,t}
\end{pmatrix}_{\rho}:=L(\Delta_{\rho}[x_{1,1},x_{s,1}],\dots,\Delta_{\rho}[x_{1,t}, x_{s,t}]).
\end{equation}
For $a, b \in \mathbb{Z}_{>0}$, we set 
\begin{equation}\label{eqn shifted generalied speh}
u_\rho(a,b):=\begin{pmatrix}
\frac{a-b}{2} & \cdots & \frac{a+b}{2} \\
\vdots & \ddots & \vdots \\
\frac{-a-b}{2} & \cdots & \frac{-a+b}{2}
\end{pmatrix}_{\rho}
\end{equation}
which is a unitary representation.

Next, we discuss the Langlands classification of $G_n.$ Fix a minimal parabolic subgroup of $G_n$ and let $P$ be a standard parabolic subgroup of $G_n$ with Levi subgroup $M\cong \GL_{n_1}(F)\times\cdots\times \GL_{n_r}(F)\times G_{m},$ where $G_m$ is a group of the same type as $G_n$, but of possibly lower rank. Given smooth representations $\tau_i$ of $\GL_{n_i}(F)$ for $i=1,2,\dots,r$ and a smooth representation $\sigma$ of $G_{m}$, the normalized parabolic induction is denoted by
$$
\tau_1\times\cdots\times\tau_r\rtimes\sigma := \mathrm{Ind}_{P}^{G_n}(\tau_1\otimes\cdots\otimes\tau_r\otimes\sigma).
$$
The Langlands classification for $G_n$ states that every irreducible admissible representation $\pi$ of $G_n$ is a unique irreducible subrepresentation of 
\[\Delta_{\rho_1}[x_1,y_1]\times\cdots\times\Delta_{\rho_r}[x_r,y_r]\rtimes\pi_{temp},\]
where $\rho_i$ is an irreducible unitary supercuspidal representation of $\GL_{n_i}(F),$ $x_1+y_1\leq\cdots\leq x_r+y_r<0,$ and $\pi_{temp}$ is an irreducible tempered representation of $G_{m}.$ In this situation, we write
\begin{equation}\label{eqn Langlands classification}
\pi=L(\Delta_{\rho_1}[x_1,y_1],\dots,\Delta_{\rho_r}[x_r,y_r];\pi_{temp}),
\end{equation}
and say that the tuple $(\Delta_{\rho_1}[x_1,y_1],\dots,\Delta_{\rho_r}[x_r,y_r];\pi_{temp})$ is the Langlands data, or $L$-data, of $\pi.$  We refer to the multi-set $\{\Delta_{\rho_1}[x_1,y_1],\dots,\Delta_{\rho_r}[x_r,y_r]\}$ as the non-tempered portion of the $L$-data of $\pi.$ 

We also compare the $L$-data of various representations as follows. Given an irreducible admissible representation 
\[
\pi=L(\Delta_{\rho_1}[x_1,y_1],\dots,\Delta_{\rho_r}[x_r,y_r];\pi_{temp}),
\]
we can obtain a new representation by ``inserting" another segment $\Delta_{\rho_{r+1}}[x_{r+1},y_{r+1}]$ with $x_{r+1}+y_{r+1}<0$. That is, we consider the irreducible admissible representation
\[
L(\Delta_{\rho_1}[x_1,y_1],\dots,\Delta_{\rho_r}[x_r,y_r],\Delta_{\rho_{r+1}}[x_{r+1},y_{r+1}];\pi_{temp}),
\]
where the data is possibly re-indexed appropriately. We adopt the following convention.

\begin{conv}\label{conv L-data}
    We do not distinguish an irreducible admissible representation from its $L$-data. If $\pi'$ is obtained from $\pi$ by inserting a collection of segments $\{\Delta_{\rho_1}[x_1,y_1],\ldots, \Delta_{\rho_f}[x_f,y_f]\}$, then we write    
    \[\pi'=\pi+\{\Delta_{\rho_1}[x_1,y_1],\ldots, \Delta_{\rho_f}[x_f,y_f]\}.\]
\end{conv}

\subsection{Aubert-Zelevinsky Duality}
Let $\pi$ be an irreducible admissible representation of the connected group $G_n^{\circ}= G(W_n)^{\circ}.$ Aubert showed that there exists $\varepsilon\in\{\pm 1\}$ such that the virtual representation defined by
\begin{align*}
    \mathsf{AZ}(\pi):=\varepsilon\sum_P (-1)^{\mathrm{dim}(A_P)}[\mathrm{Ind}_{P}^{G_n}(\textrm{Jac}_P(\pi))]
\end{align*}
is an irreducible representation (\cite{Aub95}). The above sum is over all standard parabolic subgroups $P$ of $G_n^{\circ}$, $A_P$ denotes the maximal split torus in the center of the Levi subgroup of $P$, and $Jac_P$ denotes the Jacquet module along $P.$ We say that $\mathsf{AZ}(\pi)$ is the Aubert-Zelevinsky dual or Aubert-Zelevinsky involution of $\pi.$ 

We shall also consider the Aubert-Zelevinsky involution for the disconnected group $\OO(W_n)$. In a recent work of Cheng-Kaletha (\cite{CK26}), they generalized the definition of the Aubert-Zelevinsky involution to a disconnected linear algebraic group $\widetilde{G}$ such that $(\widetilde{G})^{\circ}$ is reductive and $\pi_0(\widetilde{G})$ is finite, which includes the case of $\OO(W_n)$. For an irreducible representation $\pi$ of $\OO(W_n)$, there exists $\varepsilon \in \{\pm 1\}$ such that
\begin{align*}
    \mathsf{AZ}(\pi):=\varepsilon\sum_{P} (-1)^{\mathrm{dim}(A_P)}[\mathrm{Ind}_{\widetilde{P}}^{\OO(W_n)}(\textrm{Jac}_{\widetilde{P}}(\pi))]
\end{align*}
is an irreducible representation (\cite[Lemma 3.1.11, Corollary 3.6.3]{CK26}). Here the sum is taken over all $\OO(W_n)$-conjugacy classes of standard parabolic subgroups $P$ of $\SO(W_n)$, where $\widetilde{P}:= N_{\OO(W_n)}(P)$. More precisely, let $c \in \OO(W_n)\setminus \SO(W_n)$ that normalizes the minimal parabolic subgroup we fixed. Then $\widetilde{P}=P$ if $c$ does not normalize $P$, and $\widetilde{P}=\langle P, c \rangle$ otherwise. If we restrict to twisted invariant distribution of $\SO(W_n)$ with respect to conjugation by $c$, then it is known (\cite[Corollary 3.2.4]{CK26}) that this matches the definition of M{\oe}glin \cite[\S 4]{Moe06b} and Xu (\cite[\S 6.2]{Xu17b}).

\subsection{Theta correspondence}\label{subsec theta correspondence}

Fix an additive character $\psi_F$ of $F.$ The pairs $(G_n,H_m)$ form a reductive dual pair inside a metaplectic group. Let $\omega_{W_n, V_m, \psi_F}$ denote the Weil representation of $G_n\times H_m.$ For $\pi\in\Pi(G_n)$, the maximal $\pi$-isotypic quotient of this Weil representation is of the form
$$
\pi\boxtimes \Theta_{W_n,V_m,\psi_F}(\pi),
$$
where $\Theta_{W_n,V_m,\psi_F}(\pi)$ is a smooth representation of $H_m$ called the big theta lift of $\pi.$ The (little) \emph{theta lift} of $\pi$, denoted $\theta_{W_n,V_m,\psi_F}(\pi)$, is the maximal semisimple quotient of $\Theta_{W_n,V_m,\psi_F}(\pi)$.

Suppose that $\epsilon=1$, i.e.,  $G_n$ is a symplectic group and $H_m$ is an even orthogonal group.
By Remark \ref{rmk symplectic to orthogonal towers}, we have two possible towers: either $V_m\in \mathcal{V}^+$ or $V_m\in \mathcal{V}^-$.
For a fixed $\pi\in\Pi(G_n)$, we let $\theta_{-\alpha}^{\pm}(\pi):=\theta_{W_n,V_m^\pm,\psi_F}(\pi)$ denote its theta lift with respect to $V_m^\pm\in\mathcal{V}^\pm$, where $\dim V_m^\pm=2m$ and $\alpha=2m-2n-1$ is an odd integer. When it is clear in context which tower is the target tower, we will often omit the superscript $\pm.$ 

On the other hand, suppose that $\epsilon=-1$. Then $G_n$ is an even orthogonal group. Let  $\pi\in\Pi(G_n)$. In this case, there is only one tower for the symplectic spaces $V_m$ and so we write $\theta_{-\alpha}(\pi):=\theta_{W_n,V_m,\psi_F}(\pi),$ where $\alpha=2m+1-2n.$ To make the notation uniform, we write $\theta_{-\alpha}^{\pm}(\pi)=\theta_{-\alpha}(\pi)$. Note that $\alpha=2m-2n-\epsilon$ in general.

One of the fundamental properties of the theta lift is Howe duality. It was  conjectured by Howe (\cite{How79}) and was proved by Waldspurger (\cite{Wal90}) when the residual characteristic of $F$ is not 2 and in full generality by Gan and Takeda (\cite{GT16}) and Gan and Sun (\cite{GS17}).

\begin{thm}[Howe Duality]\label{thm Howe duality} Let $\pi,\pi'\in\Pi(G_n).$
\begin{enumerate}
    \item If $\theta_{-\alpha}(\pi)\neq 0$, then $\theta_{-\alpha}(\pi)$ is irreducible.
    \item If $\pi\not\cong\pi'$   and $\theta_{-\alpha}(\pi)$ and $\theta_{-\alpha}(\pi')$ are both nonzero, then $\theta_{-\alpha}(\pi)\not\cong\theta_{-\alpha}(\pi').$
\end{enumerate}
\end{thm}

Occasionally, we need to specify the ``going-down'' tower for the theta lift when $G_n=\OO(W_n)$. We recall this notion as follows. Consider the representations $\pi,\pi\otimes\det\in\Pi(G_n)$. For $V_m\in\mathcal{V}$, we have two ``towers'' corresponding to the representations
\[
\theta_{W_n,V_m,\psi_F}(\pi) \ \ \mathrm{and} \ \ \theta_{W_n,V_m,\psi_F}(\pi\otimes \det).
\]
We let $m(\pi)$ be the minimal integer $m=\frac{\dim V}{2}$ such that $V\in\mathcal{V}$ and $\theta_{W_n,V,\psi_F}(\pi)\neq 0.$ We define
\begin{align*}
        m^{up}(\pi)&:=\max\{m(\pi),m(\pi\otimes\det)\}, \\
    m^{down}(\pi)&:=\min\{m(\pi),m(\pi\otimes \det)\}.
\end{align*}
In general, when $\pi\in\Pi(\OO(W_n))$, we have the conservation relation between $m^{up}(\pi)$ and $m^{down}(\pi).$

\begin{thm}[Conservation relation, \cite{SZ15}]\label{thm conservation relation}
    Let $G_n=\OO(W_n)$ and $\pi\in\Pi(G_n).$ Then
    $$
    2m^{up}(\pi)+2m^{down}(\pi)=4n.
    $$
\end{thm}

As a consequence, we have that $m^{up}(\pi)\geq n\geq m^{down}(\pi).$ Also, if one inequality is strict, then both inequalities are strict. We say that $\pi'\in\{\pi,\pi\otimes\det\}$ corresponds to the going-down tower of $\pi$ if $m^{down}(\pi)=m(\pi').$ Note that if $m^{up}(\pi)=m^{down}(\pi)$, then the going-down tower is not well-defined. In this case, we assign the going-down tower arbitrarily.

For later use, we need a case of Howe duality which holds for pure inner forms. We record this here. 
\begin{lemma}\label{lem Howe duality pure}
    Suppose that $\pi^\pm\in\Pi(\OO(W_n^\pm)).$ Then for any $\alpha\geq 1$, we have $\theta_{-\alpha}(\pi^+)\not\cong \theta_{-\alpha}(\pi^-)$ provided that at least one of them is nonzero.
\end{lemma}

\begin{proof}
Suppose for contradiction that $\theta_{-\alpha}(\pi^+)=: \sigma \cong \theta_{-\alpha}(\pi^-)  \neq 0$. Then we see that the theta lift of $\sigma \in \Pi( \Sp(V_{n+ \half{\alpha-1}}))$ to the two smaller groups $\OO(W_n^+)$ and $\OO(W_n^-)$ are both nonzero. The conservation relation (for the local theta lift from the symplectic group to the even orthogonal group, \cite[Theorem 1.10]{SZ15}) in this case gives that 
\[ 2n \geq \half{4(n+\half{\alpha-1})+ 4}, \]
which is a contradiction.
\end{proof}

\section{Local Arthur packets and local Langlands correspondence} \label{sec: lap}

In this section, we recall the theory of local Arthur packets for symplectic and even orthogonal groups. In the tempered case, this gives the local Langlands correspondence.

\subsection{Local Arthur packets}\label{sec: lap for symplectic}\label{sec Local Arthur packets for orthogonal groups}
In this subsection, we recall the theory of local Arthur packets for $G_n=\Sp(W_n)=\Sp_{2n}(F)$ and for $G_n=\OO(W_n)$ simultaneously. We let $G_n^\circ$ denote the identity component of $G_n$ and we set
\[
N:=\begin{cases}
2n+1 & \text{if } G_n=\Sp(W_n),\\
2n & \text{if } G_n=\OO(W_n).
\end{cases}
\]

If $G_n=\Sp(W_n)$, then $G_n=G_n^\circ$ is split, and, to make the notation uniform, we set ${}^L{G}_n^\circ:=\SO_{2n+1}(\BC)$ and $\widehat{G}_n(\BC):=\SO_{2n+1}(\BC)$.

If $G_n=\OO(W_n)$, then $G_n$ is not connected and its theory of local Arthur packets requires more finesse. In this case $G_n^\circ=\SO(W_n)$ is a special even orthogonal group, so that $\widehat{G_n^\circ}=\SO_{2n}(\BC)$, and we may replace ${}^L G_n^{\circ}$ by $\SO_{2n}(\BC)$ if $G(W_n^+)$ is split, and by $ \OO_{2n}(\BC)$ otherwise (see \cite[\S 7]{GGP12}). In either case, we set $\widehat{G}_n(\BC):=\OO_{2n}(\BC)$.
In both cases we obtain in this way a standard embedding of ${}^LG_n^\circ$ into $\GL_{N}(\BC)$.

A local Arthur parameter of $G_n^\circ$ is a homomorphism which decomposes as a direct sum of irreducible representations after composing with the standard embedding ${}^L{G}_n^\circ\hookrightarrow\GL_{N}(\BC)$
$$\psi: W_F \times \SL_2(\mathbb{C}) \times \SL_2(\mathbb{C}) \rightarrow {}^L{G}_n^\circ$$
\begin{equation}\label{eq decomp psi +}
  \psi = \bigoplus_{i=1}^r \phi_i\lvert\cdot\rvert^{x_i} \otimes S_{a_i} \otimes S_{b_i},
\end{equation}
satisfying the following conditions:
\begin{enumerate}
    \item [(1)]$\phi_i(W_F)$ is bounded and consists of semisimple elements, and $\dim(\phi_i)=d_i$;
    \item [(2)] $x_i \in \R$ and $|x_i|<\frac{1}{2}$;
    \item [(3)]the restrictions of $\psi$ to the two copies of $\SL_2(\mathbb{C})$ are analytic, $S_k$ is the $k$-dimensional irreducible representation of $\SL_2(\mathbb{C})$, and
    $$\sum_{i=1}^r d_ia_ib_i = N.$$
\end{enumerate}
The first $\SL_2(\BC)$ is known as the Deligne-$\SL_2(\BC)$ while the second $\SL_2(\BC)$ is known as the Arthur-$\SL_2(\BC)$.

Two local Arthur parameters of $G_n^\circ$ are equivalent if they are conjugate under $\widehat{G_n^\circ}(\BC)$. We do not distinguish between $\psi$ and its equivalence class in the rest of the paper. We let $\Psi^{+}(G_n^\circ)$ denote the set of equivalence classes of local Arthur parameters, and $\Psi(G_n^\circ)$ be the subset of $\Psi^+(G_n^\circ)$ consisting of local Arthur parameters $\psi$ whose restriction to $W_F$ is bounded. In other words, $\psi$ is in $\Psi(G_n^\circ)$ if and only if $x_i=0$ for $i=1,\dots, r$ in the decomposition \eqref{eq decomp psi +}. If $G_n=\Sp(W_n)$, then $G_n=G_n^\circ$ and we simply put $\Psi^+(G_n):=\Psi^+(G_n^\circ)$ and $\Psi(G_n):=\Psi(G_n^\circ)$. Suppose that $G_n=\OO(W_n)$. Fix $c\in G_n \setminus G_n^\circ$ and let $\sigma_0$ be the outer automorphism on $G_n^\circ$ given by conjugation by $c.$ We also set $\Sigma_0$ to be the group generated by $\sigma_0.$ We have $G_n=G_n^\circ\rtimes \Sigma_0.$ Through the dual automorphism $\hat{\sigma}_0$ on $\widehat{G}_n^\circ$, we have an action of $\Sigma_0$ on $\Psi^{+}(G_n^\circ).$
We let $\Psi^+(G_n)$ and $\Psi(G_n)$ be the sets of $\Sigma_0$-orbits of $\Psi^+(G_n^\circ)$ and $\Psi(G_n^\circ),$ respectively. We also let $\Pi^{\Sigma_0}(G_n^\circ)$ denote the set of $\Sigma_0$-orbits of $\Pi(G_n^\circ)$.

By the local Langlands correspondence for $\GL_{d_i}(F)$, the bounded irreducible representation $\phi_i$ of $W_F$ can be identified with an irreducible unitary supercuspidal representation $\rho_i$ of $\GL_{d_i}(F)$ (\cite{HT01, Hen00, Sch13}). Consequently, for $\psi\in\Psi^+(G_n)$ we may write
\begin{equation}\label{A-param decomp}
  \psi = \bigoplus_{\rho}\left(\bigoplus_{i\in I_\rho} \rho\lvert\cdot\rvert^{x_i} \otimes S_{a_i} \otimes S_{b_i}\right),
\end{equation}
where the first sum runs over
irreducible unitary supercuspidal representations $\rho$ of $\GL_d(F)$, $d \in \mathbb{Z}_{\geq 1}$.

Given a local Arthur parameter $\psi$ as in \eqref{A-param decomp}, we say that $\psi$ is of \emph{good parity} if $\psi \in \Psi(G_n)$, i.e., $x_i=0$ for all $i$, and every summand $\rho \otimes S_{a_i} \otimes S_{b_i}$ is self-dual and orthogonal.
We let $\Psi_{gp}(G_n)$ denote the subset of $\Psi(G_n)$ consisting of local Arthur parameters of good parity.

Let $\psi \in \Psi^{+}(G_n).$ From the decomposition \eqref{A-param decomp}, we define a subrepresentation $\psi_{nu,>0}$ of $\psi$ by
\begin{equation}\label{eqn psi_nu,>0}
    \psi_{nu,>0}:= \bigoplus_{\rho}\left(\bigoplus_{\substack{i\in I_\rho,\\ x_i>0}} \rho\lvert\cdot\rvert^{x_i} \otimes S_{a_i} \otimes S_{b_i}\right).
\end{equation}
Since the image of $\psi$ is contained in  $\widehat{G}_n(\BC)$, $\psi$ is self-dual, and hence $\psi$ also contains $(\psi_{nu,>0})^{\vee}$. We define $\psi_{u} \in \Psi(G_m)$ for some $m\leq n$ by
\begin{align}\label{eq def of psi_u}
    \psi= \psi_{nu,>0} \oplus \psi_u \oplus (\psi_{nu,>0})^{\vee}.
\end{align}
Equivalently,
\[ \psi_{u}:= \bigoplus_{\rho}\left(\bigoplus_{\substack{i\in I_\rho,\\ x_i=0}} \rho \otimes S_{a_i} \otimes S_{b_i}\right). \]

We further decompose $\psi_{u}$. Suppose $\rho \otimes S_a \otimes S_b$ is an irreducible summand of $\psi_u$ that is either not self-dual, or self-dual but symplectic. Then the self-duality of $\psi$ implies that $\psi$ must contain the other summand $(\rho\otimes S_a\otimes S_b)^{\vee}=\rho^{\vee} \otimes S_a \otimes S_b$. Therefore, we may choose a subrepresentation $\psi_{np}$ of $\psi_u$ such that
\begin{align}\label{eq decomp of psi_u}
     \psi_{u}= \psi_{np} \oplus \psi_{gp} \oplus \psi_{np}^{\vee},
\end{align}
where $\psi_{gp} $ is of good parity, and any irreducible summand of $\psi_{np}$ is either not self-dual or self-dual but symplectic. Altogether, we obtain a decomposition
\begin{equation}\label{eqn psi = nu+np+gp}
    \psi=\psi_{nu,>0}\oplus\psi_{np}\oplus\psi_{gp}\oplus\psi_{np}^\vee\oplus\psi_{nu,>0}^\vee.
\end{equation}
Let $I_{\rho,np}$ be the subset of $I_{\rho}$ that corresponds to $\psi_{np}$. For use in Theorem \ref{thm red from nu to gp symp} below, we consider the following irreducible representations of a general linear group
\begin{align*}
\tau_{\psi_{nu,>0}}&=\bigtimes_\rho\bigtimes_{i\in I_\rho, x_i>0}u_\rho(a_i,b_i)\lvert\cdot\rvert^{x_i}, \\
\tau_{\psi_{np}}&=\bigtimes_{\rho}\bigtimes_{i\in I_{\rho,np}}u_\rho(a_i,b_i).
\end{align*}
See \eqref{eqn shifted generalied speh} for the definition of $u_\rho(a,b)$.

We now turn to the packets themselves. For each $\psi \in \Psi(G_n)$, the endoscopic character identities define a virtual representation $\pi(\psi, \varepsilon)$ of $G_n$ for each irreducible representation $\varepsilon$ of a certain component group associated to $\psi$. Arthur conjectured (\cite[Conjecture 6.1]{Art89}) that each of $\pi(\psi, \varepsilon)$ is indeed a representation, and defined the local Arthur packet $\Pi_\psi$ as the multi-set of irreducible representations given by
\[ \bigoplus_{\pi \in \Pi_{\psi}}\pi = \bigoplus_{\varepsilon} \pi(\psi, \varepsilon).\]
When $G_n=\Sp(W_n)$ or $G_n= \OO(W_n^+)$, the above conjecture is proved in \cite{Art13}. The local Arthur packet $\Pi_\psi$ is a finite multi-set consisting of irreducible unitary representations of $G_n$. In the case $G_n=\OO(W_n^+)$, Arthur defines local Arthur packets $\Pi_\psi^{\Sigma_0}(G_n^\circ) \subset \Pi^{\Sigma_0}(G_n^{\circ})$ by $\pi^0 \in \Pi_{\psi}^{\Sigma_0}(G_n^{\circ})$ if and only if there exists a $\pi \in \Pi_{\psi}(G_n)$ such that $\pi^0$ is a subquotient of  $\pi|_{G_n^{\circ}}$. When there is no ambiguity, we write $\Pi_\psi=\Pi_\psi(G_n).$

On the other hand, in a sequence of works of M{\oe}glin (\cite{Moe06b, Moe09a, Moe11}), she explicitly constructed representations $\pi_{\Moe}(\psi,\varepsilon)$, and showed that these representations satisfy the same endoscopic character identities that characterize $\pi(\psi,\varepsilon)$, assuming that the same holds for all discrete $\psi$ (\cite[Remarque 2.4]{Moe09a}). Moreover, she showed that $\bigoplus_{\varepsilon} \pi(\psi, \varepsilon)$ is multiplicity-free; in particular, $\Pi_\psi$ (and hence $\Pi_\psi^{\Sigma_0}(G_n^\circ)$ in the orthogonal case) is multiplicity free (\cite{Moe11}). It is crucial that this construction and the reduction include the pure inner form $\OO(W_n^-)$. The endoscopic character identities for discrete packets of $\Sp(W_n)$ and $\OO(W_n^+)$ are verified in \cite{Art13}, and for $\OO(W_n^-)$ they are verified in \cite{MR18}. As a consequence, we have $\pi_{\Moe}(\psi, \varepsilon)= \pi(\psi, \varepsilon)$. Throughout this paper, we take the following as the definition of local Arthur packets for symplectic groups and even orthogonal groups
\[ \bigoplus_{\pi \in \Pi_{\psi}} \pi= \bigoplus_{\varepsilon} \pi_{\Moe}(\psi, \varepsilon).\]
We say that an irreducible representation $\pi$ of $G_n$ is \emph{of Arthur type} if $\pi\in\Pi_\psi$ for some local Arthur parameter $\psi \in \Psi(G_n)$.

For $\psi \in \Psi^+(G_n)$, Arthur defined (\cite[(1.5.1)]{Art13})
\begin{align}\label{eq def packet +}
    \Pi_{\psi}:= \{ \tau_{\psi_{nu,>0}} \rtimes \pi_u \ | \ \pi_{u} \in \Pi_{\psi_u}   \},
\end{align}
where $\tau_{\psi_{nu,>0}}$ is defined above. 
Since $|x_i|<\frac{1}{2}$ in the decomposition \eqref{A-param decomp}, the parabolic induction in \eqref{eq def packet +} is always irreducible (\cite[Proposition 5.1]{Moe11b}; see also Theorem \ref{thm red from nu to gp symp} below).
Finally, M{\oe}glin constructed the local Arthur packet $\Pi_{\psi}$ from $\Pi_{\psi_{gp}}$, which we record below.

\begin{thm}[{\cite[\S4.2]{Moe11}, \cite[Proposition 5.1]{Moe11b}}]\label{thm red from nu to gp symp}
Let $G_n=\Sp(W_n)$ or $\OO(W_n)$, and let $\psi\in\Psi^+(G_n)$ with decomposition as in \eqref{eqn psi = nu+np+gp}. Then, for any $\pi_{gp}\in\Pi_{\psi_{gp}},$ the induction $\tau_{\psi_{nu,>0}}\times\tau_{\psi_{np}}\rtimes\pi_{gp}$ is irreducible. As a consequence,
\begin{equation}\label{non-unitary A-packet}
    \Pi_\psi=\{\tau_{\psi_{nu,>0}}\times\tau_{\psi_{np}}\rtimes\pi_{gp} \ | \ \pi_{gp}\in\Pi_{\psi_{gp}}\}.
\end{equation}
\end{thm}

\subsection{Good parity and the local Langlands correspondence}

In this section, we let $G_n=\Sp(W_n)$ or $\OO(W_n)$.
An $L$-parameter of $G_n$ is a homomorphism
\[
\phi:W_F\times\SL_2(\BC)\rightarrow{}^LG_n^\circ,
\]
satisfying the following conditions:
\begin{enumerate}
    \item $\phi(W_F)$ consists of semisimple elements;
    \item the restriction of $\phi$ to $\SL_2(\BC)$ is analytic; and
    \item  if $G(W_n^+)$ is not split, then the composition of $\phi$ with the projection ${}^L G_n^\circ\rightarrow W_{E/F}\cong\pi_0({}^LG_n^\circ)$ is compatible with $W_F\times\SL_2(\BC)\rightarrow W_F.$
\end{enumerate}
By composing with the standard embedding ${}^LG_n^\circ\hookrightarrow\GL_{N}(\mathbb{C}),$ where $N \in \{2n, 2n+1\}$,
we may identify $\phi$ with a direct sum of irreducible representations,
\begin{equation}\label{eqn L-par decomp}
    \phi=\bigoplus_{i=1}^r \phi_i\otimes S_{a_i},
\end{equation}
where $\phi_i$ is an irreducible representation of $W_F.$ We say that two $L$-parameters $\phi,\phi'$ are \emph{equivalent} if $\phi$ and $\phi'$ are conjugate under $\widehat{G}_n(\mathbb{C})$ (not necessarily under $\widehat{G}_n^\circ(\BC)$).
This is equivalent to $\phi$ and $\phi'$ are isomorphic as representations of $W_F \times \SL_2(\BC)$ (see \cite[Theorem 8.1]{GGP12}). Note that when $G_n=\OO(W_n)$ this equivalence is for the full even orthogonal groups rather than special even orthogonal groups (see \cite[Remark below Corollary 8.2]{GGP12}). We let $\Phi(G_n)$  denote the set of equivalence classes of $L$-parameters.

\begin{defn}\label{defn Lpar good parity}
    Let $\phi\in\Phi(G_n)$ be decomposed as in \eqref{eqn L-par decomp}. Then,
there exists $e_i\in\mathbb{R}$ such that $\phi_i\lvert\cdot\rvert^{-e_i}$ is a bounded irreducible representation of $W_F.$ We say that $\phi_i$ is \emph{of good parity} if $e_i\in\frac{1}{2}\mathbb{Z}$ and $\phi_i\lvert\cdot\rvert^{-e_i}\otimes S_{a_i+2|e_i|}$ is self-dual and orthogonal. 

Let $I_{\good}=\{i\in\{1,\dots,r\}\ | \ \phi_i \ \text{is of good parity}\}.$ We define $\phi_\good=\bigoplus_{i\in I_\good}\phi_i \otimes S_{a_i}.$ We say that $\phi$ is \emph{of good parity} if $\phi=\phi_\good.$

Let $\psi\in\Psi(G_n).$ We associate an $L$-parameter $\phi_\psi\in\Phi(G_n)$ which is defined by $\phi_\psi(w,x)=\psi(w,x,d_w)$ where $d_w=\mathrm{diag}(|w|^{\frac{1}{2}},|w|^{\frac{-1}{2}})$.
We have that $\psi$ is \emph{of good parity} if and only if its associated $L$-parameter $\phi_\psi$ is of good parity.
\end{defn}

When $G_n=\OO(W_n)$, we let $\Pi^\pure(G_n)=\Pi(\OO(W_n^+))\sqcup \Pi(\OO(W_n^-))$, where $W_n^\pm\in\mathcal{W}^\pm$ (see \S\ref{subsec Hermitian spaces}).
Next we state the pure local Langlands correspondence for $G_n$, due to Arthur in the quasi-split case (\cite{Art13}) and extended to pure inner forms by M{\oe}glin and Renard (\cite{MR18}). When $G_n=\Sp(W_n),$ we let $\Pi^\pure(G_n)=\Pi(G_n).$

\begin{thm}[Local Langlands correspondence]\label{thm pure LLC}
    There is a finite-to-one map called the local Langlands correspondence, 
    \[\LLC:\Pi^\pure(G_n)\rightarrow\Phi(G_n)\] 
    which satisfies the following properties (among others).
    \begin{enumerate}
        \item For $\phi\in\Phi(G_n)$, there is a bijection between $\LLC^{-1}(\phi)$ and $\mathrm{Irr}(A_\phi)$, where $\mathrm{Irr}(A_\phi)$ denotes the set of irreducible representations of $A_\phi=\pi_0(\mathrm{Cent}_{\widehat{G}_n}(\mathrm{im}(\phi))).$ Here, we have implicitly fixed a choice of Whittaker datum for $G(W_n^+)$.
        \item The following are equivalent.
    \begin{enumerate}
        \item $\LLC^{-1}(\phi)$ contains a tempered representation.
        \item $\LLC^{-1}(\phi)$ consists only of tempered representations.
        \item $\phi\in \Phi(G_n)$ is tempered, i.e., $\phi(W_F)$ is bounded.
    \end{enumerate}
    \item The local Langlands correspondence is compatible with the Langlands classification in the following sense: Suppose that 
    \[ \pi= L(\Delta_{\rho_1}[x_1,y_1],\dots,\Delta_{\rho_r}[x_r,y_r]; \pi_{temp}).\]
    Then 
    \begin{align*}
        \LLC(\pi)&= \rho_1\lvert \cdot \rvert^{\half{x_1+y_1}} \otimes S_{x_1-y_1+1} \oplus \cdots \oplus \rho_r\lvert \cdot \rvert^{\half{x_r+y_r}} \otimes S_{x_r-y_r+1}  \oplus  \LLC(\pi_{temp}) \\
       & \oplus \rho_1^{\vee}\lvert \cdot \rvert^{\half{-x_1-y_1}} \otimes S_{x_1-y_1+1} \oplus \cdots \oplus \rho_r^{\vee}\lvert \cdot \rvert^{\half{-x_r-y_r}} \otimes S_{x_r-y_r+1}.
    \end{align*}
    \end{enumerate}
\end{thm}

The pre-image $\Pi_\phi^\pure:=\LLC^{-1}(\phi)$ is called the \emph{pure $L$-packet} associated to $\phi.$ We let $\Pi_\phi(G_n)=\Pi^\pure_\phi\cap\Pi(G_n).$ For $\pi\in\Pi^\pure(G_n)$, we call $\phi_\pi:=\LLC(\pi)$ the \emph{$L$-parameter} associated to $\pi.$

We recall the combinatorial description of the component group $A_\phi$ from \cite[Equation (1.4.8)]{Art13} (with slight modification for pure inner forms).  First, we have that $A_\phi\cong A_{\phi_{\good}}.$ Write $\phi_\good=\bigoplus_{i\in I_\good} \phi_i\otimes S_{a_i}$ as in Definition \ref{defn Lpar good parity}. Then $A_{\phi_{\good}}\cong A_{\phi_{\good,\temp}}$, where \[\phi_{\good,\temp}=\bigoplus_{i\in I_{\good,\temp}} \phi_i\otimes S_{a_i}=\bigoplus_{i=1}^s m_i(\phi_i'\otimes S_{a_i'}),\]
here $I_{\good,\temp}=\{i\in I_\good \ | \ \phi_i \ \text{is bounded}\},$ and $m_i$ denotes the multiplicity of $\phi_i'\otimes S_{a_i'}.$ {If we consider the centralizer in the full even orthogonal group, then we have
\[
    A_\phi^\OO:=\pi_0(\mathrm{Cent}_{\OO_N(\BC)}(\mathrm{im}(\phi)))\cong\bigtimes_{i=1}^s \{\pm 1\},
    \]
    where each $\{\pm 1\}$ corresponds to the summands $m_i(\phi_i'\otimes S_{a_i'})$. When $G_n=\OO(W_n)$, we have $A_\phi^\OO=A_\phi$ and consequently, we may identify $\textrm{Irr}(A_{\phi})$ with the set of functions $\varepsilon:\{\phi_i'\otimes S_{a_i'}\}_{i=1,\ldots, s} \to \{\pm 1\}$. When $G_n=\Sp(W_n)$, we must take the centralizer in $\SO_{2n+1}(\BC)$, not $\OO_{2n+1}(\BC)$ (see the remark below). Thus, in this case, we obtain that
    \[
    A_\phi\cong\bigtimes_{i=1}^{s-1}\{\pm 1\},
    \]
    and consequently, we may identify $\textrm{Irr}(A_{\phi})$ with the set of functions $\varepsilon:\{\phi_i'\otimes S_{a_i'}\}_{i=1,\ldots, s} \to \{\pm 1\}$ subject to the condition that
    \[
    \prod_{i=1}^s\varepsilon(\phi_i'\otimes S_{a_i'})^{m_i}=1.
    \]
    Here, we have implicitly identified $A_\phi$ as a subgroup of $A_\phi^\OO$.}
\begin{remark}
    Note that unlike \cite[Equation (1.4.9)]{Art13}, our centralizer in the even orthogonal case is taken in the full even orthogonal group, not the special even orthogonal group (which is the reason for the kernel in \cite[Equation (1.4.8)]{Art13}). Consequently, we do not require the determinant 1 condition imposed in \cite[Equation (1.4.9)]{Art13}.
\end{remark}

Let $G_n=\OO(W_n).$ The Kottwitz isomorphism (\cite[Proposition 6.4]{Kot84}) associates to $G_n^\circ$ a homomorphism $\chi_{G_n}\in\mathrm{Hom}(\pi_0(Z(\widehat{G_n^\circ})^{\Gamma_F}),\mathbb{C}^\times)$ where $\Gamma_F$ denotes the absolute Galois group of $F.$ We have $\pi_0(Z(\widehat{G_n^\circ})^{\Gamma_F})\cong\{\pm1\}$ and so $\chi_{G_n}$ is the trivial character if $W_n\in\mathcal{W}^+$ and the sign character otherwise. {We have a map $\pi_0(Z(\widehat{G_n^\circ})^{\Gamma_F})\rightarrow A_\phi$ given by $\pm1\mapsto((\pm1)^{m_1},\dots,(\pm1)^{m_s}).$ We define 
\[ \mathrm{Irr}(A_{\phi, G_n})= \{\varepsilon \in \textrm{Irr}(A_{\phi})\ | \ \prod_{i=1}^s\varepsilon(\phi_i' \otimes S_{a_i'})^{m_i}= \chi_{G_n} (-1) \}. \]
}
 The local Langlands correspondence gives a further bijection which follows from (\cite{Art13,MR18}).
When $G_n$ is symplectic, we set $\mathrm{Irr}(A_{\phi,G_n}):=\mathrm{Irr}(A_{\phi})$.
\begin{thm}\label{thm enh LLC form}
    Let $\phi\in\Phi(G_n)$. Then there is a bijection 
    \begin{align*}
        \Pi_\phi(G_n) \xrightarrow{\sim}\mathrm{Irr}(A_{\phi, G_n}).
    \end{align*}
    For $\varepsilon \in \textrm{Irr}(A_{\phi,G_n})$, we let $\pi(\phi,\varepsilon)$ be the corresponding representation in $\Pi_{\phi}(G_n)$.
\end{thm}

Finally, we recall the definition of good parity part for representations.

 \begin{defn}\label{defn good parity rep}
 For $\pi=\pi(\phi,\varepsilon)\in \Pi(G_n)$, we define the \emph{good parity part} of $\pi$ by
 \[ \pi_{\good}:= \pi(\phi_{\good}, \varepsilon),\]
 which is a representation of $G_m$ for some $m \leq n$. Here, we regard $\varepsilon \in \textrm{Irr}(A_{\phi_{\good}})$ via the isomorphism $A_\phi\cong A_{\phi_\good}$.

 We say $\pi$ is of good parity if $\pi=\pi_{\good}$, and we let $\Pi_{gp}(G_n)$ denote the subset of $\Pi(G_n)$ consisting of representations of good parity. 
\end{defn}

\subsection{\texorpdfstring{M{\oe}glin's parameterization of the local Arthur packets $\Pi_{\psi}$}{}}

 Let $G_n=G(W_n)$ be a symplectic or even orthogonal group. In this section, we review M{\oe}glin's parameterization of the local Arthur packets $\Pi_{\psi}$ for $\psi\in\Psi^+(G_n)$ (\cite{Moe06a, Moe06b, Moe09a, Moe10, Moe11}). In view of Theorem \ref{thm red from nu to gp symp}, it is enough to describe the case $\psi\in\Psi_{gp}(G_n).$ To this end, we recall the following combinatorial data introduced by M{\oe}glin. 

 \begin{defn}\label{def Moeglin parameter}\
 \begin{enumerate} 
     \item 
     A \emph{M{\oe}glin parameter} for $G(W_n)$ is a quintuple $\MM=(\psi, \underline{\zeta}, >, \underline{l}, \underline{\eta})$ described as follows:
     \begin{enumerate}
         \item $\psi$ is a local Arthur parameter of $G(W_n)$ of good parity with the following decomposition
          \begin{equation}\label{eqn gp decomp}
     \psi:= \bigoplus_{\rho}\left(\bigoplus_{i\in I_\rho} \rho \otimes S_{a_i} \otimes S_{b_i}\right).
 \end{equation}
 We set $I:= \sqcup I_{\rho}$ and $d_i=\min(a_i,b_i)$ for $i\in I$.
 \item $\underline{\zeta}= (\zeta_i)_{i \in I}$ is a collection of signs, i.e., $\zeta_i \in \{\pm 1\}$, such that $\zeta_i (a_i -b_i) \geq 0$.
 \item $>= (>_{\rho})_{\rho}$ is a collection of total orders on $I_{\rho}$ such that for any $i \neq j \in I_{\rho}$, 
 $$
a_i+b_i>a_j+b_j, \, |a_i-b_i|>|a_j-b_j|,\, \mathrm{and}\, \zeta_i=\zeta_j \implies i>_{\rho}j.
$$
We call a collection of total orders with the above property an \emph{admissible order}.
\item $\underline{l}=(l_i)_{i \in I}$ is a collection of non-negative integers such that $0 \leq l_i \leq \half{d_i}$.
\item $\underline{\eta}=(\eta_i)_{i \in I}$ is a collection of signs, i.e., $\eta_i \in \{\pm 1\}$, such that
the following {hold: 
\begin{itemize}
    \item if $l_i=\frac{d_i}{2},$ then $\eta_i=1$, and
    \item the sign condition holds: 
\begin{equation}\label{Moeglin sign}
    \prod_{\rho} \prod_{i \in I_\rho} (-1)^{\lfloor\frac{d_i}{2}\rfloor+l_i} \eta_i^{d_i} = \epsilon_{G(W_n)}.
\end{equation}
\end{itemize}
}
Here, $\epsilon_{G(W_n)}=1$ if $G(W_n)$ is symplectic and is given by $\epsilon_{G(W_n)}=\chi_{G_n}(-1)$, where $\chi_{G_n}$ is the character obtained from the Kottwitz isomorphism, otherwise.
     \end{enumerate}
    We often write $ \mathcal{M}=(\psi_{>},\underline{\zeta}, \underline{l},\underline{\eta})$ instead of $(\psi, \underline{\zeta}, >, \underline{l}, \underline{\eta}).$
We let $\Moe(G(W_n))$ denote the set of M{\oe}glin parameters for $G(W_n).$ 

\item Let $\MM=(\psi,\underline{\zeta},>, \underline{l},\underline{\eta})$ with a decomposition of $\psi$ as in \eqref{eqn gp decomp}. For each $\rho$, we let
\[ \MM_{\rho}:= \left( \bigoplus_{i \in I_{\rho}} \rho \otimes S_{a_i} \otimes S_{b_i}, (\zeta_{i})_{i \in I_{\rho}}, >_{\rho}, (l_{i})_{i \in I_{\rho}},(\eta_{i})_{i \in I_{\rho}}  \right), \ \ \MM^{\rho}:= \bigcup_{\rho' \not\cong \rho} \MM_{\rho'}\]
so that $\MM= \MM_{\rho} \cup \MM^{\rho}$.
  \item Let $\MM=(\psi_{>}, \underline{\zeta},\underline{l},\underline{\eta})$. In the notation of Part (1), suppose that there is a decomposition of the index set $I= I_1 \sqcup I_2$ such that for any $i_1 \in I_1 \cap I_{\rho}$ and $i_2 \in I_2\cap I_{\rho}$, we have $i_1<i_2$. Then for $j= 1,2$, we define $\MM_j=((\psi_j)_{\gg}, \underline{\zeta_j},\underline{l_j},\underline{\eta_j})$, where
     \begin{itemize}
         \item $ \psi_j= \bigoplus_{\rho} \bigoplus_{i \in I_{\rho}\cap I_j} \rho \otimes S_{a_i} \otimes S_{b_i}$, and
         \item the total order $\gg$ on $I_{\rho}\cap I_j$ is given by the restriction of $>$, and
         \item the other data are given by $\underline{\zeta_j}= (\zeta_i)_{i \in I_j}, \underline{l_j}=(l_i)_{i \in I_j}, \underline{\eta_j}=(\eta_{i})_{i \in I_j}$.
     \end{itemize}
     Then we write $\MM=\MM_1+ \MM_2$. Note that $\MM_1 + \MM_2 \neq \MM_2+\MM_1$.

\item  We say the admissible order $>$ of $\mathcal{M}=(\psi_{>},\underline{\zeta}, \underline{l},\underline{\eta})$ satisfies condition $(P')$ if the following holds:
\begin{itemize}
        \item if $\zeta_i=-1$ and $\zeta_j=1$, then $i<j$;
        \item if $\zeta_i=\zeta_j$ and $|\frac{a_i-b_i}{2}|<|\frac{a_j-b_j}{2}|$ , then $i<j.$
    \end{itemize} 
    We let $\Moe^{(P')}(G(W_n))$ denote the set of M{\oe}glin parameters whose order satisfies condition $(P')$.
     \end{enumerate}
 \end{defn}

\begin{remark}\label{rmk Moeglin parameter}\ 
\begin{enumerate}
    \item   We often regard $\underline{\zeta}, \underline{l}$, and $\underline{\eta}$ as functions on $I$ with values in $\{\pm 1\}, \Z$, and $\{\pm 1\}$ respectively.
    \item Suppose $\MM \in \Moe^{(P')}(G(W_n))$. Then there is a unique decomposition 
    \[\MM= ((\psi_{-})_{>}, \underline{-1}, \underline{l_{-}},\underline{\eta_{-}})+((\psi_{+})_{>}, \underline{+1}, \underline{l_{+}},\underline{\eta_{+}}),\]
    where $\underline{-1}$ (resp. $\underline{+1}$) is a constant function with value $-1$ (resp. $1$).
\end{enumerate}
\end{remark}

To each parameter $(\psi_{>},  \underline{\zeta},\underline{l},\underline{\eta})$, M{\oe}glin constructed a representation $\pi_{\Moe}(\psi_>,\underline{\zeta},\underline{l},\underline{\eta})$ of $G_n$ which either vanishes or is irreducible and belongs to the local Arthur packet $\Pi_{\psi}$ (\cite{Moe06a, Moe06b, Moe09a, Moe10, Moe11}, see also \cite{Xu17b}). This definition behaves nicely with respect to the Aubert-Zelevinsky involution.

\begin{thm}[{\cite{Moe11}}]\label{thm Moeglin construction} 

Suppose that $(\psi_{>},  \underline{\zeta},\underline{l},\underline{\eta}) \in \Moe(G(W_n))$.
\begin{enumerate}
    \item The representation $\pi_{\Moe}(\psi_>,\underline{\zeta},\underline{l},\underline{\eta})$ is either irreducible or zero.
    \item The local Arthur packet is given by exhausting the above representations. Namely, fixing any admissible order $>$ for $\psi$ and a choice of $\underline{\zeta}$, then
    \[\Pi_\psi =
    \bigcup_{\underline{l},\underline{\eta}}\{\pi_{\Moe}(\psi_>,\underline{\zeta},\underline{l},\underline{\eta})\}\setminus\{0\}.\] 
    Moreover, if $\pi_{\Moe}(\psi_>,\underline{\zeta},\underline{l},\underline{\eta})= \pi_{\Moe}(\psi_>,\underline{\zeta},\underline{l}',\underline{\eta}') \neq 0$, then $(\underline{l},\underline{\eta})=(\underline{l}',\underline{\eta}')$.
    
    \item We have that $\AZ(\pi_{\Moe}(\psi_>,\underline{\zeta},\underline{l},\underline{\eta}))=\pi_{\Moe}(\hat{\psi}_>,-\underline{\zeta},\underline{l},\underline{\eta})$, where $\hat{\psi}(w,x,y)=\psi(w,y,x).$
\end{enumerate}
\end{thm}
\begin{remark}
Part (3) of the above theorem follows from M{\oe}glin's explicit construction. First, if $\psi$ is elementary (see \cite[\S 6]{Xu17b}), the local Arthur packet is constructed by applying composition of generalized Aubert involution to a discrete local Arthur packet $\Pi_{\psi^{\Delta}}$ (\cite[Theorem 6.10]{Xu17b}), and Part (3) holds by definition. Next, if $\psi$ is of discrete diagonal restriction (\cite[\S 7]{Xu17b}), then by the explicit description of $\pi_{\Moe}(\psi_{>}, \underline{\zeta}, \underline{l}, \underline{\eta})$ (\cite[Definition 7.1, Theorem 7.8]{Xu17b}), Part (3) for $\psi$ can be reduced to the case of elementary packets. Finally, general local Arthur packets are constructed from discrete diagonal restriction packets by applying a sequence of highest derivatives (\cite[Proposition 8.5]{Xu17b}), which is compatible with Aubert-Zelevinsky involution. Therefore, Part (3) for general packets can be reduced to the discrete diagonal restriction case.
\end{remark}

Given a summand $\rho\otimes S_{a_i}\otimes S_{b_i}$ of $\psi,$ we define the corresponding Jordan block to be the tuple $(\rho, a_i, b_i)$. We let the set of all Jordan blocks attached to $\psi$ be denoted by $\Jord(\psi).$ We can view the collections $\underline{\zeta},\underline{l},\underline{\eta}$ as functions on $\Jord(\psi)$. For example, we have $\underline{\zeta}(\rho,a_i,b_i)=\zeta_i$. Equivalently, we view  $\Jord(\psi)$ as the set of tuples $(\rho, A_i, {|B_i|}, \zeta_i)$ where $A_i=\frac{a_i+b_i}{2}-1$ and $|B_i|=\frac{|a_i-b_i|}{2}.$

\begin{remark}
    We will not recall the details of the definition of $\pi_\Moe(\MM)$; however, we note the following consequence. Suppose $\MM=(\psi_>,\underline{\zeta},\underline{l},\underline{\eta})$ where $\psi$ is tempered and $\underline{\zeta}$ is a constant function with value $1$. Write $\psi=\bigoplus_{i=1}^r\phi_i\otimes S_{a_i}\otimes S_1$. Then $\pi_{\Moe}(\MM)\neq 0$ if and only if $\underline{\eta}(\phi_i\otimes S_{a_i}\otimes S_1)=\underline{\eta}(\phi_j\otimes S_{a_j}\otimes S_1)$ whenever $\phi_i\otimes S_{a_i}\otimes S_1\cong \phi_j\otimes S_{a_j}\otimes S_1$. In this case, we have
 \begin{equation}\label{eqn Moe tempered}
        \pi_\Moe(\MM)=\pi(\phi_\psi,\varepsilon),
    \end{equation}
    where, $\varepsilon(\phi_i\otimes S_{a_i})=\underline{\eta}(\phi_i\otimes S_{a_i}\otimes S_1).$
\end{remark}

Finally, we let $\Rep_{\Moe}$ be the collection of M{\oe}glin parameters $\MM$ such that $\pi_{\Moe}(\MM)\neq 0$, and let $\Rep_{\Moe}^{(P')}$ be the subset of $\Rep_{\Moe}$ whose order satisfies condition $(P')$.

\subsection{Adams Conjecture}

As in \cite[\S3.2]{GI14}, we fix a pair of characters $\chi_W,\chi_V$ associated to $W_n$ and $V_m$ respectively. These characters depend on the Witt towers containing $W_n$ and $V_m$, not on the dimensions. More specifically, when $\epsilon=1,$ we have $\chi_W$ is the trivial character and $\chi_V$ is the quadratic character associated to $F(\sqrt{\mathrm{disc}V_m})/F.$  When $\epsilon=-1$, we have that $\chi_W$ is the quadratic character associated to $F(\sqrt{\mathrm{disc}W_m})/F$ and $\chi_V$ is trivial. Recall also that $\alpha=2m-2n-\epsilon.$

It is known that the theta correspondence does not preserve $L$-packets. As a remedy, Adams proposed that it should instead preserve local Arthur packets (\cite{Ada89}). We recall Adams' conjecture below.

\begin{conj}\label{conj Adams}
Let $G_n=G(W_n)$ and
    assume that $m>n.$
    Suppose that $\pi\in\Pi(G_n)$ lies in a local Arthur packet $\Pi_\psi$ for some $\psi\in\Psi(G_n).$ If $\theta_{-\alpha}(\pi)\neq 0,$ then $\theta_{-\alpha}(\pi)\in\Pi_{\psi_\alpha}$ where
    \begin{equation}\label{eqn psi_alpha}
    \psi_\alpha=
       (\chi_W\chi_V^{-1}\otimes\psi)\oplus \chi_W\otimes S_1\otimes S_\alpha.
    \end{equation}
\end{conj}

Though the above conjecture may not always hold, M{\oe}glin verified that when $\alpha$ is large, Adams' conjecture is true. See Definition \ref{def Moeglin parameter}(2)(3) for our notation.

\begin{thm}[{\cite[Theorem 5.1]{Moe11c}}]\label{thm Moeglin Adams}
    Let $\alpha\gg 0$, e.g., $\alpha\geq 2^{2n}+1$. Suppose that $\pi\in\Pi_\psi$ for some $\psi\in\Psi_{gp}(G_n)$ and that $\pi=\pi_{\Moe}(\MM)$ with $\MM_{\rho}=((\psi_{\rho})_{>},\underline{\zeta_{\rho}}, \underline{l_{\rho}},\underline{\eta_{\rho}})$. Then $\theta_{-\alpha}(\pi)\in\Pi_{\psi_\alpha}$. Moreover, $\theta_{-\alpha}(\pi)=\pi_{\Moe}(\MM_{\alpha}'),$ where 
    \begin{align*}
        \MM_{\alpha}' = &\left(\bigcup_{\rho\not \cong \chi_V} ((\chi_W\chi_V^{-1}\otimes \psi_{\rho})_{>},\underline{\zeta_{\rho}}, \underline{l_{\rho}},\underline{\eta_{\rho}}) \right)\\
        &\cup \left( (( \chi_W \chi_V^{-1} \psi_{\chi_V})_{>}, \underline{\zeta_{\chi_V}},\underline{l_{\chi_V}}, - \underline{\eta_{\chi_V}})+  (\chi_W\otimes S_1 \otimes S_{\alpha},-1, 0,   \epsilon' )\right),
    \end{align*}
    where we identify the index sets for $\psi_{\rho}$ and $(\chi_W\chi_V^{-1}\otimes \psi_{\rho})$ in an obvious way, and $\epsilon'$ is determined by the sign condition via \eqref{Moeglin sign}. 
\end{thm}

We recall \cite[Corollary 2.4]{BH22} which gives an inverse of the Adams conjecture for $\alpha\gg 0.$

\begin{prop}[{\cite[Corollary 2.4]{BH22}}]\label{prop inverse theta}
    Let $\pi'=\pi_\Moe((\psi_{>})_{\alpha}, \underline{\zeta}',\underline{l}', \underline{\eta}')\neq 0$ and $\psi\in\Psi_{gp}(G_n).$ Assume $\alpha\gg 0.$ Then there exists $\pi=\pi_\Moe(\psi_{>},\underline{\zeta},\underline{l}, \underline{\eta})\neq0$ (of $G_n$ or its pure inner form) such that $\theta_{-\alpha}(\pi)=\pi'$ or $\theta_{-\alpha}(\pi)=\pi'\otimes\mathrm{det}.$ Moreover, $\underline{l}', \underline{\eta}', \underline{\zeta}'$ and $\underline{l}, \underline{\eta}, \underline{\zeta}$ are related as in Theorem \ref{thm Moeglin Adams}, except possibly when $\pi'$ is a representation of an orthogonal group. In this case, $\eta'$ may need to be replaced by $\tilde{\eta}'$ (see Remark \ref{rmk det} below) before applying $\eta'\rightarrow\eta.$
\end{prop}

The remark below explains the situation where there may be confusion between $\pi$ and $\pi\otimes\mathrm{det}$. It is essentially a rephrasing of \cite[Remark 2.3]{BH22}.

\begin{remark}\label{rmk det}
    If $\pi=\pi_\Moe(\psi_{>}, \underline{\zeta},\underline{l}, \underline{\eta})$ is a representation of an even orthogonal group, then
    $\pi\otimes \mathrm{det}=\pi_\Moe(\psi_{>},\underline{\zeta},\underline{l}, \underline{\tilde{\eta}}),$ where
    \[
    \tilde{\eta}(\rho_i,a_i,b_i)=\left\{\begin{array}{cl}
        -{\eta}(\rho_i,a_i,b_i), & \mathrm{if} \ \dim(\rho_i\otimes S_{|a_i-b_i|+1}) \ \mathrm{is \ odd;} \\
        {\eta}(\rho_i,a_i,b_i), & \mathrm{if} \ \dim(\rho_i\otimes S_{|a_i-b_i|+1}) \ \mathrm{is \ even.}
    \end{array}\right.
    \]
\end{remark}

Finally, note that the admissible order of the parameter $\MM_{\alpha}'$ constructed in Theorem \ref{thm Moeglin Adams} may not satisfy condition $(P')$. Thus, we use the change of order formulas in \cite[\S 6]{Xu21b} to define another parameter $\MM_{\alpha}\in \Rep_{\Moe}^{(P')}$ such that $\pi_{\Moe}(\MM_{\alpha})= \pi_{\Moe}(\MM_{\alpha}')$.

\begin{defn}
Suppose that $\MM\in \Rep^{(P')}_{\Moe}$. Write 
\[\MM_{\chi_{V}}= ((\psi_{\chi_V,-})_{>}, \underline{-1},\underline{l_{-}}, \underline{\eta_{-}})+((\psi_{\chi_V,+})_{>}, \underline{+1},\underline{l_{+}}, \underline{\eta_{+}})
\] 
as in Remark \ref{rmk Moeglin parameter}(2). Then we define
\begin{align*}
     \MM_{\alpha} := &\left(\bigcup_{\rho\not \cong \chi_V} ((\chi_W\chi_V^{-1}\otimes \psi_{\rho})_{>},\underline{\zeta_{\rho}}, \underline{l_{\rho}},\underline{\eta_{\rho}}) \right)\\
        &\cup \left( ((\chi_{W}\chi_V^{-1}\psi_{\chi_V,-})_{>}, \underline{-1},\underline{l_{-}}, -\underline{\eta_{-}})+(\chi_W\otimes S_1 \otimes S_{\alpha},-1, 0,   \epsilon )+((\chi_{W}\chi_V^{-1}\psi_{\chi_V,+})_{>}, \underline{+1},\underline{l_{+}}, \underline{\eta_{+}})  \right),
\end{align*}
where $\epsilon$ is determined by the sign condition via \eqref{Moeglin sign}.
\end{defn}
When $\alpha \gg 0$, the admissible order of $\MM_{\alpha}$ in the above definition satisfies $(P')$.
{
\begin{lemma}\label{lemma M_alpha to M_alpha^P'}
Suppose that $\MM\in \Rep^{(P')}_{\Moe}$. Then
    we have that $\MM_{\alpha}\in \Rep_{\Moe}^{(P')}$ and $\pi_{\Moe}(\MM_{\alpha})= \pi_{\Moe}(\MM_{\alpha}')$.
\end{lemma}

\begin{proof}
    This is a direct consequence of Definition \ref{def Moeglin parameter}(4) and \cite[Theorem 6.2]{Xu21b}.
\end{proof}
}

\section{Extended multi-segments and raising operators}\label{sec: extended multi-segments}

In this section, we recall the notions of extended multi-segments and the operators acting on them for symplectic and even orthogonal groups. In the symplectic case  $G_n=\Sp(W_n)$, these combinatorial objects were introduced by Atobe to parameterize local Arthur packets, ultimately, their intersections (Theorem \ref{thm Sp}). The translation of these definitions and algorithms into the notation of \cite{BS25} is given in \S \ref{appendix: Bosnjak Stadler}.

\subsection{Extended multi-segments}
We start with the definition of extended multi-segments. Note that we define it for both symplectic and even orthogonal groups, including pure inner forms.
\begin{defn}(Extended multi-segments)\label{def multi-segment}
\begin{enumerate}
\item
An \emph{extended segment} is a triple $([A,B]_\rho, l, \eta)$,
where
\begin{itemize}
\item
$[A,B]_\rho = \{\rho\lvert\cdot\rvert^A, \rho\lvert\cdot\rvert^{A-1}, \dots, \rho\lvert\cdot\rvert^B \}$ is a segment 
for an irreducible unitary supercuspidal representation $\rho$ of some $\GL_d(F)$; 
\item
$l \in \Z$ with $0 \leq l \leq \frac{b}{2}$, where $b = \#[A,B]_\rho = A-B+1$; 
\item
$\eta \in \{\pm1\}/E$, where $E= \{\pm 1\}$ if $2l=b$ and $\{1\}$ otherwise. 
\end{itemize}

\item
An \emph{extended multi-segment} for $G_n$ is 
a multi-set of extended segments 
\[
\EE = \cup_{\rho}\{ ([A_i,B_i]_{\rho}, l_i, \eta_i) \}_{i \in (I_\rho,>)}
\]
such that 
\begin{itemize}
\item
$I_\rho$ is a totally ordered finite set with a fixed total order $>$ satisfying that 
\begin{align*}\label{eq P}
    \tag{P}A_i>A_j \text{ and } B_i > B_j \ \Longrightarrow \ i>j.
\end{align*}
We call any total order of $I_{\rho}$ with the above property an \emph{admissible order}.

\item
$A_i + B_i \geq 0$ for all $\rho$ and $i \in I_\rho$. 

\item
as a representation of $W_F \times \SL_2(\BC) \times \SL_2(\BC)$, 
\[
\psi_{\EE} = \bigoplus_\rho \bigoplus_{i \in I_\rho} \rho \otimes S_{a_i} \otimes S_{b_i} 
\]
where $(a_i, b_i) = (A_i+B_i+1, A_i-B_i+1)$,
is a local Arthur parameter for $G_n$ of good parity. We shall denote $\psi_{\EE}$ the local Arthur parameter associated with $\EE$. 
\item The sign condition
\begin{align}\label{eq sign condition}
\prod_{\rho} \prod_{i \in I_\rho} (-1)^{\lfloor\frac{b_i}{2}\rfloor+l_i} \eta_i^{b_i}= \epsilon_{G_n}
\end{align}
holds. 
\end{itemize}
We say $\EE$ is an extended multi-segment if it is an extended multi-segment of some $G(W_n)$.

\item We say an admissible order $>$ on $I_{\rho}$ (or $\EE$) satisfies condition $(P')$ if $B_i< B_j$ implies that $i<j$.

\item
We define the \emph{support} of $\EE$ to be the collection of ordered multi-sets 
\[
\supp(\EE) = \cup_{\rho}\{ [A_i,B_i]_{\rho} \}_{i \in (I_\rho,>)}.
\]
\item If the admissible order $>$ is clear in the context, for $k \in I_{\rho}$, we often let $k+1 \in I_{\rho}$ be the unique element adjacent with $k$ and $k+1>k$.

\end{enumerate}
\end{defn}

When $G_n=\Sp(W_n)$, there is an associated representation, which we denote by $\pi_\Ato(\EE)$, defined explicitly in the work of Atobe (\cite[\S3.2]{Ato22a}) using the theory of highest derivatives developed in \cite{AM20}.
The definition does not extend to even orthogonal groups since the theory of highest derivatives is not developed yet. On the other hand, in a joint work with Jiang and Zhang, we 
gave another algorithm to compute $\pi_\Ato(\EE)$ using certain operators (\cite[Algorithms 6.3]{HJLLZ25}). This algorithm readily extends to even orthogonal groups.
We recall the definition of the operators in the next subsection and the algorithm in Algorithm \ref{algo pi_com}.

\subsection{Raising operators}

In this subsection, we recall the operators defined in \cite{Ato22a, HLL22} which were used to classify the intersection of local Arthur packets of symplectic and split odd special orthogonal groups. The combinatorial definitions apply to extended multi-segment for even orthogonal groups, including pure forms, without any modification.

The effect of an operator often only depends on a fixed $\rho.$ To simplify their definitions, we introduce the following notation, analogous to Definition \ref{def Moeglin parameter}(2),(3) for M{\oe}glin parameters.

\begin{defn}
Suppose $\EE= \cup_{\rho} \{([A_i,B_i]_{\rho}, l_i, \eta_i)\}_{i \in (I_{\rho},>)}$ is an extended multi-segment.
\begin{enumerate}
    \item We set
\[ \EE_{\rho}=\{([A_i,B_i]_{\rho}, l_i, \eta_i)\}_{i \in (I_{\rho},>)},\  \EE^{\rho}=\cup_{\rho' \not\cong \rho}\{([A_i,B_i]_{\rho'}, l_i, \eta_i)\}_{i \in (I_{\rho'},>)}. \]
\item Suppose that $I_{\rho}=I_{\rho,1} \cup I_{\rho,2}$ where $i_1<i_2$ for any $i_1 \in I_{\rho,1}$, $i_2 \in I_{\rho,2}$. Then we write
\[ \EE_{\rho}=  \{([A_i,B_i]_{\rho}, l_i, \eta_i)\}_{i \in (I_{\rho,1},>)}+  \{([A_i,B_i]_{\rho}, l_i, \eta_i)\}_{i \in (I_{\rho,2},>)}.\]
Note that $\FF_1+\FF_2 \neq \FF_2+\FF_1$.
\end{enumerate}
\end{defn}

The first operator is called row exchange. Its effect is to change the admissible order on an extended multi-segment. 
Xu studied how M{\oe}glin's parametrization for $\Pi_{\psi}$ changes for different admissible orders of $\psi$ in \cite[Section 6]{Xu21b}. Atobe translated this result for extended multi-segments of split symplectic groups and odd special orthogonal groups (\cite[Section 4.2]{Ato22a}). Since the output of this operator may no longer be an extended multi-segment, we introduce a larger class of combinatorial objects.
\begin{defn}\label{def symbol}
  A {formal extended multi-segment} of $G_n$ is a multi-set of extended segments 
\[
\EE = \cup_{\rho}\{ ([A_i,B_i]_{\rho}, l_i, \eta_i) \}_{i \in (I_\rho,>)},
\] which satisfies the same conditions in Definition \ref{def multi-segment}(2) except we drop the condition $0 \leq l_i \leq \half{b_i}$, for each $i \in I_{\rho}$.
\end{defn}

Now we recall the definition of the row exchange operators $R_k$.

\begin{defn}[Row exchange]\label{def row exchange} 
Suppose $\EE$ is a formal extended multi-segment of $G_n$ where
$$\EE_{\rho}=\{([A_i,B_i]_{\rho},l_i,\eta_i)\}_{i \in (I_{\rho},>)}.$$
For $k<k+1 \in I_{\rho}$, let $\gg$ be the total order on $I_\rho$ defined by $k\gg k+1$ and if $(i,j)\neq (k,k+1)$, then $ i \gg j$ if and only if $
i >j .$ Note that $\gg$ is an admissible order on $I_{\rho}$, i.e. \eqref{eq P} holds, if and only if $[A_k,B_k]_{\rho} \supseteq [A_{k+1},B_{k+1}]_{\rho}$ or $[A_k,B_k]_{\rho} \subseteq [A_{k+1},B_{k+1}]_{\rho}$.

If $\gg$ is not an admissible order, then we define $R_k(\EE)=\EE$. Otherwise, we define 
\[R_{k}(\EE_{\rho})=\{([A_i,B_i]_{\rho},l_i',\eta_i')\}_{i \in (I_{\rho},\gg)},\]
where $( l_i',\eta_i')=(l_i,\eta_i)$ for $i \neq k,k+1$, and $(l_k',\eta_k')$ and $(l_{k+1}', \eta_{k+1}')$ are given as follows: Denote $\epsilon=(-1)^{A_k-B_k}\eta_k\eta_{k+1}$.
\begin{enumerate}
    \item [Case 1.] $ [A_k,B_k]_{\rho} \supset [A_{k+1},B_{k+1}]_{\rho}$:
    
    In this case, we set $(l_{k+1}',\eta_{k+1}')=(l_{k+1}, (-1)^{A_k-B_k}\eta_{k+1})$, and
    \begin{enumerate}
    \item [(a)] If $\epsilon=1$ and $b_k- 2l_k < 2(b_{k+1}-2l_{k+1})$, then
    \[ (l_k', \eta_{k}')= (b_k-(l_k+ (b_{k+1}-2l_{k+1})), (-1)^{A_{k+1}-B_{k+1}} \eta_k).  \]
    \item [(b)] If $\epsilon=1$ and $b_k- 2l_k \geq  2(b_{k+1}-2l_{k+1})$, then
    \[ (l_{k}', \eta_{k}')= (l_k+ (b_{k+1}-2l_{k+1}), (-1)^{A_{k+1}-B_{k+1}+1} \eta_k).  \]
    \item [(c)] If $\epsilon=-1$, then
    \[ (l_{k}', \eta_{k}')= (l_k- (b_{k+1}-2l_{k+1}), (-1)^{A_{k+1}-B_{k+1}+1} \eta_k).  \]
\end{enumerate}
    \item [Case 2.] $ [A_k,B_k]_{\rho} \subseteq [A_{k+1},B_{k+1}]_{\rho}$:
    
    In this case, we set $(l_{k}',\eta_{k}')=(l_{k}, (-1)^{A_{k+1}-B_{k+1}}\eta_{k})$, and
    \begin{enumerate}
   \item [(a)] If $\epsilon=1$ and $b_{k+1}- 2l_{k+1} < 2(b_{k}-2l_{k})$, then
    \[ (l_{k+1}', \eta_{k+1}')= (b_{k+1}-(l_{k+1}+ (b_{k}-2l_{k})), (-1)^{A_{k}-B_{k}} \eta_{k+1}).  \]
    \item [(b)] If $\epsilon=1$ and $b_{k+1}- 2l_{k+1} \geq  2(b_{k}-2l_{k})$,
    then
    \[ (l_{k+1}', \eta_{k+1}')= (l_{k+1}+ (b_{k}-2l_{k}), (-1)^{A_{k}-B_{k}+1} \eta_{k+1}).  \]
    \item [(c)] If $\epsilon=-1$, then
    \[ (l_{k+1}', \eta_{k+1}')= (l_{k+1}- (b_{k}-2l_{k}), (-1)^{A_{k}-B_{k}+1} \eta_{k+1}).  \]
\end{enumerate}
\end{enumerate}
Finally, we define $R_{k}(\EE)= \EE^{\rho} \cup R_{k}(\EE_{\rho})$.
\end{defn}

Let $\EE= \cup_{\rho}\{([A_i,B_i]_{\rho},l_i,\eta_i)\}_{i \in (I_{\rho,>})}$.
If $>'$ is another admissible order of $I_{\rho}$, then we can find a sequence of admissible orders 
$\{>=\gg_0, \gg_1,\ldots, \gg_r=>'\}$ such that  $\gg_{i+1}$ is obtained from $\gg_{i}$ by swapping the order of two adjacent (under $\gg_{i}$) elements $(k_i,k_i+1)$ in $I_{\rho}$. Thus, we give the following definition.
\begin{align}\label{eq change of admissible order}
     \EE_{>'}:= R_{k_{r-1}}\circ \cdots \circ R_{k_1}\circ R_{k_0}(\EE).
\end{align}

The next operator is known as union-intersection.
\begin{defn}[union-intersection]\label{ui def}
 Let $\EE$ be an extended multi-segment. For $k< k+1 \in I_{\rho}$, we define an operator $ui_k$, called union-intersection, on $\EE$ as follows. Write 
 \[ \EE_{\rho}= \{([A_i,B_i]_\rho, l_i,\eta_i)\}_{i \in (I_{\rho},>)}.\]
  Denote $\epsilon=(-1)^{A_k-B_k}\eta_k \eta_{k+1}.$ If $A_{k+1}>A_k$, $B_{k+1}>B_k$ and any of the following cases holds:
\begin{enumerate}
    \item [{Case 1}.] $ \epsilon=1$ and $A_{k+1}-l_{k+1}=A_k-l_k,$
    \item [{Case 2}.] $ \epsilon=1$ and $B_{k+1}+l_{k+1}=B_k+l_k,$
    \item [{Case 3}.] $ \epsilon=-1$ and $B_{k+1}+l_{k+1}=A_k-l_k+1,$
\end{enumerate}
we define
\begin{align*}
     ui_{k}(\EE_{\rho})=\{ ([A_i',B_i']_{\rho},l_i',\eta_i')\}_{i \in (I_{\rho}, >)},
\end{align*} 
where $ ([A_i',B_i']_{\rho},l_i',\eta_i')=([A_i,B_i]_{\rho},l_i,\eta_i)$ for $i \neq k,k+1$, $[A_k',B_k']_{\rho}=[A_{k+1},B_k]_{\rho}$, $[A_{k+1}',B_{k+1}']_{\rho}=[A_k,B_{k+1}]_{\rho}$, and $( l_k', \eta_k', l_{k+1}',\eta_{k+1}' )$ are given case by case as follows:
\begin{enumerate}
    \item[$(1)$] in Case 1, $( l_k', \eta_k', l_{k+1}',\eta_{k+1}' )= (l_k,\eta_k, l_{k+1}-(A_{k+1}-A_k), (-1)^{A_{k+1}-A_k}\eta_{k+1})$;
    \item [$(2)$] in Case 2, if $b_k-2l_k \geq A_{k+1}-A_k$, then
    \[( l_k', \eta_k', l_{k+1}',\eta_{k+1}' )= (l_k+(A_{k+1}-A_k),\eta_k, l_{k+1}, (-1)^{A_{k+1}-A_k}\eta_{k+1}),\]
    if $b_k-2l_k < A_{k+1}-A_k$, then
    \[( l_k', \eta_k', l_{k+1}',\eta_{k+1}' )= (b_k-l_k,-\eta_k, l_{k+1}, (-1)^{A_{k+1}-A_k}\eta_{k+1});\]
    \item [$(3)$] in Case 3, if $l_{k+1} \leq  l_k$, then
    \[( l_k', \eta_k', l_{k+1}',\eta_{k+1}' )= (l_k,\eta_k, l_{k+1}, (-1)^{A_{k+1}-A_k}\eta_{k+1}),\]
    if $l_{k+1}> l_{k}$, then
    \[( l_k', \eta_k', l_{k+1}',\eta_{k+1}' )= (l_k,\eta_k, l_{k}, (-1)^{A_{k+1}-A_k+1}\eta_{k+1});\]
    \item [$(3')$] if we are in Case 3 and $l_k=l_{k+1}=0$, then we delete $ ([A_{k+1}',B_{k+1}']_{\rho},l_{k+1}',\eta_{k+1}')$ from $ui_k(\EE_{\rho})$.
\end{enumerate}
Otherwise, we define $ui_k(\EE_{\rho})=\EE_{\rho}$. In any case, we define $ui_k(\EE)= \EE^{\rho} \cup ui_k(\EE_{\rho})$.

We say $ui_k$ is applicable on $\EE$ or $\EE_{\rho}$ if $ui_k(\EE)\neq \EE$. We say this $ui_k$ is of type 1 (resp. 2, 3, 3') if $\EE_{\rho}$ is in case 1 (resp. 2, 3, 3'). 
\end{defn}

We also extend the definition of union intersection to indices which are adjacent in another admissible order. 

\begin{defn} \label{def ui}
Suppose $\EE$ is an extended multi-segment of $G_n$, and write
\[\EE_{\rho}=\{ ([A_i,B_i]_{\rho},l_i,\eta_i)\}_{i\in (I_{\rho,>})}.\] 
Given $i,j \in I_{\rho}$, we define $ui_{i,j}(\EE)=\EE$ unless
\begin{enumerate}
    \item [1.] We have $ A_i< A_j$, $B_i <B_j$ and $j\gg i$ is adjacent under an admissible order $\gg$ of $I_{\rho}$.
    \item [2.] $ui_i$ is applicable on $\EE_{\gg}$ (see \eqref{eq change of admissible order} for this notation).
\end{enumerate}
In this case, we define $ui_{i,j}(\EE):=(ui_{i}(\EE_{\gg}))_{>}$, so that the admissible order of $ui_{i,j}(\EE)$ and $\EE_{\rho}$ are the same. (If the $ui_i$ is of type 3', then we delete the $j$-th row.) 

We say $ui_{i,j}$ is applicable on $\EE$ if $ui_{i,j}(\EE) \neq \EE$. Furthermore, we say that $ui_{i,j}$ is of type 1, 2, 3, or 3' if the operation $ui_i$ is of type 1, 2, 3, or 3', respectively, in Definition \ref{ui def}.
\end{defn}

\begin{remark}
A priori, it is not clear whether the above definition is independent of the choice of the admissible order $\gg$. If $\EE$ is an extended multi-segment of a symplectic group with $\pi_{\Ato}(\EE)\neq 0$, this is proved in \cite[Proposition 5.3]{HLL22}. The same result for even orthogonal groups is a consequence of the above result of symplectic groups and Lemma \ref{lem EE{Sp}}(2) below.    
\end{remark}

Given an extended multi-segment $\EE$ of a symplectic group satisfying condition $(P')$ and $\pi_{\Ato}(\EE)\neq 0$, Atobe introduced a formula to compute an extended multi-segment $dual(\EE)\in\Rep^{(P')}$ such that $\pi_{\Ato}(dual(\EE))$ is the Aubert-Zelevinsky dual $\AZ(\pi_\Ato(\EE))$ for symplectic groups (\cite[Theorem 6.2]{Ato22a}; see Theorem \ref{thm Sp}(3) below). We recall this definition of $dual$ below.
\begin{defn}[dual]\label{dual segment}
Let $\EE= \cup_\rho \{([A_i,B_i]_{\rho},l_i,\eta_i)\}_{i\in (I_\rho, >)}$ be an extended multi-segment such that the admissible order $>$ on $I_{\rho}$ satisfies $(P')$ for all $\rho$. We define 
$$dual(\EE)=\cup_{\rho}\{([A_i,-B_i]_{\rho},l_i',\eta_i')\}_{i\in (I_\rho, >')}$$ as follows:
\begin{enumerate}
    \item The order $>'$ is defined by $i>'j$ if and only if $j>i.$ 
    \item We set \begin{align*}
l_i'=\begin{cases}
l_i+B_i  & \mathrm{if} \, B_i\in\mathbb{Z},\\
 l_i+B_i+\frac{1}{2}(-1)^{\alpha_{i}}\eta_i  & \mathrm{if} \, B_i\not\in\mathbb{Z},
\end{cases}
\end{align*}
and
\begin{align*}
\eta_i'=\begin{cases}
(-1)^{\alpha_i+\beta_i}\eta_i  & \mathrm{if} \, B_i\in\mathbb{Z},\\
 (-1)^{\alpha_i+\beta_i+1}\eta_i  & \mathrm{if} \, B_i\not\in\mathbb{Z},
\end{cases}
\end{align*}
where $\alpha_{i}=\sum_{j\in I_\rho, j<i}a_j,$ and $\beta_{i}=\sum_{j\in I_\rho, j>i}b_j,$ $a_j=A_j+B_j+1$, $b_j=A_j-B_j+1$.
\item When $B_i\not\in\mathbb{Z}$ and $l_i=\frac{b_i}{2}$, we set $\eta_i=(-1)^{\alpha_i+1}$ in the formula above.
\end{enumerate}
If $\FF= \EE_{\rho}$, we define $dual(\FF):= (dual(\EE))_{\rho}$.
\end{defn}

Finally, we define the last operator which is known as the partial dual.

\begin{defn}[partial dual]\label{def partial dual}
Suppose $\EE$ satisfies $(P')$ and write \[\EE_{\rho}=\FF= \{([A_i,B_i]_{\rho},l_i,\eta_i)\}_{i \in (I_{\rho},>)}.\] For $i \in I_{\rho}$, denote 
\[\alpha_i= \sum_{ j<i} (A_j+B_j+1) ,\ \beta_i= \sum_{j> i} (A_j-B_j+1).\]
Suppose that there exists a $k \in I_{\rho}$ such that
\begin{itemize}
    \item  $B_k=\half{1},l_k=0,$ 
    \item  $(-1)^{\alpha_k}\eta_k= -1$, and
    \item  for any $i < k$, $B_i < \half{1}$.
\end{itemize}
Then we define $dual_k^{+}(\FF)$ as follows. We write the decomposition
\[ \FF= \FF_1 + \{([A_k,1/2]_{\rho},0,\eta_k)\} + \FF_2,\]
where $\FF_1=\FF_{<1/2}$, and 
\[dual(\FF)= \widetilde{\FF_2} + \{([A_k,-1/2]_{\rho},0,(-1)^{\beta_k}) \}+ \widetilde{\FF_1},\]
where $\widetilde{\FF_1}=(dual(\FF))_{>-1/2}$. Finally, write
\[ dual(\widetilde{\FF_2} + \{([A_k,1/2]_{\rho},0,(-1)^{\beta_k+1})\} + \widetilde{\FF_1})= \widetilde{\widetilde{\FF_1}} + \{([A_k,-1/2]_{\rho},0,-\eta_k)\}+ \widetilde{\widetilde{\FF_2}},\]
where $\widetilde{\widetilde{\FF_2}}=(dual(\widetilde{\FF_2} + \{([A_k,1/2]_{\rho},0,(-1)^{\beta_k+1})\} + \widetilde{\FF_1}))_{>-1/2}$. Then we define
\[ dual_k^{+}( \FF)= \widetilde{\widetilde{\FF_1}} + \{([A_k,-1/2]_{\rho},0,-\eta_k)\} + \FF_2, \]
and say $dual_k^{+}$ is applicable on $\FF$.

Suppose $dual(\FF)$ satisfies above condition, then we define
\[ dual_k^{-}(\FF)= dual \circ dual_{k}^{+} \circ dual (\FF), \]
and say $dual_k^{-}$ is applicable on $\FF$.
We call this operator partial dual. Finally, we define $dual_k^{\pm}(\EE)= \EE^{\rho} \cup dual_k^{\pm}(\EE_{\rho}).$ 
\end{defn}

We remark that, by the definition of $dual$, it follows that $dual_k^{-}$ is applicable on $\FF$ if and only if $B_k=-\half{1}$, $l_k=0$ and $B_{j} >-\half{1}$ for all $j >k$ (and then $\eta_k$ is determined by the non-vanishing condition in Definition \ref{def non-vanishing}(1) below). Next, we introduce several notions on extended multi-segments and operators.

\begin{defn}\ 

\begin{enumerate}
\item Let $\EE,\EE'$ be two formal extended multi-segments. We write $\EE \overset{R}{=} \EE'$ if there exists a sequence of row exchange operators $R_{k_1},\ldots, R_{k_s}$ such that
\[ R_{k_1}\circ \cdots \circ R_{k_s}(\EE)= \EE'.  \]
We let $[\EE]:= \{\EE' \ | \ \EE' \overset{R}{=}\EE\}.$
    \item We call $ui^{-1}_{i,j}$, $dual \circ ui_{i,j} \circ dual$, and $dual_{k}^{-}$ raising operators.

    \item We say an extended multi-segment $\EE_{\rho}$ (or $[\EE_{\rho}]$) is absolutely maximal (resp. absolutely minimal) if there is no raising operator (resp. inverse of a raising operator) applicable on $\EE_{\rho}$. We say $\EE$ is absolutely maximal (resp. absolutely minimal) if so is $\EE_{\rho}$ for all $\rho$.
\end{enumerate}
\end{defn}

 The operators we define so far can be applied to local Arthur parameters in an obvious way, which we record below. This allows us to define a partial order $\geq_O$ on the set of local Arthur parameters of good parity.

\begin{defn}\label{def operators on parameters}\ 
\begin{enumerate}
    \item 
Let $\psi$ be a local Arthur parameter of $G_n$ with irreducible summands $\rho \otimes S_{a_i} \otimes S_{b_i}$, $\rho \otimes S_{a_j} \otimes S_{b_j}$, and $\rho \otimes S_{a_k} \otimes S_{b_k}$ of good parity. Set 
\[ A_i:= \half{a_i+b_i}-1,\ B_i:= \half{a_i-b_i}\]
so that $a_i= A_i+B_i+1,$ $b_i=A_i-B_i+1$, and similar for $j,k$.
\begin{enumerate}
    \item [(a)] Write $\psi= \psi'+ \rho \otimes S_{A_i+B_i+1} \otimes S_{A_i-B_i+1}+ \rho \otimes S_{A_j+B_j+1} \otimes S_{A_j-B_j+1}$. Define $ui_{i,j}(\psi):= \psi$ unless 
    $A_{j}> A_i $, $B_j> B_i$ and $A_i \geq B_j-1$, in which case we define 
    \[ ui_{i,j}(\psi):=  \psi'+ \rho \otimes S_{A_j+B_i+1} \otimes S_{A_j-B_i+1}+\rho \otimes S_{A_i+B_j+1} \otimes S_{A_i-B_j+1}. \]
    If $A_i=B_j-1$, then we omit the last term and say this $ui_{i,j}$ is of type 3'.
    \item [(b)] Write $\psi= \psi'+ \rho \otimes S_{A_k+B_k+1} \otimes S_{A_k-B_k+1}$.
    \begin{itemize}
        \item If $B_k=\half{1}$, then we set
        \[ dual_k^{+}(\psi):= \psi'+ \rho \otimes S_{A_k-B_k+1} \otimes S_{A_k+B_k+1}.\]
        Otherwise we set $dual_k^+(\psi):=\psi.$
        \item If $B_k=-\half{1}$, then we set
        \[ dual_k^{-}(\psi):= \psi'+ \rho \otimes S_{A_k-B_k+1} \otimes S_{A_k+B_k+1}.\]
        Otherwise we set $dual_k^-(\psi):=\psi.$
    \end{itemize}
    \item [(c)] We define $dual(\psi):= \widehat{\psi}$ to be another local Arthur parameter given by
    \[  \widehat{\psi}(w,x,y):= \psi(w,y,x). \]
\end{enumerate}
For $T= ui_{i,j}$, $dual_k^{+}$ or $dual_k^{-}$, we say $T$ is applicable on $\psi$ if $T(\psi) \neq \psi$.

\item We define a partial order $\geq_{O}$ on $\Psi_{gp}(G_n)$ by $\psi_1 \geq_O \psi_2 $ if and only if $\psi_1=\psi_2$ or there exists a sequence of raising operators $T_1,\ldots, T_r$ such that
\[ \psi_1= T_r \circ \cdots \circ T_1 (\psi_2).\]
\end{enumerate}
\end{defn}

\begin{remark}\ 
    \begin{enumerate}
        \item The same proof of \cite[Theorem 11.6(1)]{HLL22} implies that $\geq_O$ is a partial order.
        \item If $T$ is a raising operator, then it is not hard to see from Definition that $dual \circ T \circ dual$ is an inverse of a raising operator. Thus, $\psi_1 \geq_O \psi_2$ if and only if $\widehat{\psi_2} \geq_O \widehat{\psi_1}$. Namely, the map $\psi \mapsto \widehat{\psi}$ reverses the partial order $\geq_O$.
    \end{enumerate}
\end{remark}

We end this subsection by defining the shift operator $sh$ and the add operator $add$. The latter plays an important role in Algorithm \ref{algo pi_com}. 

\begin{defn}\label{def shift and add}
Let $\EE = \cup_{\rho}\{ ([A_i,B_i]_{\rho}, l_i, \eta_i) \}_{i \in (I_\rho,>)}$ be an extended multi-segment of $G_n$. For fixed $\rho'$, $j \in I_{\rho'}$, and $d \in \Z$, we define the following operators. It is immediate that the operators commute with each other and so we denote the composition by summation. 

\begin{enumerate}
    \item [1.] $sh_j^{d}(\EE)= \cup_{\rho}\{ ([A_i',B_i']_{\rho}, l_i, \eta_i) \}_{i \in (I_\rho,>)}$ with 
    \[ [A_i',B_i']_{\rho}= \begin{cases}
    [A_i+d,B_i+d]_{\rho} & \text{ if }\rho=\rho' \text{ and } i = j,\\
     [A_i,B_i]_{\rho} & \text{ otherwise, }\end{cases} \]
    and $sh^d_{\rho'}=\sum_{j\in I_{\rho'}} sh_j^{d}$. Also, we define $sh^d:= \sum_{\rho} sh_{\rho}^d.$
     \item [2.] $add_j^{d}(\EE)= \cup_{\rho}\{ ([A_i',B_i']_{\rho}, l_i', \eta_i) \}_{i \in (I_\rho,>)}$ with 
    \[ ([A_i',B_i']_{\rho},l_i')= \begin{cases}
    ([A_i+d,B_i-d]_{\rho},l_i+d) & \text{ if }\rho=\rho' \text{ and } i = j,\\
     ([A_i,B_i]_{\rho},l_i) & \text{ otherwise. }\end{cases} 
    \]
    {Here if $A_i+d=B_i-d -1$ and $l_i+d=0$, it is understood that the $i$-th extended segment is removed.}
Also, we define $add^d_{\rho'}=\sum_{j\in I_{\rho'}} add_j^{d}$, and $add^d:= \sum_{\rho} add_{\rho}^d.$
\end{enumerate}
We will use these notations in the case that the resulting object is still an extended multi-segment. When $d=1$, we often omit it from the notation, e.g., we write $add_j^d=add_j.$
\end{defn}

\subsection{\texorpdfstring{Non-vanishing criterion}{}}\label{sec non-vanishing criterion}
In this subsection, we recall a combinatorial criterion on whether $\pi_{\Ato}(\EE)$ vanishes or not when $\EE$ is an extended multi-segment of a split symplectic or odd special  orthogonal group (see Theorem \ref{thm Sp}(1)). This combinatorial criterion also makes sense for extended multi-segments of pure inner forms of even orthogonal groups (see Theorem \ref{thm non-vanishing}). 

\begin{defn}\label{def non-vanishing}
    Let $\EE=\cup_{\rho}\{([A_i,B_i]_{\rho},l_i,\eta_i)\}_{i \in (I_{\rho},>)}$ be an extended multi-segment whose admissible order satisfies ($P'$) if $B_i<0$ for some $i \in I_{\rho}$.
    \begin{enumerate}
        \item We write $\NV_{\ast}(\EE)\neq 0$ if for any $\rho$ and $i\in I_{\rho}$, the following inequality holds.
    \[ (\ast) \ \ \ \ \ \ \ \ \ B_i+l_i \geq \begin{cases}  0 & \text{ if }B_i \in \Z, \\
    \frac{1}{2} & \text{ if } B_i \not\in \Z \text{ and }\eta_i= (-1)^{\alpha_i+1},\\
        -\frac{1}{2} & \text{ if } B_i \not\in \Z \text{ and }\eta_i= (-1)^{\alpha_i},
        \end{cases} \]
        where 
        \[ \alpha_i:= \sum_{j < i }(A_j+B_j+1). \]

    \item We write $\widetilde{\NV}_{\Xu}(\EE) \neq 0$ if the following condition holds for any $\rho$ and any adjacent elements $k<k+1$ in $(I_{\rho},>)$. Let  $\epsilon= (-1)^{A_k-B_k}\eta_k \eta_{k+1}.$
   \begin{enumerate}
        \item [(1)] If $ A_k\leq  A_{k+1}$, $B_k \leq B_{k+1}$, then
        \[ \begin{cases}
        \epsilon=1 &\Rightarrow B_k+l_k \leq B_{k+1}+l_{k+1}, A_k-l_k \leq A_{k+1}-l_{k+1},\\
        \epsilon=-1 &\Rightarrow  A_k-l_k < B_{k+1}+l_{k+1}.
        \end{cases} \]
        \item [(2)] If $[A_k,B_k]_{\rho}\subset [A_{k+1},B_{k+1}]_{\rho}$, then
        \[ \begin{cases}
        \epsilon=1 &\Rightarrow 0 \leq l_{k+1} -l_{k} \leq b_{k+1}-b_k,\\
        \epsilon=-1 &\Rightarrow  l_k+l_{k+1} \geq b_k.    \end{cases} \]
        \item [(3)] If $[A_k,B_k]_{\rho}\supset [A_{k+1},B_{k+1}]_{\rho}$, then
        \[ \begin{cases}
        \epsilon=1 &\Rightarrow 0 \leq l_{k} -l_{k+1} \leq b_{k}-b_{k+1},\\
        \epsilon=-1 &\Rightarrow  l_k+l_{k+1} \geq b_{k+1}.    \end{cases} \]
           \end{enumerate} 

           \item We write $\NV_{\Xu}(\EE)\neq 0$ if $\widetilde{\NV}_{\Xu}(\EE') \neq 0$ for any $\EE'$ such that $\EE\overset{R}{=}\EE'$.
           \item We write $\NV(\EE)\neq 0$ if $\NV_{\ast}(\EE) \neq 0$ and $\NV_{\Xu}(\EE) \neq 0$.
           We let $\Rep$ be the collection of extended multi-segments $\EE$ such that $\NV(\EE) \neq 0$. We let $\Rep^{(P')}$ be the subset of $\Rep$ consisting of extended multi-segments whose admissible order satisfies $(P')$.
    \end{enumerate}
\end{defn}

\begin{remark}\label{rmk NV EE rho}\ 
\begin{enumerate}
    \item  We have $\NV(\EE)\neq 0$ if and only if $\NV(\EE_{\rho})\neq 0$ for all $\rho$, where $\NV(\EE_{\rho})\neq 0$ is defined in the obvious manner.
    \item The operator $dual$ is an involution on $\Rep^{(P')}$. If $G$ is a symplectic group, this follows from Parts (1), (2) and (3) of Theorem \ref{thm Sp} below. Indeed, since $\pi_{\Ato}(dual(\EE)) = \AZ(\pi_{\Ato}(\EE)) \neq 0$, we have $\NV(dual(\EE))\neq0$ and $dual(\EE) \in \Rep^{(P')}$. Then since $\supp(dual\circ dual (\EE))= \supp(\EE)$ and 
    \[ \pi_{\Ato}(\EE)=\AZ\circ \AZ(\pi_{\Ato}(\EE))=\pi_{\Ato}(dual \circ dual (\EE)),  \]
    we must have $dual\circ dual (\EE)=\EE$. In fact, one can also verify this directly from the combinatorial definition. 
    
    For even orthogonal groups, the result can be reduced to the symplectic case using the extended multi-segment $\EE_{\rho}^{\Sp} $ introduced in the next subsection.
\end{enumerate}
 
\end{remark}

\subsection{\texorpdfstring{Algorithm for $\pi_{\com}(\EE)$}{}}
Let $\EE \in \Rep$ be an extended multi-segment of $G_n$. In this subsection, we recall a combinatorial algorithm to associate a representation $\pi_{\com}(\EE)$ of $G_n$ from \cite[\S 6.1]{HJLLZ25}. If $G_n$ is a split symplectic or odd special orthogonal group, then $\pi_{\com}(\EE)= \pi_{\Ato}(\EE)$. However, we emphasize that the algorithm applies to extended multi-segments of pure inner forms of even orthogonal groups as well. We first give a notation.

\begin{defn}
    If $\NV(\EE) \neq 0$, we let $\mathscr{E}(\EE)$ be the collection of $[\EE']$ such that there exists a collection of raising operators $T_1,\ldots, T_r, T_{r+1},\ldots, T_{s} $ such that
     \[ T_1 \circ \cdots \circ T_r(\EE) \overset{R}{=} T_{r+1} \circ \cdots \circ T_{s}(\EE'). \]
We define $\mathscr{E}(\EE_{\rho})$ similarly so that the following map is a bijection.
\begin{align*}
    \mathscr{E}(\EE)& \to \prod_{\rho} \mathscr{E}(\EE_{\rho}),\\
    [\EE'] &\mapsto ([(\EE')_{\rho}])_{\rho}.
\end{align*}
\end{defn}

Now we extend several facts to extended multi-segments of pure inner forms of even orthogonal groups. Suppose that $\EE=\cup_{\rho}\{([A_i,B_i]_{\rho},l_i,\eta_i)\}_{i \in (I_{\rho},>)}$ is an extended multi-segment of an even orthogonal group $G_n$. For each $\rho$, we let 
\[ \psi_{\rho}:= \bigoplus_{i \in I_{\rho}} \rho \otimes S_{A_i+B_i+1} \otimes S_{A_i-B_i+1},\]
  let $d_{\rho}:= \sum_{i\in I_{\rho}} \dim(\rho)(A_i+B_i+1)(A_i-B_i+1)$, and let $\chi_{\rho}:=\det(\psi_{\rho}): W_F \to \{\pm 1\}$ be the associated quadratic character. We further set
  \[
  \epsilon_\rho:=\prod_{i \in I_\rho} (-1)^{\lfloor\frac{b_i}{2}\rfloor+l_i} \eta_i^{b_i}.
  \]
  We are going to associate an extended multi-segment $\EE^{\Sp}_{\rho}$ of a symplectic group.

If $d_\rho$ is odd, then we fix a symplectic supercuspidal representation $\rho'$ of $\GL_2(F)$ and define
\[ \EE_{\rho}^{\Sp}:= \{ ([A_i,B_i]_{\chi_{\rho} \rho},l_i,\eta_i )\}_{i \in (I_{\rho, >})}\cup\{([3/2,3/2]_{\rho'},0,\epsilon_\rho)\},\]
which is an extended multi-segment of $\Sp_{d_{\rho}+7}(F)$. 
 If $d_{\rho}$ is even and $\rho\not\cong \chi_{\rho}$, then we define
\[ \EE_{\rho}^{\Sp}:= \{ ([A_i,B_i]_{ \rho},l_i,\eta_i )\}_{i \in (I_{\rho, >})} \cup \{([0,0]_{\chi_{\rho}},0,\epsilon_\rho)\},\]
which is an extended multi-segment of $\Sp_{d_{\rho}}(F)$.
Note that if $d_{\rho}$ is even and $\rho$ is a non-trivial quadratic character, then $\chi_{\rho}= \rho^{d_{\rho}}=1$, which belongs to the case above. Finally, if $d_{\rho}$ is even and $\rho=1$, then we let $X$ be the set of non-trivial quadratic characters, and define
\[ \EE_{\rho}^{\Sp}:= \{ ([A_i,B_i]_{ \rho},l_i,\eta_i )\}_{i \in (I_{\rho, >})} \cup \bigcup_{\chi \in X} \{([0,0]_{\chi},0,\epsilon_\rho)\}.\]
Note that this is an extended multi-segment of $\Sp_{d_{\rho}+|X|-1}(F)$ since $|X|$ is odd and $\prod_{\chi \in X} \chi =1$.

\begin{remark}
    The extra extended segments are chosen such that $\psi_{\EE_\rho^\Sp}$ has odd dimension, the sign condition \eqref{eq sign condition} holds, and a raising operator applies on $\EE_\rho^\Sp$ if and only if the corresponding operator applies on $\EE$ (see the lemma below).
\end{remark}

Now from the definition of row exchange operators and raising operators, we have the following direct observation.

\begin{lemma}\label{lem EE{Sp}}
    Let $\EE$ be an extended multi-segment of a pure inner form of an even orthogonal group.
    \begin{enumerate}
        \item $\NV(\EE_{\rho}) \neq0$ if and only if $\NV(\EE^{\Sp}_{\rho}) \neq 0$.
        \item Suppose that $\NV(\EE_{\rho})\neq 0$. If $T$ or $T^{-1}$ is in $\{R_k, ui_{i,j}, dual \circ ui_{i,j} \circ dual, dual_k^{-}\}$ that is applicable on $\EE_{\rho}$, then there is a corresponding operator $T^{\Sp}$ applicable on $\EE^{\Sp}_{\rho}$ such that $(T(\EE_{\rho}))^{\Sp}= T^{\Sp}(\EE_{\rho}^{\Sp})$. In particular, the following map is a bijection
        \begin{align*}
        \mathscr{E}(\EE_{\rho}) & \to \mathscr{E}(\EE^{\Sp}_{\rho}),\\
        [\FF] & \mapsto [\FF^{\Sp}].
    \end{align*}
    \item $\EE_{\rho}$ is absolutely maximal (resp. minimal) if and only if so is $\EE^{\Sp}_{\rho}$. 
    \end{enumerate}
\end{lemma}

As an application, we extend the following combinatorial results on extended multi-segments of symplectic or special odd orthogonal groups to pure inner forms of even orthogonal groups.

\begin{thm}\label{thm E max}
    Let $\EE= \cup_{\rho} \{([A_i,B_i]_{\rho},l_i,\eta_i)\}_{i \in (I_{\rho},>)} \in \Rep^{(P')}$. 
    \begin{enumerate}
        \item There exists a unique member of $\mathscr{E}(\EE)$ that is absolutely maximal (resp. minimal), which we denote by $[\EE^{|max|}]$ (resp. $[\EE^{|min|}]$).
        
        \item Suppose $\EE$ is absolutely maximal and $\psi_{\EE}$ is not tempered. Fix a $\rho$ such that $b_{\rho} := \max \{A_i-B_i+1\ | \ i \in I_{\rho}\} >1 $. Let $ j = \min\{ i \in I_{\rho}\ | \ A_i-B_i+1=b_{\rho} \}$ and assume that $j= \max\{ i \in I_{\rho} \ | \ B_i=B_j \}$ by applying row exchanges if necessary. Let $\EE^{-}:=add_j^{-1}(\EE)$, which satisfies $(P')$ by the assumptions. Then, $\NV( \EE^{-}) \neq 0.$
    \end{enumerate}
\end{thm}
\begin{proof}
 If $\EE$ is an extended multi-segment of a split symplectic or odd special orthogonal group, then the desired conclusion is already proved in \cite[Theorem 11.4]{HLL22} and \cite[Proposition 6.3]{HLLZ25}. If $\EE$ is an extended multi-segment of an even orthogonal group, then the desired properties hold for $\EE_{\rho}^{\Sp}$ for all $\rho$; hence they also hold for $\EE$ as well by Lemma \ref{lem EE{Sp}}. This completes the proof of the theorem.
\end{proof}

Now we state the algorithm for $\pi_{\com}(\EE).$

\begin{algo}\label{algo pi_com}
Suppose that $\EE \in \Rep$.
\begin{enumerate}
    \item [Step 1.] Compute $\EE^{|max|}=\cup_{\rho}\{([A_i,B_i]_{\rho}, l_i, \eta_i)\}_{i \in (I_{\rho},>)}$. If $\psi_{\EE^{|max|}}$ is tempered, then define 
    \[\pi_{\com}(\EE):= \pi\left(\bigoplus_{\rho} \bigoplus_{i \in I_{\rho}} \rho \otimes S_{2 A_i+1},\varepsilon\right),\]
    where for each $i \in I_{\rho}$, $\varepsilon( \rho\otimes S_{2A_i+1} ):=  \eta_i$.

    \item [Step 2.] If $\psi_{\EE^{|max|}}$ is not tempered, then define $(\EE^{|max|})^-= add_j^{-1}(\EE^{|max|})$ for the $j \in I_{\rho}$ chosen in Theorem \ref{thm E max}(2).  Then define
    \[ \pi_{\com}(\EE):= \{\Delta_{\rho}[B_j,-A_j] \}+ \pi_{\com}((\EE^{|max|})^- ).\]
\end{enumerate}
\end{algo}

\begin{remark}\label{rmk EE Sp L data}\ 
\begin{enumerate}
    \item Let $T$ be either a raising operator or $add^{-1}_j$ and write
    \[ \psi_{\EE}= \bigoplus_{i\in I} \rho_i \otimes S_{a_i} \otimes S_{b_i},\ \psi_{T(\EE)}= \bigoplus_{i' \in I'} \rho_{i'} \otimes S_{a_{i'}} \otimes S_{b_{i'
    }}. \]
    Then we have 
    \[ \max_{i \in I}\{A_i-B_i+1\} \geq \max_{i' \in I'}\{A_{i'}-B_{i'}+1\}. \]
    Therefore, if we write
    \[ \pi_{\com}( (\EE^{|max|})^-)= L(\Delta_{\rho_1}[x_1,y_1],\ldots, \Delta_{\rho_f}[x_f,y_f]; \pi_{temp}),\]
then $\half{B_j-A_j} \leq \half{x_i+y_i}$ for all $1 \leq i \leq f$ with $\rho_i=\rho$.  As a consequence, we have an injection
\[ \pi_{\com}(\EE) \hookrightarrow \Delta_{\rho[B_j,-A_j]} \rtimes \pi_{\com}( (\EE^{|max|})^-)\]
from the Langlands classification.
    \item  Let $\EE$ be an extended multi-segment of an even orthogonal group with $\NV(\EE)\neq 0$. Then for any $\rho$, $\NV(\EE_{\rho}^{\Sp})\neq 0$ by Lemma \ref{lem EE{Sp}}; hence   $\pi_{\com}(\EE_{\rho}^{\Sp})$ is also defined. Moreover,
    $\pi_{\com}(\EE)$ and the collection $\{\pi_\com(\EE_{\rho}^{\Sp})\}_{\rho}$ can be read from each other in a straightforward way. Note that $\pi_\com(\EE_\rho^\Sp)=\pi_\Ato(\EE_\rho^\Sp)$ (Theorem \ref{thm Sp}(7)).
\end{enumerate}   
\end{remark}

We let $\Pi_{gp}^\com(G_n)$ denote the set of $\pi\in\Pi(G_n)$ such that $\pi=\pi_\com(\EE)$ for some $\EE.$
We have the following consequence of Algorithm \ref{algo pi_com}.

\begin{cor}\label{cor com intersect}
    Suppose that $\pi_\com(\EE)\neq 0.$ 
    \begin{enumerate}
        \item If $T$ is a raising operator, then $\pi_\com(\EE)=\pi_\com(T(\EE)).$
        \item If $\EE'$ is another extended multi-segment for which $\pi_\com(\EE)=\pi_\com(\EE'),$ {then there exists a sequence of raising operators or their inverses $T_1,\ldots, T_r$ such that
        \[ [\EE']= [ T_1 \circ \cdots \circ T_r(\EE)].\]}
    \end{enumerate}
\end{cor}
\begin{proof}
    Part (1) is clear. For Part (2), by Remark \ref{rmk EE Sp L data}(2), we have $\pi_{\com}(\EE_{\rho}^{\Sp})=\pi_{\com}((\EE_{\rho}')^{\Sp}) $. By the known desired result for symplectic (see Theorem \ref{thm Sp}(5),(7) below), we have $ \EE_{\rho}^{\Sp}$ and $(\EE_{\rho}')^{\Sp}$ are related by a sequence of raising operators and their inverses together with some row exchanges. Then the same holds for $\EE_{\rho}$ and $\EE_{\rho}'$ {by Lemma \ref{lem EE{Sp}}(2)}.
\end{proof}

\subsection{\texorpdfstring{Some consequences for $\pi_{\com}(\EE)$}{}}

We proceed by discussing some consequences of Algorithm \ref{algo pi_com}. The first  result describes  precisely when $\pi_\com(\EE)$ lies in the $L$-packet associated to $\psi_\EE$ (Theorem \ref{thm in L-packet}). The second  result is an algorithm which determines, for a given $\pi\in\Pi(G_n),$ whether there exists $\EE$ such that $\pi=\pi_\com(\EE)$ (Algorithm \ref{algo Arthur type pi^-}). We also prove a technical result for later use (Lemma \ref{lem $L$-data}).

Let $\psi$ be a local Arthur parameter of good parity, and let \[\phi_{\psi}(w,x):=\psi\left(w,x,\begin{pmatrix}
    |w|^{\half{1}} & 0 \\
    0 & |w|^{\half{-1}}
\end{pmatrix}\right)\] be the $L$-parameter associated to $\psi$. Specifically, if we write $$
\psi=\bigoplus_{i=1}^r (\rho_i\otimes S_{a_i} \otimes S_{b_i}).
$$
Then the associated $L$-parameter is 
\begin{align}\label{eq associated L-parameter}
\phi_{\psi}=\bigoplus_{i=1}^r  \left(\bigoplus_{j=0}^{b_i-1} \rho_i \lvert\cdot\rvert^{\frac{b_i-1}{2}-j}\otimes S_{a_i} \right).    
\end{align}

We are going to give a necessary and sufficient condition on $\EE$ such that $\pi_\com(\EE)$ is in the $L$-packet $\Pi_{\phi_{\psi}}$. Beforehand, we remark that the results of this subsection extend to general local Arthur parameters $\psi \in \Psi^{+}(G_n)$ not necessarily of good parity. Indeed, we decompose \[\psi=\psi_{nu,>0}\oplus \psi_{np}\oplus \psi_{gp}\oplus \psi_{np}^{\vee} \oplus \psi_{nu,>0}^{\vee}\] as in Theorem \ref{thm red from nu to gp symp} from which we obtain
\[ \Pi_{\psi}= \{ \tau_{\psi_{nu,>0}} \times \tau_{\psi_{np}} \rtimes \pi \ | \ \pi \in \Pi_{\psi_{gp}}\}. \]
It follows that $\tau_{\psi_{nu,>0}} \times \tau_{\psi_{np}} \rtimes \pi \in \Pi_{\phi_{\psi}}$ if and only if $ \pi \in \Pi_{\phi_{\psi_{gp}}}$.

For symplectic and split odd special orthogonal groups, the necessary and sufficient condition on $\EE$ was called $(L)$ (\cite[Definition 9.1 and Theorem 9.5]{HLL22}). It extends to pure inner forms of even orthogonal groups.

\begin{defn}\label{def (L)}
We say an extended multi-segment $\EE=\cup_{\rho} \{([A_i,B_i]_{\rho},l_i,\eta_i)\}_{i \in (I_{\rho},>)}$ of $G_n$ satisfies $(L)$ if after row exchanges, it satisfies the following conditions:
For all $\rho$ and $i<j \in I_{\rho}$,
\begin{enumerate}
    \item [(i)] $ A_i +B_i \leq A_{j}+B_{j}$;
    \item [(ii)] $(A_i-B_i+1)-2l_i \leq 1$;
    \item [(iii)] if $A_i+B_i=A_{j}+B_{j}$ and $A_i-B_i+1$ is odd, then $\eta_i=\eta_{j}$.
\end{enumerate}
\end{defn}

\begin{prop}\label{prop (L) is maximal}
    If $\EE$ satisfies $(L)$, then $\EE$ is absolutely maximal.
\end{prop}
\begin{proof}
    For symplectic and split odd special orthogonal groups, this is proved in \cite[Proposition 10.3]{HLL22}. For even orthogonal groups, it follows from the fact that $\EE_{\rho}$ satisfies $(L)$ (resp. is absolutely maximal) if and only if $\EE_{\rho}^{\Sp}$ satisfies $(L)$ (resp. is absolutely maximal).
\end{proof}

Now, we extend \cite[Theorem 9.5]{HLL22} to pure inner forms of even orthogonal groups. Namely, we show that the condition $(L)$ characterizes the $L$-packets associated to local Arthur packets.

\begin{thm}\label{thm in L-packet}
    Let $\psi\in\Psi(G_n)$ and $\EE$ be an extended multi-segment of $G_n$ with $\psi_\EE=\psi.$ Then $\pi_{\com}(\EE)\in\Pi_{\phi_\psi}$ if and only if $\EE$ satisfies $(L)$. 
\end{thm}

\begin{proof}
If $\EE$ satisfies $(L)$, then $\EE$ is absolutely maximal by the above Proposition \ref{prop (L) is maximal}.  From Definition \ref{def (L)}, we observe directly that $\EE^-$ is also of type $(L)$. By the definition of $\pi_{\com}$, a straightforward induction argument (where the base case is the tempered case) implies that $\pi_{\com}(\EE) \in \Pi_{\phi_{\psi}}. $ Moreover, the collection of signs $\{ \eta_i \ | \ A_i-B_i+1 \text{ is odd}\}$ gives the character for $\pi_{\com}(\EE)$ in $\mathcal{S}_{\phi_{\psi}}.$

For the converse direction, observe that the number of extended multi-segments $\EE'$ satisfying the condition (i)-(iii) in Definition \ref{def (L)} with $\psi_{\EE'}=\psi$ coincides with the size of the $L$-packet $\Pi_{\phi_{\psi}}$. Thus, if $\pi_{\com}(\EE) \in \Pi_{\phi_{\psi}}$, then there exists an $\EE'$ satisfying the condition $(L)$ such that $\pi_{\com}(\EE')= \pi_{\com}(\EE)$. Now Corollary \ref{cor com intersect}(2) implies that $\EE'=\EE$ after row exchanges. This completes the proof of the theorem.
\end{proof}

Next, we give an algorithm (Algorithm \ref{algo Arthur type pi^-}) which determines when a representation $\pi\in\Pi_{gp}(G_n)$ is of Arthur type by mirroring \cite[Algorithms 6.11]{HJLLZ25}. Much of the ideas of this section follow that of \cite[\S6]{HJLLZ25}. We begin by recalling a definition which forms a key step in the algorithms.

\begin{defn}\label{def pi minus}
Let $\pi\in \Pi(G_n)$ be non-tempered. Write
\[ \pi= L(\Delta_{\rho_1}[x_1,-y_1], \dots , \Delta_{\rho_f}[x_f,-y_f]; \pi(\phi, \varepsilon)).\]
 We define $\pi^{\rho,-}$ to be the representation whose $L$-data is obtained by removing all copies of $\Delta_{\rho}[x,-y]$ from $\pi$ such that
    \begin{align*}
         x&= \min \{ x_i \ | \ x_i -y_i= \min\{ x_j-y_j \ | \  \rho_j \cong \rho \}, {\ \rho_i\cong\rho} \},\\
         y&= x- \min\{ x_j-y_j \ | \  \rho_j \cong \rho \}. 
    \end{align*}
    We shall write $\pi=\pi^{\rho,-}+\{(\Delta_{\rho}[x,-y])^r\}$, where $r$ indicates the multiplicity. 
\end{defn}
\begin{remark}    
    In \cite[\S 6]{HJLLZ25}, we have considered another reduction $\pi_{\rho,-}$.  While we only treat $\pi^{\rho,-}$ in this paper, similar results for $\pi_{\rho,-}$ can be obtained without much difficulty.
\end{remark}

Definition \ref{def pi minus} has an analogue for extended multi-segments as follows.

\begin{defn}\label{def EE_rho_minus}
Suppose that $\EE= \cup_{\rho} \{([A_i,B_i]_{\rho},l_i,\eta_i)\}_{i \in (I_{\rho},>)} \in \Rep^{(P')}$ is absolutely maximal. If necessary, change the admissible order on each $I_{\rho}$ so that $A_i \geq A_j$ if $i>j$ and $B_i=B_j$. For each $\rho$ such that $I_{\rho}\neq \emptyset$ and $b_{\rho} := \max \{A_i-B_i+1\ | \ i \in I_{\rho}\} >1 $, let $ j_1 = \min\{ i \in I_{\rho}\ | \ A_i-B_i+1=b_{\rho} \}$ and denote 
\[\{j \in I_{\rho} \ |\  [A_j,B_j]_{\rho}= [A_{j_1},B_{j_1}]_{\rho} \}= \{j_1, \dots ,j_s\},\]
with $j_1 < \cdots <j_s$. Then define 
$\EE^{\rho,-}:=\sum_{k=1}^s add_{j_k}^{-1}(\EE)$, 
which satisfies $(P')$ by the assumptions.
\end{defn}

We provide a sufficient condition for determining when $\pi_{\com}(\EE^{\rho,-})\neq 0.$

\begin{thm}\label{thm pi_rho_minus}
 Let $\pi \in \Pi_{gp}(G_n)$ be non-tempered and $\pi=\pi^{\rho,-}+\{(\Delta_{\rho}[x,-y])^r\}$. Let $\EE$ be an absolutely maximal extended multi-segment with $\pi_\com(\EE)=\pi$. The following holds.
 \begin{enumerate}
     \item [(a)] We have $\pi_\com(\EE^{\rho,-})\neq 0$. \item [(b)] Moreover, $\pi_\com(\EE^{\rho,-})= \pi^{\rho,-}$. In particular, we have
     \[ \psi_{\EE^{\rho,-}}= \psi_{\EE^{|max|}}-(\rho\otimes S_{y+x+1} \otimes S_{y-x+1})^{\oplus r}+(\rho\otimes S_{y+x+1} \otimes S_{y-x-1})^{\oplus r}. \]
     \item [(c)] The parameter $\psi_{\EE^{|max|}}$ contains exactly $r$ copies of $ \rho\otimes S_{y+x+1} \otimes S_{y-x+1}$.
 \end{enumerate}
\end{thm}

\begin{proof}
    When $G_n$ is symplectic or split odd orthogonal, the result is precisely \cite[Theorem 6.6]{HJLLZ25}. When $G_n$ is an even orthogonal group, the result follows from Lemma \ref{lem EE{Sp}} and Remark \ref{rmk EE Sp L data}(2). 
\end{proof}

Next, we give a definition which is related to computing the inverse of the previous definitions.

\begin{defn}\label{def criterion for Arthur type}
Let $\pi\in\Pi_{gp}(G_n)$ and $\pi=\pi^{\rho,-}+\{(\Delta_{\rho}[x,-y])^r\}$. 
\begin{enumerate}
    \item [(1)] Let $\Psi(\pi^{\rho,-}; \Delta_{\rho}[x,-y],r)$ denote the set of local Arthur parameters $\psi$ such that
\begin{enumerate}
    \item [$\oldbullet$] $\pi^{\rho,-}=\pi_\com(\EE)$ for some $\EE$ with $\psi_\EE=\psi$;
    \item [$\oldbullet$] if $y-x-1>0$, then $\psi$ contains $r$ copies of $\rho\otimes S_{x+y+1}\otimes S_{y-x-1}$; 
    \item [$\oldbullet$] any summand of $\psi$ of the form $\rho \otimes S_{a} \otimes S_b$ satisfies $b \leq y-x+1$, and $a>x+y+1$ if $b=y-x+1$.
\end{enumerate}
For any $\psi \in \Psi(\pi^{\rho,-}; \Delta_{\rho}[x,-y],r)$, we define
\[\psi^{+}:= \psi- (\rho\otimes S_{x+y+1}\otimes S_{y-x-1})^{\oplus r}+(\rho\otimes S_{x+y+1}\otimes S_{y-x+1})^{\oplus r}.\]

\item [(2)] Let $ \mathscr{E}(\pi^{\rho,-}; \Delta_{\rho}[x,-y],r)$ denote the set of extended multi-segments $\EE \in \Rep^{(P')}$ such that $\pi_\com(\EE)=\pi^{\rho,-}$ and $\psi_{\EE}\in \Psi(\pi^{\rho,-}; \Delta_{\rho}[x,-y],r)$. For each $\EE \in \mathscr{E}(\pi^{\rho,-}; \Delta_{\rho}[x,-y],r)$,
we define $\EE^{\rho,+}$ as follows.
\begin{enumerate}
    \item [(i)] If $y-x=1$, then we define $\EE^{\rho,+}$ by inserting $r$ copies of {$([x+1,x]_{\rho},1,1)$} in $\EE$ with admissible order $\gg$ on  $I_{\rho} \sqcup \{j_1,\dots ,j_r\}$, where $j_{k}$ corresponds to the $k$-th copy we inserted, as follows
\[\begin{cases}
j_r \gg j_{r-1} \gg \cdots \gg j_1,\\
\alpha \gg \beta  \Longleftrightarrow \alpha > \beta & \text{for }\alpha, \beta \in I_{\rho},\\
\alpha \gg j_k \Longleftrightarrow {B_\alpha >x} & \text{for }\alpha \in I_{\rho}.
\end{cases}\]
\item [(ii)] If $y-x>1$, then change the admissible order if necessary so that there exists $j_1,\ldots, j_r \in I_{\rho}$ such that
\begin{enumerate}
    \item [$\oldbullet$] $j_1 < \cdots < j_r$ are adjacent under the admissible order on $I_{\rho}$,
    \item [$\oldbullet$] $[A_{j_1},B_{j_1}]_{\rho}= \cdots = [A_{j_r},B_{j_r}]_{\rho}=[y-1,x+1]_{\rho}$,
    \item [$\oldbullet$]  $j_1 = \min\{i \in I_{\rho} \ | \ B_i=B_{j_1}\}$,
    \item [$\oldbullet$] $ A_{j_1}-B_{j_1}+3 \geq A_i-B_i+1$ for all $i \in I_{\rho}$ and the equality does not hold for $i<j_1$.
\end{enumerate}
Then we define $\EE^{\rho,+}:=\sum_{k=1}^r add_{j_k}^{1}(\EE) \in \Rep^{(P')}$.
\end{enumerate}
\end{enumerate}
\end{defn}

The precise relation between these definitions is given in the following lemma.

\begin{lemma}\label{lemma pi^rho^plus} 
Let $\pi\in\Pi_{gp}(G_n)$ and $\EE  \in \mathscr{E}( \pi;\Delta_{\rho}[x,-y],r )$. Then we have that $\pi_\com(\EE^{\rho,+})\neq 0$ and 
    \[ \pi_\com(\EE^{\rho,+})= \pi_\com(\EE^{\rho,+})^{\rho,-}+ \{ (\Delta_{\rho}[x,-y])^r\}. \]
\end{lemma}

\begin{proof}
    When $G_n$ is symplectic or split odd special orthogonal, the result is precisely \cite[Lemma 6.9]{HJLLZ25}. Moreover, it is proved there that any raising operator applicable on $\EE^{\rho,+}$ is also applicable on $\EE.$ Then the same holds for even orthogonal groups by Lemma \ref{lem EE{Sp}} and Remark \ref{rmk EE Sp L data}(2). Namely, we have $\pi_{\com}(\EE^{\rho,+})\neq0$ and we may assume $\EE^{\rho,+} $ is absolutely maximal. Then from the definitions, we have that $(\EE^{\rho,+})^{\rho,-}=\EE$ and thus
        \[ \pi_\com(\EE^{\rho,+})= \pi_\com(\EE^{\rho,+})^{\rho,-}+ \{ (\Delta_{\rho}[x,-y])^r\} \]
        by Theorem \ref{thm pi_rho_minus}.
\end{proof}

From this lemma, we immediately obtain the analogue of \cite[Theorem 6.10]{HJLLZ25} with the same proof.

\begin{thm}\label{thm algo for Arthur type}
Assume that we are in the setting of Definition \ref{def criterion for Arthur type}. Then we have $\pi=\pi_\com(\EE')$ for some extended multi-segment $\EE'$ if and only if $\Psi(\pi^{\rho,-}; \Delta_{\rho}[x,-y],r)$ is non-empty. Moreover, for any $\EE \in  \mathscr{E}(\pi^{\rho,-}; \Delta_{\rho}[x,-y],r)$, we have $\pi=\pi_\com(\EE^{\rho,+})$. 
\end{thm}

The theorem gives us an algorithm to determine whether a representation of good parity is of the form $\pi_\com(\EE)$ for some $\EE$. This algorithm coincides with \cite[Algorithm 6.11]{HJLLZ25} in the symplectic case.

\begin{algo}\label{algo Arthur type pi^-}
Given a representation $\pi\in\Pi_{gp}(G_n)$, proceed as follows:

\begin{enumerate}
    \item [\textbf{Step 1:}] If $\pi$ is tempered,  then we write $\pi=\pi(\phi, \varepsilon)$, where
    \[\phi= \bigoplus_{\rho}\bigoplus_{i\in I_{\rho}} \rho \otimes S_{2A_i+1}.\]
    We equip $I_{\rho}$ with a total order $>$ such that $a_i$ is non-decreasing. Then let 
    \[ \EE:= \cup_{\rho} \{ ([A_i,A_i],0, \varepsilon( \rho \otimes S_{2A_i+1})) \}_{i \in (I_{\rho, >})},\]
    and output $\mathscr{E}(\EE):=\{ [\EE'] \ | \ \pi_\com(\EE')\cong \pi_\com(\EE) \}$ by Corollary \ref{cor com intersect}. We regard the trivial representation of $G_0$ as a tempered representation, and output $\{\EE\}$, where $\EE=\emptyset$.
    \item [\textbf{Step 2:}] If $\pi$ is not tempered, say $\pi^{\rho,-}$ is obtained from $\pi$ by removing $r$ copies of $\Delta_{\rho}[x,-y]$ $($see Definition $\ref{def pi minus})$. Then apply the algorithm on $\pi^{\rho,-}$ to construct the set (possibly empty)
\[ \{[\EE'] \ | \ \pi_\com(\EE') \cong \pi^{\rho,-}\}.\]
The representation $\pi$ is of the form $\pi_\com(\EE)$ for some $\EE$ if and only if $\mathscr{E}(\pi^{\rho,-}; \Delta_{\rho}[x,-y],r) $ is non-empty. If this set is non-empty, take any $\EE$ in this set and then output $\mathscr{E}(\EE^{\rho,+})$ by Corollary \ref{cor com intersect}.
\end{enumerate}
\end{algo}

Our next goal is to establish an analogue of \cite[Lemma 4.7]{HLL22} (see Lemma \ref{lem $L$-data} below). This is a technical lemma which is needed later (for Theorem \ref{thm N ordering}).
We begin by recalling some definitions related to this lemma.
First, we associate a multi-set with an irreducible representation. 
\begin{defn}\label{def Omega pi}
Consider $\pi\in\Pi(G_n)$ and write
\[ \pi= L\left( \Delta_{\rho_1}[x_1,-y_1],\dots,\Delta_{\rho_t}[x_t,-y_t]; \pi(\sum_{j=t+1}^m \rho_j\otimes S_{2z_j+1},\varepsilon) \right). \]
We define
\[ \Omega(\pi):= \{ \rho_1\lvert\cdot\rvert^{x_1},\dots,\rho_t\lvert\cdot\rvert^{x_t} \} + \{ \rho_1\lvert\cdot\rvert^{y_1},\dots, \rho_t\lvert\cdot\rvert^{y_t} \} + \{ \rho_{t+1}\lvert\cdot\rvert^{z_{t+1}},\dots,\rho_m\lvert\cdot\rvert^{z_m} \}. \]
We denote $\Omega(\pi)_\rho$ to be the maximal sub-multi-set of $\Omega(\pi)$ whose elements are all of the form $\rho\lvert\cdot\rvert^x$ for some $x \in \R$.
\end{defn}

Next, we associate a similar multi-set to an extended multi-segment.
\begin{defn}
Suppose we have a union of extended segments of $G_n$, say
\[ \FF=\{([A_i,B_i]_{\rho},l_i,\eta_i)\}_{i \in (I_{\rho, >})}. \]
For example, $\FF= \EE_{\rho}$ for some extended multi-segment $\EE$. We define the following ordered multi-set  
\begin{align*}
    \Omega(\FF):=& \sum_{i\in I_{\rho}} [A_i,B_i]_{\rho}= \{ \rho\lvert\cdot\rvert^{\alpha_1} ,\dots,\rho\lvert\cdot\rvert^{\alpha_t} \}.
\end{align*}
where $\alpha_1 \leq \cdots\leq \alpha_t$. 
Suppose $\EE=\cup_{\rho} \EE_{\rho}$ is an extended multi-segment. Fix an arbitrary order on the set 
\[ \{ \rho \ | \ \EE_{\rho} \neq \emptyset \}= \{ \rho_1,\dots,\rho_r\}.\]
We define the multi-sets $\Omega(\EE)$ 
to be the sum of multi-sets 
\begin{align*}
    \Omega(\EE)&:= \Omega(\EE_{\rho_1}) +\cdots+ \Omega(\EE_{\rho_r}).
\end{align*}
\end{defn}

When $G_n=\Sp(W_n),$ \cite[Lemma 4.7]{HLL22} relates $\Omega(\EE)$ and $\Omega(\pi)$ when $\pi=\pi_\Ato(\EE)=\pi_\com(\EE)$ (see Lemma \ref{lem $L$-data} below); however, the proof there makes use of the theory of derivatives which we currently do not have access to for even orthogonal groups. Consequently, we extend the result to even orthogonal groups below.

\begin{lemma}\label{lem $L$-data}
For any $\EE \in \Rep$, we have 
\[ \Omega(\EE_{\rho}) \supseteq \Omega(\pi_\com(\EE))_{\rho}\]
as multi-sets. Moreover, 
\begin{enumerate}
    \item [(i)] If {$\EE^{\rho} \cup sh^{-1}(\EE_{\rho})$ is again an extended multi-segment} and  $\pi_\com(\EE^{\rho} \cup sh^{-1}(\EE_{\rho})) \neq 0$, then $\Omega(\EE_{\rho}) = \Omega(\pi_\com(\EE))_{\rho}$.
    \item [(ii)] The difference multi-set $\Omega(\EE_{\rho})\setminus \Omega(\pi_\com(\EE))_{\rho}$ is symmetric about $\rho\lvert\cdot\rvert^{-1/2}$ in the following sense: The multiplicity of $\rho\lvert\cdot\rvert^x$ in $ \Omega(\EE_{\rho})\setminus\Omega(\pi_\com(\EE))_{\rho}$ is the same as that of  $\rho\lvert\cdot\rvert^{-x-1}$.
\end{enumerate}
\end{lemma}

{
\begin{proof}
When $G_n=\Sp(W_n),$ the result is precisely \cite[Lemma 4.7]{HLL22}. Assume then that $\EE$ is an extended multi-segment for $\OO(W_n).$  Recall the extended multi-segment $\EE_\rho^\Sp$ from \S\ref{sec non-vanishing criterion}. {The set  $\Omega(\EE_\rho)$ can be directly read from $\Omega(\EE_\rho^\Sp).$} Similarly, by Lemma \ref{lem EE{Sp}} and Algorithm \ref{algo pi_com}, {the set $\Omega(\pi_\com(\EE))_\rho$ can be directly read from $\Omega(\pi_\com(\EE_\rho^\Sp))$} (see also Remark \ref{rmk EE Sp L data}(2)). Thus, the result for symplectic groups (\cite[Lemma 4.7]{HLL22}) implies that
\[ \Omega(\EE_{\rho}) \supseteq \Omega(\pi_\com(\EE))_{\rho}.\] 
We note that Part (ii) follows directly from the above observations since the difference multi-set $\Omega(\EE_{\rho})\setminus \Omega(\pi_\com(\EE))_{\rho}$ can be directly read from  $\Omega(\EE_{\rho}^{\Sp})\setminus \Omega(\pi_\com(\EE^{\Sp}))$.

We prove Part (i) now.
We suppose that $\pi_\com(\EE^{\rho} \cup sh^{-1}(\EE_{\rho})) \neq 0$. By Remark \ref{rmk NV EE rho}(1), this is equivalent to $\NV(sh^{-1}(\EE_{\rho}))\neq 0$ (since we assumed that $\EE\in\Rep$ and so $\NV(\EE^\rho)\neq0$). Thus, we have $\NV(sh^{-1}((\EE_\rho^\Sp)_\rho))\neq 0$. We also have $\NV(\EE_\rho^\Sp)\neq 0$ and so $\pi_\com((\EE_\rho^\Sp)^{\rho} \cup sh^{-1}((\EE_\rho^\Sp)_\rho)) \neq 0$. The result for symplectic groups (\cite[Lemma 4.7(i)]{HLL22}) implies that $\Omega((\EE_\rho^\Sp)_{\rho}) = \Omega(\pi_\com(\EE_\rho^\Sp))_{\rho}$. However, as noted above, $\Omega(\EE_\rho)=\Omega((\EE_\rho^\Sp)_\rho)$ and $\Omega(\pi_\com(\EE))_\rho=\Omega(\pi_\com(\EE_\rho^\Sp))_\rho$. Therefore we obtain $\Omega(\EE_{\rho}) = \Omega(\pi_\com(\EE))_{\rho}$. This proves Part (i) and completes the proof of the lemma.
\end{proof}
}

\subsection{Translation from M{\oe}glin parameters}

We recall the translation from M{\oe}glin parameters to extended multi-segments given in \cite[Theorem 6.6]{Ato22a} for $\SO_{2n+1}$ and $\Sp_{2n}$. The definition works for general extended multi-segments. 

\begin{defn}\label{defn Moeglin to Atobe dictionary}
Let $\psi\in\Psi_{gp}(G_n),$ $\MM=(\psi_>,\underline{\zeta},\underline{l},\underline{\eta})\in\Moe^{(P')}(G_n)$, and $I_\rho^\pm=\{i\in I_\rho \ | \ \zeta_i=\pm1\}.$
Write
\[
dual(\cup_\rho\{([A_i,-B_i],l_i,{\eta}_i)\}_{i\in (I_\rho^-,>)})=\cup_\rho\{([A_i,B_i],\hat{l}_i,\hat{\eta}_i)\}_{i\in (I_\rho^-,>')}.
\]
We define a function from  $\Moe^{(P')}(G_n)$ to extended multi-segments of $G_n$ via $\MM\mapsto\EE_\MM,$ where
\[
\EE_\MM:=\cup_\rho\left(\{([A_i,B_i],\hat{l}_i,\hat{\eta}_i)\}_{i\in (I_\rho^-,>')}\cup \{([A_i,B_i],l_i,\eta_i)\}_{i\in (I_\rho^+,>)}\right).
    \]
\end{defn}

\begin{remark}\label{rmk Moe Ato dictionary}
    If $\EE \in \Rep^{(P')}$, then Remark \ref{rmk NV EE rho}(2) implies that there exists an $\MM \in \Moe^{(P')}$ such that $\EE=\EE_{\MM}$.
\end{remark}

\subsection{Results for symplectic groups}
In this subsection, we collect several facts for the representations $\pi=\pi_{\com}(\EE)$ of a split symplectic group which study the structure of the set
\[
\Psi(\pi)=\{\psi\in\Psi(G_n) \ | \ \pi\in\Pi_\psi\}.
\]
We remark that the analogous results hold for split odd special orthogonal groups, but we do not require these results.

\begin{thm}\label{thm Sp}
    Let $\EE$ be an extended multi-segment of $\Sp(W_n)$.
    \begin{enumerate}

    \item $($\cite[Theorem 3.7, 4.4]{Ato22a}$)$ We have $\pi_{\Ato}(\EE) \neq 0$ if and only if $\NV(\EE) \neq 0$.

        \item $($\cite[Theorem 3.3]{Ato22a}$)$
Suppose $\psi= \bigoplus_{\rho} \bigoplus_{i \in I_{\rho}} \rho \otimes S_{a_i} \otimes S_{b_i}$ is a local Arthur parameter of good parity of $\Sp(W_n)$. Choose an admissible order $>_{\psi}$ on $I_{\rho}$ for each $\rho$ that satisfies ($P'$) if $\half{a_i-b_i}<0$ for some $i \in I_{\rho}$. Then
\[ \bigoplus_{\pi \in \Pi_{\psi}} \pi= \bigoplus_{\EE} \pi_{\Ato}(\EE),\]
where $\EE$ runs over all extended multi-segments with $\supp(\EE)= \supp(\psi)$ (for a fixed admissible order) and $\pi_\Ato(\EE) \neq 0$.

\item $($\cite[Theorem 6.2]{Ato22a}$)$ Let $\EE$ be an extended multi-segment of $\Sp(W_n)$ with an admissible order satisfying $(P')$. If $\NV(\EE)\neq 0,$ then $\pi_{\Ato}(dual(\EE))=\mathsf{AZ}(\pi_{\Ato}(\EE)).$

\item $($\cite[Theorem 6.6]{Ato22a}$)$ Suppose that $\MM\in \Rep_{\Moe}^{(P')}(\Sp(W_n))$. Then $ \pi_{\Moe}(\MM)= \pi_{\Ato}(\EE_{\MM})$.

\item $($\cite[Theorem 1.4]{HLL22}$)$ Suppose that $\pi_{\Ato}(\EE) \neq 0$. Then the following holds.
    \begin{enumerate}
        \item If $T$ is a raising operator, then $\pi_\Ato(\EE)=\pi_\Ato(T(\EE)).$
        \item If $\EE'$ is another extended multi-segment for which $\pi_\Ato(\EE)=\pi_\Ato(\EE'),$ {then there exists a sequence of raising operators or their inverses $T_1,\ldots, T_r$ such that
        \[ [\EE']= [ T_1 \circ \cdots \circ T_r(\EE)].\]}
    \end{enumerate}
    \item  $($\cite[Theorem 1.7]{HLL22}$)$
    Let $\pi$ be a representation of $\Sp(W_n)$ of Arthur type. Then there exists unique maximal and minimal elements of $\Psi(\pi)$ with respect to $\geq_O$ (see Definition \ref{def operators on parameters}(2)). We denote these elements by $\psi^{max}(\pi)$ and $\psi^{min}(\pi)$, and the corresponding class of extended multi-segment by $[\EE^{|max|}]$ and $[\EE^{|min|}]$, respectively.

    \item $($\cite[Algorithm 6.3]{HJLLZ25}$)$ For any $\EE \in \Rep$, we have $\pi_{\com}(\EE)= \pi_{\Ato}(\EE)$.
    \end{enumerate}
\end{thm}
{
\begin{remark}
 Part (5) of the above theorem is logically equivalent to \cite[Theorem 1.4]{Ato23}. See \cite[Remark 10.5]{HLL22}.   
\end{remark}
}
The main goal of this paper is to establish the analogous results for even orthogonal groups (with $\pi_\Ato(\EE)$ replaced by $\pi_\theta(\EE)$ defined in the next section).

\section{\texorpdfstring{Intersection of local Arthur packets for $\OO_{2n}(F)$}{}}\label{sec intersections for O(2n)}

In this section, we attach representations to extended multi-segments of even-orthogonal groups and show that the results of Theorem \ref{thm Sp} extend to this setting. We begin by attaching the representations. We remark that we are primarily concerned in this section with the case that $\epsilon=-1$ and so $G_n=G(W_n)$ is an even orthogonal group and $H_m=G(V_m)$ is a split symplectic group. To emphasize that $G_n$ is orthogonal, we sometimes write $G_n=\OO(W_n).$ However, there are several instances where we consider $\epsilon=1$ which we emphasize by writing $G_n=\Sp(W_n).$
\subsection{Representations associated to extended multi-segments}

Recall that in the symplectic case, $\pi_\Ato(\EE)$ is defined using highest derivatives, but the formula of highest derivatives has not been established for even orthogonal groups. Our remedy is to use the  theta correspondence instead. First, we give a definition which will be useful to connect extended multi-segments of even orthogonal groups with symplectic groups.
\begin{defn}\label{defn EE_alpha}
    Let $\EE$ be an extended multi-segment of $G_n=\OO(W_n)$ and $\alpha$ be a positive odd integer. Write
    \[
    \EE = \cup_{\rho\in R}\{ ([A_i,B_i]_{\rho}, l_i, \eta_i) \}_{i \in (I_\rho,>)}.
    \]
    Let $\EE_{W,V}=\cup_{\rho\in R}\{ ([A_i,B_i]_{\chi_W\chi_V\inv\rho}, l_i, \eta_i) \}_{i \in (I_\rho,>)}$ be the extended multi-segment defined by 
    replacing each $\rho$ with $\chi_W\chi_V\inv\rho$.
We define
    \[
    \EE_\alpha=\left\{\left(\left[\frac{\alpha-1}{2},-\frac{\alpha-1}{2}\right]_{\chi_W}, \frac{\alpha-1}{2},\epsilon_{G_n} \right)\right\}+ \EE_{W,V}.
    \]
    We remark that the added extended segment is inserted as first extended segment in the admissible order, i.e. corresponds to the minimal element in the total ordered set for $\chi_W$. The extended segments in $\EE_{W,V}$ are ordered similarly to $\EE.$ We also set $\alpha_0=2^{2n}+1.$
\end{defn}

This definition allows us to attach a representation to extended multi-segments of even orthogonal groups using the Adams conjecture.

\begin{prop}\label{prop orthog extended multi-segment rep}
    Let $\EE$ be an extended multi-segment of $G_n=\OO(W_n)$ and $\alpha$ be a positive odd integer. Let $\EE_\alpha$ be the extended multi-segment defined above.
    Then $\EE_{\alpha}$ is an extended multi-segment of $\Sp_{2n+\alpha-1}(F)$. Furthermore, let $\alpha_0=2^{2n}+1.$ If $\pi_{\Ato}(\EE_{\alpha_0})\neq0$, then there exists a unique $\pi_\EE\in\Pi_{\psi_\EE}(G_n)$ such that $\theta_{-\alpha}(\pi_\EE)=\pi_{\Ato}(\EE_{\alpha})$ for any odd $\alpha\geq\alpha_0.$
\end{prop}

\begin{proof}
The uniqueness of $\pi_{\EE}$ is guaranteed by Howe duality (Theorem \ref{thm Howe duality}), and we prove the existence now. It follows from \eqref{eq sign condition} that $\EE_{\alpha}$ is an extended multi-segment of $\Sp_{2n+\alpha-1}(F)$. Furthermore, we have \[\psi_{\EE_\alpha}=(\chi_W\chi_V\inv\otimes\psi_{\EE})\oplus\chi_W\otimes S_1\otimes S_{\alpha}.\] Suppose that $\pi_{\Ato}(\EE_\alpha)\neq 0.$ By Theorem \ref{thm Sp}(1), we have $\pi_{\Ato}(\EE_\alpha)\in\Pi_{\psi_{\EE_\alpha}}.$ From M{\oe}glin's parameterization of the local Arthur packet (Theorem \ref{thm Moeglin construction}), write
    \[
    \pi_{\Ato}(\EE_{\alpha})=\pi_\Moe((\psi_{\EE_{\alpha}})_>,\underline{\zeta},\underline{l},\underline{\eta}).
    \]
    Let $\alpha\gg 0$ such that the Jordan block $(\chi_W, \frac{\alpha-1}{2},\frac{\alpha-1}{2},-1)$ is far away from the rest of the Jordan blocks (in the sense of \cite[\S2]{Xu21b}). In particular, we may take $\alpha\geq\alpha_0=2^{2n}+1.$ By Proposition \ref{prop inverse theta}, there exists a $\pi_{\EE}=\pi_\Moe((\psi_{\EE})_>,\underline{\zeta}',\underline{l}',\underline{\eta}')\neq0$ of $G_n$ or its pure inner form, such that $\theta_{-\alpha}(\pi_{\EE})=\pi_\Ato(\EE_\alpha).$ From the recipe in Theorem \ref{thm Moeglin Adams} and the definition of $\EE_{\alpha}$, we see that $\pi_{\EE}$ is a representation of $G_n$.
    Note that the parameter $((\psi_{\EE})_>,\underline{\zeta}',\underline{l}',\underline{\eta}')$ is independent of $\alpha$. This completes the proof of the proposition.
\end{proof}

The above proposition allows us to attach a representation to extended multi-segments of quasi-split even orthogonal groups.

\begin{defn}\label{def orthog rep of segment}
    Let $\EE$ be an extended multi-segment of $G_n=\OO(W_n)$. If $\pi_{\Ato}(\EE_{\alpha_0})\neq0$, then we define $\pi_\theta(\EE)=\pi_\EE$ as in Proposition \ref{prop orthog extended multi-segment rep}. Otherwise, we set $\pi_\theta(\EE)=0.$ 
\end{defn}

\begin{remark}
   Proposition \ref{prop inverse theta} implies that for any $\pi_{\alpha} \in \Pi_{\psi_{\alpha}}$, there exists a representation $\pi\in \Pi_{\psi}(\OO(W_n^\epsilon))$ such that $ \pi_{\alpha}= \theta_{-\alpha}(\pi)$, for some $\epsilon\in \{\pm 1\}$. With the notation of extended multi-segments, we describe the sign $\epsilon$ explicitly. Write $\pi_{\alpha}= \pi_\Ato(\EE)$. After row exchanges, we may write
   \[ \EE_{\chi_{W}}= \{([(\alpha-1)/2, (1-\alpha)/2]_{\chi_W}, (\alpha-1)/2, \eta)\}+ \FF.\]
   Then $\epsilon=\eta$.
\end{remark}

Let $\EE$ be an extended multi-segment of $G_n=\OO(W_n)$. We give a computable extrinsic algorithm to compute the Langlands classification of $\pi_\theta(\EE)$ (or determine its vanishing).

\begin{algo}\label{algo orthog Langlands Classification}
    Let $\EE$ be an extended multi-segment of $G_n=\OO(W_n)$ and $\alpha_0=2^{2n}+1.$
    Recall $\EE_{W,V}$ is the extended multi-segment defined by replacing each $\rho$ with $\chi_W\chi_V\inv\rho$ and as in Definition \ref{defn EE_alpha}, we set \[
    \EE_{\alpha_0}=\left\{\left(\left[\frac{\alpha_0-1}{2},-\frac{\alpha_0-1}{2}\right]_{\chi_W}, \frac{\alpha_0-1}{2},\epsilon_{G_n} \right)\right\}+ \EE_{W,V}.
    \]
    The admissible order is induced from that of $\EE$ with the added caveat that it should satisfy $({P}').$ 
    \begin{enumerate}
        \item If $\pi_\Ato(\EE_{{\alpha_0}})=0$ (which can be checked explicitly by Theorem \ref{thm Sp}(1)), then $\pi_\theta(\EE)= 0$ and the algorithm terminates. Otherwise, proceed to the next step.
        \item Compute the Langlands classification for $\pi_\Ato(\EE_{{\alpha_0}})$ (for example, via Algorithm \ref{algo pi_com} and Theorem \ref{thm Sp}(7)). 
        \item Since $\pi_\Ato(\EE_{{\alpha_0}})=\theta_{-\alpha_0}(\pi_\theta(\EE)),$ we compute the Langlands classification of $\pi_\theta(\EE)=\theta_{\alpha_0}(\theta_{-\alpha_0}(\pi_\theta(\EE)))$ by explicit computation of the theta lifts given in \cite{AG17a, BH21}. Here $\theta_\alpha$ denotes the theta lift from $\Sp(V_{n+\half{\alpha_0-1}})$ to $\OO(W_n)$. 
    \end{enumerate}
   
\end{algo}

We remark that Algorithm \ref{algo orthog Langlands Classification} is less than desirable for several reasons. In particular, it is not intrinsic to even orthogonal groups. For theoretical applications, it is desirable to have a simpler approach to computing the Langlands classification for $\pi_\theta(\EE)$. Later, we will achieve this by demonstrating that $\pi_\theta(\EE)=\pi_\com(\EE)$ (Theorem \ref{thm theta=com}).

Recall from Theorem \ref{thm Sp}(4) that Atobe showed how to translate between the parameterizations of local Arthur packets. More specifically, if $G_n=\Sp(W_n)$ and $\MM\in\Moe^{(P')}(G_n)$, then $\pi_\Ato(\EE_\MM)=\pi_\Moe(\MM)$.
We extend this to the orthogonal case.

\begin{thm}\label{cor Moeglin to Atobe dictionary orthog}
    Suppose that $\MM\in \Rep_{\Moe}^{(P')}(\OO(W_n))$. Then $ \pi_{\Moe}(\MM)= \pi_{\theta}(\EE_{\MM})$.
\end{thm}

\begin{proof}
Let $G_n=\OO(W_n)$ and $\pi=\pi_{\Moe}(\MM).$
    By Howe duality (Theorem \ref{thm Howe duality}), it is enough to show that $\theta_{-\alpha}(\pi_\theta(\EE_\MM))=\theta_{-\alpha}(\pi)$ for some $\alpha\gg 0.$ From Lemma \ref{lemma M_alpha to M_alpha^P'}, it suffices to verify the combinatorial identity
    \begin{equation}\label{eqn Moeglin to Atobe dictionary}
    \EE_{(\MM_{\alpha})} =(\EE_\MM)_{\alpha}.
    \end{equation}
    It is immediate from the definition that $(\EE_{\MM_{\alpha}})_{\rho}= ((\EE_\MM)_{\alpha})_{\rho}$ for $\rho\not\cong \chi_{W}$. Thus, we focus on the $\chi_W$ Jordan block in the rest of the proof. 
       Write 
           \[
    (\EE_{\MM})_{\chi_V}=dual(\{([A_i,-B_i]_{\chi_V},l_i,\eta_i)\}_{i\in (I_{\chi_V}^-,>')}) \cup \{([A_i,B_i]_{\chi_V},l_i,\eta_i)\}_{i\in( I_{\chi_V}^+,>)}.
    \]
We compare the following multi-set of extended segments
\begin{align*}
    \FF&:= dual(\{([A_i,-B_i]_{\chi_W},l_i,\eta_i)\}_{i\in (I_{\chi_V}^-,>')}),\\
    \FF'&:= dual\left(\{([A_i,-B_i]_{\chi_W},l_i,-\eta_i)\}_{i\in (I_{\chi_V}^-,>')}+\{([\half{\alpha-1},\half{\alpha-1}]_{\chi_W},0,\epsilon_{G_n})\}\right),
\end{align*}
following Definition \ref{dual segment}. Note that $B_i \in \Z$. Write $\alpha_i, \beta_i$ (resp. $\alpha_i', \beta_i'$) for the ingredients in the formula for $\FF$ (resp. $\FF'$).  Then for $i \in I_{\chi_V}^-$ we have $\alpha_i' =\alpha_i$ and $\beta_i'=\beta_i+1$. Then the formula for $dual$ implies that
\[ \FF'= \left\{\left(\left[\frac{\alpha-1}{2},-\frac{\alpha-1}{2}\right]_{\chi_W}, \frac{\alpha-1}{2}, \widehat{\eta} \right)\right\}+ \FF.\]
for some $\widehat{\eta}$. So far we have explained that 
\[ \EE_{(\MM_{\alpha})}= \left\{\left(\left[\frac{\alpha-1}{2},-\frac{\alpha-1}{2}\right]_{\chi_W}, \frac{\alpha-1}{2}, \widehat{\eta} \right)\right\}+(\EE_{\MM})_{W,V}  \]
for some $\widehat{\eta}$. However, we know
$\EE_{(\MM_{\alpha})}$ and $(\EE_{\MM})_{\alpha}$ are both extended multi-segments of a symplectic group, and 
\[ (\EE_{\MM})_{\alpha}= \left\{\left(\left[\frac{\alpha-1}{2},-\frac{\alpha-1}{2}\right]_{\chi_W}, \frac{\alpha-1}{2}, \epsilon_{G_n} \right)\right\}+(\EE_{\MM})_{W,V}.  \]
Thus, the sign condition \eqref{eq sign condition} implies that $\widehat{\eta}=\epsilon_{G_n}$ and $\EE_{(\MM_{\alpha})}=(\EE_{\MM})_{\alpha}$.
\end{proof}

Let $\pi=\pi_\theta(\EE).$
As a corollary of the above theorem, we can explain how to find $\EE'$ such that $\pi_{\theta}(\EE')=\pi\otimes\det.$

\begin{cor}\label{cor extended multi-segment for det}
    Let $\EE$ be an extended multi-segment of $G_n=\OO(W_n)$ and $\pi=\pi_\theta(\EE)\neq0.$ We define 
    \[
    (\EE_{\det})_\rho:=\left\{\begin{array}{ll}
    \EE_\rho     & \ \mathrm{if} \ \dim(\rho) \ \mathrm{is \ even,}  \\
    -\EE_\rho     & \ \mathrm{otherwise,} 
    \end{array}\right. \  \ \ \EE_{\det}:=\cup_\rho(\EE_{\det})_\rho,\]
    where, if $\EE_\rho=\{([A_i,B_i]_\rho,l_i,\eta_i)\}_{i\in( I_\rho, >)},$ then $-\EE_\rho:=\{([A_i,B_i]_{\rho},l_i,-\eta_i)\}_{i\in (I_\rho, >)}.$
    Then \[\pi\otimes\det=\pi_\theta(\EE_{\det}).\]
\end{cor}

\begin{proof}
First, observe that $-R_k(\EE_{\rho})= R_k(-\EE_{\rho})$. Therefore, we may assume that $\EE\in \Rep^{(P')}$ and write $\EE=\EE_{\MM}$ for some M{\oe}glin parameter $\MM\in\Rep_\Moe^{(P')}$. By Theorem \ref{cor Moeglin to Atobe dictionary orthog},
    \[
    \pi_\theta(\EE)=\pi_\Moe(\MM).
    \]
    Let $\MM=(\psi_{>}, \underline{\zeta},\underline{l}, \underline{\eta})$ and $\widetilde{\MM}=(\psi_{>}, \underline{\zeta},\underline{l}, \underline{\tilde{\eta}})$ where 
    \[
    \tilde{\eta}(\rho_i,a_i,b_i)=\left\{\begin{array}{cl}
        -{\eta}(\rho_i,a_i,b_i), & \mathrm{if} \ \dim(\rho_i\otimes S_{|a_i-b_i|+1}) \ \mathrm{is \ odd;} \\
        {\eta}(\rho_i,a_i,b_i), & \mathrm{if} \ \dim(\rho_i\otimes S_{|a_i-b_i|+1}) \ \mathrm{is \ even.}
    \end{array}\right.
    \]
    By Remark \ref{rmk det}, we have $\pi_\Moe(\widetilde{\MM})=\pi_\Moe(\MM)\otimes\det.$ 
    
    Now we observe that if $\dim(\rho_i)$ is odd, the good parity condition implies that $a_i+b_i$ is even. Hence, $|a_i-b_i|+1$ is odd. We conclude that $\dim(\rho_i)$ is odd if and only if  $\dim(\rho_i\otimes S_{|a_i-b_i|+1})$ is odd. Then according to the recipe for $\MM \mapsto \EE_{\MM}$ in Definition \ref{defn Moeglin to Atobe dictionary}, we have 
    \[ \EE_{\widetilde{\MM}}= \EE_{\det}.\]
    This completes the proof of the corollary.
\end{proof}

We now show the analogue of Theorem \ref{thm Sp}(2). Namely, extended multi-segments parameterize local Arthur packets of $G_n=\OO(W_n).$

\begin{thm}\label{thm Atobe reformulation orthog}
    Suppose $\psi= \bigoplus_{\rho} \bigoplus_{i \in I_{\rho}} \rho \otimes S_{a_i} \otimes S_{b_i}$ is a local Arthur parameter of good parity of $G_n=\OO(W_n)$. Choose an admissible order $>_{\psi}$ on $I_{\rho}$ for each $\rho$ that satisfies ($P'$) if $\half{a_i-b_i}<0$ for some $i \in I_{\rho}$. Then
\begin{align}\label{eq Atobe reformuation orthog}
    \bigoplus_{\pi \in \Pi_{\psi}} \pi= \bigoplus_{\EE} \pi_\theta(\EE),
\end{align}
where $\EE$ runs over all extended multi-segments of $G_n$ with $\supp(\EE)= \supp(\psi)$ (for the fixed admissible order $>_{\psi}$) and $\pi_\theta(\EE) \neq 0$. 
\end{thm}

\begin{proof}
 Theorems \ref{thm Moeglin construction}(2) and \ref{cor Moeglin to Atobe dictionary orthog} give 
\[  \bigoplus_{\pi \in \Pi_{\psi}} \pi= \bigoplus_{\MM} \pi_{\Moe}(\MM)= \bigoplus_{\MM} \pi_{\theta}(\EE_{\MM}),\]
where the direct sum is taken over all M{\oe}glin parameters of the form $\MM=(\psi_{>},  \underline{\zeta},\underline{l},\underline{\eta})$ with $\psi_{>}$ fixed. Since $\psi_{\EE_{\MM}}= \psi$, we see that the left hand side of \eqref{eq Atobe reformuation orthog} is a subrepresentation of the right hand side. On the other hand, Proposition \ref{prop orthog extended multi-segment rep} already implies that $\pi_{\theta}(\EE)\in \Pi_{\psi}$. Since the left hand side is multiplicity free (\cite{Moe11}), it remains to show that if $\supp(\EE_1)=\supp(\EE_2)$ and $\pi_{\theta}(\EE_1)= \pi_{\theta}(\EE_2) \neq 0$, then $\EE_1=\EE_2$. By definition, this means $\pi_\Ato((\EE_1)_{\alpha})= \pi_\Ato((\EE_2)_{\alpha}) \neq 0$ for $\alpha \gg 0$. Then $(\EE_1)_{\alpha}=(\EE_2)_{\alpha}$ by Theorem \ref{thm Sp}(2) (which follows from the multiplicity freeness of $\Pi_{\psi_{\alpha}}$); hence $\EE_1=\EE_2$. This completes the proof of the theorem.
\end{proof}

Our next goal is to establish a combinatorial condition to verify when $\pi_\theta(\EE)\neq 0.$ Specifically, we verify an analogue of Theorem \ref{thm Sp}(1).

\begin{thm}\label{thm non-vanishing}
Let $\EE$ be an extended multi-segment of $G_n=\OO(W_n)$. Then $\pi_\theta(\EE)\neq 0$ if and only if $\NV(\EE)\neq 0.$

\end{thm}

\begin{proof}
By definition, $\pi_{\theta}(\EE) \neq 0$ if and only if $\pi_{\Ato}(\EE_{\alpha})\neq 0$ for $\alpha \gg 0$. Then it is clear from Definition \ref{def non-vanishing} that $\NV(\EE)\neq 0$ if and only if $\NV(\EE_{\alpha}) \neq 0$. Now the desired result follows from the symplectic case (Theorem \ref{thm Sp}(1)).
\end{proof}

Next we show the analogue of Theorem \ref{thm Sp}(3), i.e, $\pi_\theta(dual(\EE))$ is the Aubert-Zelevinsky dual of $\pi_\theta(\EE).$

\begin{prop}\label{prop dual is Aubert dual}
    Let $\EE$ be an extended multi-segment of $G_n=\OO(W_n)$ satisfying $(P')$. If $\pi_\theta(\EE)\neq 0,$ then $\pi_\theta(dual(\EE))=\AZ(\pi_\theta(\EE)).$
\end{prop}

\begin{proof}
By Remark \ref{rmk Moe Ato dictionary}, we take $\MM=((\psi_{\EE})_>,\underline{\zeta},\underline{l},\underline{\eta}) \in \Moe^{(P')}$ such that $\EE= \EE_{\MM}$. Then by Theorems \ref{cor Moeglin to Atobe dictionary orthog} and \ref{thm Moeglin construction}(3), we have 
\[ \AZ(\pi_\theta(\EE) )= \AZ(\pi_\Moe(\MM) )= \pi_\Moe(-\MM), \]
where $-\MM= ((\widehat{\psi}_{\EE})_>,-\underline{\zeta},\underline{l},\underline{\eta})$. Note that here $-\MM$ does not satisfy $(P')$ in general. In this proof, we use \cite[Theorem 6.3]{Xu21b} to compute $-\MM^{(P')}\in \Moe^{(P')}$ with the same local Arthur parameter as $-\MM$, and $\pi_{\Moe}(-\MM^{(P')})= \pi_{\Moe}(-\MM)$. Then we check that
\[ dual(\EE)= \EE_{ -\MM^{(P')}},\]
which will prove the proposition. 

First, we give more notation. For $(\rho, a_i, b_i)$ in $\Jord(\psi)$, we write 
\[ (\zeta_i, l_i, \eta_i):= (\underline{\zeta}(\rho, a_i, b_i),\underline{l}(\rho, a_i, b_i), \underline{\eta}(\rho, a_i, b_i))\]
for short. We write $ A_i= \half{a_i+b_i}-1$, $B_i= \half{a_i-b_i}$ throughout the proof.

Let $(I_{\rho},>)$ (resp. $(\widehat{I_{\rho}},>)$) be the index set for the $\rho$-part of $\MM$ (resp. $-\MM$), and let 
\[ I_\rho^\pm=\{i\in I_\rho \ | \ \zeta_i=\pm1\},\ \widehat{I_\rho}^\pm=\{i\in \widehat{I_\rho} \ | \ -\zeta_i=\pm1\}.\]
Thus one may identify $I_{\rho}^{\pm}= \widehat{I_{\rho}}^{\mp}$, but we write $\widehat{I_{\rho}}^{\pm}$ when we emphasize them as index sets of $-\MM$.
By Definition \ref{defn Moeglin to Atobe dictionary}, we write 
    \[
    \EE:= \EE_{\MM}=\cup_\rho\{([A_i,B_i]_{\rho},\hat{l}_i,\hat{\eta}_i)\}_{i\in (I_\rho^-,>')}\cup \{([A_i,B_i]_{\rho},l_i,\eta_i)\}_{i\in (I_\rho^+,>)}.
    \]
    Explicitly, for $i\in I_\rho^-,$ we have
    \begin{align*}
\hat{l}_i=\begin{cases}
l_i-B_i  & \mathrm{if} \, B_i\in\mathbb{Z},\\
 l_i-B_i+\frac{1}{2}(-1)^{\widehat{\alpha}_{i}}\eta_i  & \mathrm{if} \, B_i\not\in\mathbb{Z},
\end{cases}
\ \ \ \ 
\hat{\eta}_i=\begin{cases}
(-1)^{\widehat{\alpha}_i+\widehat{\beta}_i}\eta_i  & \mathrm{if} \, B_i\in\mathbb{Z},\\
 (-1)^{\widehat{\alpha}_i+\widehat{\beta}_i+1}\eta_i  & \mathrm{if} \, B_i\not\in\mathbb{Z},
\end{cases}
\end{align*}
where $\widehat{\alpha}_{i}=\sum_{j\in I_\rho^-, j<i}b_j,$ $\widehat{\beta}_{i}=\sum_{j\in I_\rho^-, j>i}a_j.$

Next, we compute $-\MM^{(P')}$. As mentioned above, $-\MM$ does not satisfy $(P')$ if both $I_{\rho}^+$ and $I_{\rho}^{-}$ are non-empty. Indeed, for any $i\in I_\rho^-$ and $j\in I_\rho^+$ we have that $i<j$ since the admissible order of $\MM$ satisfies $(P')$. However, then we have $-\zeta_i=1$, $-\zeta_j=-1$ but $i<j$ in $-\MM$, which violates $(P')$. Therefore, we deform to the admissible order $>_R$ satisfying $(P')$ on $\widehat{I_\rho}$ defined by the following. For any $i\in \widehat{I_\rho}^-$ and $j\in \widehat{I_\rho}^+$, we have $i<_R j$. Furthermore, the order $>_R$ agrees with $>$ on $\widehat{I_\rho}^\pm.$ That is, if $-\zeta_i=-\zeta_j$, then $i<_R j$ if and only if $i<j$. By \cite[Theorem 6.3]{Xu21b} (sequentially row exchanging each element of $\widehat{I_\rho}^-$ with the entirety of $\widehat{I_\rho}^+$), we obtain that 
    \[
    -\MM^{(P')}= ((\widehat{\psi}_{\EE})_{>_R},\underline{\zeta}_R,\underline{l}_R,\underline{\eta}_R),
    \]
    with $\pi_\Moe(-\MM^{(P')})=\pi_\Moe(-\MM),$
    where
    \begin{align*}
        \underline{\zeta}_R(\rho,b_i,a_i):=\zeta_{R,i}=-\zeta_i, \ 
        \underline{l}_R(\rho,b_i,a_i):=l_{R,i}=l_i,
    \end{align*}
    for any $i \in I_\rho$ and $\underline{\eta}_R(\rho,b_i,a_i):=\eta_{R,i}=(-1)^{\alpha^{\pm}}\eta_i$ for $i\in  \widehat{I_\rho}^\pm.$
    Here, 
    \[
    \alpha^{\pm}=\sum_{i\in \widehat{I_\rho}^\mp} (A_i-|B_i|+1).
    \]
    Note that $\alpha^+=\sum_{i\in \widehat{I_\rho}^-} b_i=\sum_{i\in I_\rho^+} b_i$ and $\alpha^-=\sum_{i\in \widehat{I_\rho}^+} a_i=\sum_{i\in I_\rho^-} a_i$. 

 We define
    \[
    \widehat{\EE}:=\EE_{-\MM^{(P')}}=\cup_\rho\{([A_i,-B_i]_\rho,\hat{l}_{R,i},\hat{\eta}_{R,i})\}_{i\in (\widehat{I_\rho}^-,>_R')}\cup \{([A_i,-B_i]_\rho,l_{R,i},\eta_{R,i})\}_{i\in (\widehat{I_\rho}^+,>_R)}.
    \]
    Here 
    \begin{align*}
\hat{l}_{R,i}=\begin{cases}
l_{R,i}+B_i  & \mathrm{if} \, B_{i}\in\mathbb{Z},\\
 l_{R,i}+B_i+\frac{1}{2}(-1)^{\alpha_{i}'}\eta_{R,i}  & \mathrm{if} \, B_i\not\in\mathbb{Z},
\end{cases}
\ \ \ \ 
\hat{\eta}_{R,i}=\begin{cases}
(-1)^{\alpha_i'+\beta_i'}\eta_{R,i}  & \mathrm{if} \, B_i\in\mathbb{Z},\\
 (-1)^{\alpha_i'+\beta_i'+1}\eta_{R,i}  & \mathrm{if} \, B_i\not\in\mathbb{Z},
\end{cases}
\end{align*}
where
\begin{align*}
    \alpha_i'= \sum_{j \in \widehat{I_{\rho}}^-,\   j <_R i} a_j=  \sum_{j \in I_{\rho}^+,  \ j < i} a_j,\ \ \beta_i'= \sum_{j \in \widehat{I_{\rho}}^-,  \ j >_R i} b_j = \sum_{j \in I_{\rho}^+,  \ j > i} b_j. 
\end{align*}
On the other hand, write 
\[dual(\EE)= \cup_{\rho} \{ ([A_i,-B_i]_{\rho}, l_i', \eta_i')\}_{i \in (\widehat{I_{\rho}}^{-}, >_{R}')} \cup \{ ([A_i,-B_i]_{\rho}, l_i', \eta_i')\}_{i \in (\widehat{I_{\rho}}^{+}, >_{R})}. \]
Here
\begin{align*}
l_i'=\begin{cases}
\hat{l}_i  + B_i  & \mathrm{if} \, B_{i}\in\mathbb{Z},\\
 \hat{l}_i+ B_i+\frac{1}{2}(-1)^{\alpha_{i}}\hat{\eta}_{i}  & \mathrm{if} \, B_i\not\in\mathbb{Z},
\end{cases}
\ \ \ \ 
\eta_i'=\begin{cases}
(-1)^{\alpha_i+\beta_i}\hat{\eta}_i  & \mathrm{if} \, B_i\in\mathbb{Z},\\
 (-1)^{\alpha_i+\beta_i+1}\hat{\eta}_i & \mathrm{if} \, B_i\not\in\mathbb{Z},
\end{cases}
\end{align*}
where
\[ \alpha_i= \begin{cases}
    \sum_{j \in I_{\rho}^-,\ j<' i }a_j & \text{ if }i \in I_{\rho}^-,\\
  \alpha^- + \sum_{j \in I_{\rho}^+,\ j< i }a_j & \text{ if }i \in I_{\rho}^+,
\end{cases}\ \   \beta_i= \begin{cases}
   \alpha^+ + \sum_{j \in I_{\rho}^-,\ j>'i }b_j& \text{ if }i \in I_{\rho}^-,\\
     \sum_{j \in I_{\rho}^+,\ j> i }b_j & \text{ if }i \in I_{\rho}^+.
\end{cases} \]
Now we check that $\widehat{\EE}=dual(\EE)$.

First, we match the block for $\widehat{I_{\rho}}^-$. If $i \in \widehat{I_{\rho}}^- = I_{\rho}^+$, then $\alpha_i'= \alpha_i- \alpha^-$ and $\beta_i'= \beta_i$. Therefore, 
\begin{align}\label{eq dual matching pi theta 1}
    (-1)^{\alpha_{i}'}\eta_{R,i}=(-1)^{\alpha_i'+ \alpha^-} \eta_i= (-1)^{\alpha_i} \eta_i, 
\end{align}
and similarly $(-1)^{\alpha_i'+\beta_i'} \eta_{R,i}= (-1)^{\alpha_i+\beta_i} \eta_i$. We conclude that 
\[\{ ([A_i,-B_i]_{\rho}, l_i', \eta_i')\}_{i \in (\widehat{I_{\rho}}^{-}, >_{R}')} = \{([A_i,-B_i],\hat{l}_{R,i},\hat{\eta}_{R,i})\}_{i\in (\widehat{I_\rho}^-,>_R')}.\]

Next, we compare the rest from the formula for $l_i', \eta_i'$ for $i \in \widehat{I_{\rho}}^+= I_{\rho}^-$. In this case, we have (note that for $i,j \in I_{\rho}^-$, $i>j$ if and only if $i<' j$)
\[ \alpha_i= \widehat{\beta}_i,\ \ \beta_i= \widehat{\alpha}_i+ \alpha^+.\]
First, suppose that $B_i \in \Z$. Then $l_i' =l_i= l_{R,i} $ and 
\[ \eta_{i}'= (-1)^{\alpha_i+ \beta_i} \hat{\eta}_i= (-1)^{\alpha_i+ \widehat{\alpha}_i+ \beta_i +\widehat{\beta}_i} \eta_i= (-1)^{\alpha^+} \eta_i= \eta_{R,i}. \]
Next, suppose that $B_i \not\in \Z$. Then by \eqref{eq dual matching pi theta 1},
\[(-1)^{\widehat{\alpha}_i} \eta_i + (-1)^{\alpha_i} \hat{\eta}_i=  (-1)^{\widehat{\alpha}_i} \eta_i + (-1)^{\alpha_i+\widehat{\alpha}_i+ \widehat{\beta}_i +1} \eta_i=0, \]
which implies that  $l_i'=l_{R,i}= l_i$. Finally, again by \eqref{eq dual matching pi theta 1},
\[ \eta_{i}'= (-1)^{\alpha_i+\beta_i+ 1 }\hat{\eta}_i= (-1)^{\alpha_i+ \widehat{\alpha}_i+ \beta_i +\widehat{\beta}_i} \eta_i= (-1)^{\alpha^+} \eta_i= \eta_{R,i}.\]
This completes the verification that 
\[\{ ([A_i,-B_i]_{\rho}, l_i', \eta_i')\}_{i \in (\widehat{I_{\rho}}^{+}, >_{R})} = \{([A_i,-B_i]_{\rho},\hat{l}_{R,i},\hat{\eta}_{R,i})\}_{i\in (\widehat{I_\rho}^+,>_R)},\]
and the proof of the proposition.
\end{proof}

\subsection{Intersections of Arthur packets} In this section, we prove the raising operators and their inverses describe all intersections of local Arthur packets for $\OO(W_n)$, as in the case of split symplectic groups and odd special orthogonal groups (\cite[Theorem 1.4]{HLL22}; see Theorem \ref{thm Sp}(5)).

We begin with a lemma describing a bijection between raising operators on $\EE$ and $\EE_\alpha.$

\begin{lemma}\label{lemma bijection between raising operators}
Let $\EE$ be an extended multi-segment of $G_n=\OO(W_n)$ with $\pi_\theta(\EE)\neq 0$. Then there is a bijection between the set of raising operators $T$ which are applicable on $\EE$ and the set of raising operators $T'$ which are applicable on $\EE_\alpha$ for $\alpha\gg0$ such that $T(\EE)_{\alpha}= T'(\EE_{\alpha})$.
\end{lemma}

\begin{proof}
    We describe the map from raising operators applicable on $\EE$ to those of $\EE_{\alpha}$. Write $I_{\rho}(\EE)$ (resp. $I_{\rho}(\EE_{\alpha})$) for the index set of $\EE_\rho$ (resp. $(\EE_{\alpha})_\rho$). If $\rho\not\cong \chi_V$, then there is a bijection between $I_{\rho}(\EE)$ and $I_{\chi_V^{-1} \chi_W \rho}(\EE_{\alpha})$. If $\rho \cong \chi_V$, then we identify
    \[ I_{\chi_W}(\EE_{\alpha})= \{0\} \sqcup I_{\chi_V}(\EE),\]
    where $0$ corresponds to the extended segment we inserted in Definition \ref{defn EE_alpha}. Now if $T= ui_{i,j}$ (resp. $dual \circ ui_{i,j} \circ dual$, $dual_i^{\pm}$) is applicable on $\EE$, where $i,j \in I_{\rho}(\EE)$, then it is clear from the definition that $T'=ui_{i,j}$ (resp. $dual \circ ui_{i,j} \circ dual$, $dual_i^{\pm}$), where now we regard $i,j \in  I_{\chi_V^{-1} \chi_W\rho}(\EE_{\alpha})$, is applicable on $\EE_{\alpha}$ as well.

    To show that the above map is a bijection, it suffices to show that if $T'= ui_{i,j}$ (resp. $dual \circ ui_{i,j} \circ dual$, $dual_i^{\pm}$) with $i,j \in I_{\chi_W}(\EE_{\alpha})$, then $i,j \neq 0$. However, this is again immediate from the definition of these operators and $\alpha \gg 0$. This completes the proof of the lemma.
\end{proof}

Next, we show that raising operators and their inverses completely determine intersections of local Arthur packets for $\OO(W_n)$.

\begin{thm}\label{thm O}
    Let $\EE$ and $\EE'$ be extended multi-segments of $G_n=\OO(W_n)$ with $\pi_\theta(\EE)\neq0.$ Then $\pi_\theta(\EE)=\pi_\theta(\EE')$ if and only if there exists a finite sequence consisting of raising  operators or their inverses  $(T_i)_{i=1}^l$ such that \[[\EE]=[(T_l\circ\dots\circ T_1)(\EE')].\] 
\end{thm}

\begin{proof}
   Fix an $\alpha \gg 0$. Suppose that $T$ is a raising operator or its inverse applicable on an extended multi-segment $\FF$. We write $T'$ for the corresponding operator applicable on $\FF_{\alpha}$ as in Lemma \ref{lemma bijection between raising operators}. Suppose that
\[[\EE]=[(T_l\circ\dots\circ T_1)(\EE')],\]
where each $T_i$ is a raising operator or its inverse. Then 
\[[\EE_{\alpha}]=[(T_l'\circ\dots\circ T_1')(\EE'_{\alpha})].\]
Now $\pi_{\Ato}(\EE_{\alpha})= \pi_{\Ato}(\EE'_{\alpha})$ by Theorem \ref{thm Sp}(5).
Thus, Howe duality (Theorem \ref{thm Howe duality}) implies that $\pi_{\theta}(\EE)= \pi_{\theta}(\EE')$. This shows the backward direction of the theorem.

    Conversely, assume that $\pi_\theta(\EE)=\pi_\theta(\EE')$.  Then $\theta_{-\alpha}(\pi_\theta(\EE))=\pi_\Ato(\EE_{\alpha})=\pi_\Ato(\EE_{\alpha}').$ The result for symplectic groups (Theorem \ref{thm Sp}(5)) then implies that there is a finite sequence consisting of raising operators or their inverses $\{T_i'\}_{i=1}^l$ such that \[[\EE_{\alpha}]=[(T_l'\circ\dots\circ T_1')(\EE_{\alpha}')].\]
    Then using the bijection in Lemma \ref{lemma bijection between raising operators} again, we obtain that
    \[
   [ \EE]=[((T_l\circ\dots\circ T_1)(\EE'))].
    \]
This completes the proof of the theorem.
\end{proof}

\begin{remark}
Let $\pi$ be an irreducible representation of $\OO(W_n)$ and let $\alpha\gg 0$. 
    The proof of the above theorem shows that the following map is a bijection.
    \begin{align*}
        \{[\EE]\ | \ \pi_{\theta}(\EE)=\pi\} &\to \{[\EE]\ | \ \pi_{\Ato}(\EE)=\theta_{-\alpha}(\pi)\},\\
        \EE& \mapsto \EE_{\alpha}.
    \end{align*}
\end{remark}

As an immediate consequence of Theorems \ref{thm E max}, \ref{thm Atobe reformulation orthog}, and \ref{thm O}, we obtain the analogue of Theorem \ref{thm Sp}(6) for even orthogonal groups. 

\begin{thm}\label{thm psi max and min}
    Let $\pi$ be a representation of $G_n=\OO(W_n)$ of Arthur type. Then there exists unique maximal and minimal elements of $\Psi(\pi)$ with respect to $\geq_O$. We denote these elements by $\psi^{\max}(\pi)$ and $\psi^{\min}(\pi)$, respectively.
\end{thm}

\section{\texorpdfstring{Computation of the $L$-data of $\pi_\theta(\EE)$}{}}\label{sec pi comb = pi theta}

In this section, we verify that $\pi_\theta(\EE)=\pi_\com(\EE)$ for an extended multi-segment $\EE$ of $G_n=\OO(W_n)$ (Theorem \ref{thm theta=com}). This allows for an algorithm which computes the $L$-data of $\pi_\theta(\EE)$ explicitly while avoiding computations of the  theta correspondence. In particular, the computation of $\pi_\com(\EE)$ is intrinsic to the group $G_n.$ Another advantage of the algorithm to compute $\pi_\com(\EE)$ is that it allows for the use of Algorithm \ref{algo Arthur type pi^-} to determine when a representation $\pi\in\Pi_{gp}(G_n)$ is of Arthur type.

\subsection{M{\oe}glin's operator}\label{subsec Moeglin operators}

In this subsection, we recall a specific operator of M{\oe}glin (\cite[\S4.1]{Moe08}), denoted by $(add_j^\Moe)^{-1}$ below, that applies to M{\oe}glin parameters. Suppose that $\pi=\pi_\Moe(\psi_>,\underline{\zeta},\underline{l},\underline{\eta})=\pi_{\Moe}(\MM)\in\Pi_\psi$ for some local Arthur parameter $\psi$ of good parity. Write
\[
\psi=\bigoplus_{i=1}^r \rho_i\otimes S_{a_i}\otimes S_{b_i}.
\]
Fix $j\in\{1,\dots,r\}$ such that $b_j\geq 2.$ We define
\[
\psi^-=\bigoplus_{i=1}^r \rho_i\otimes S_{a_i}\otimes S_{b_i'},
\]
where $b_i'=b_i$ if $i\neq j,$ and $b_j'=b_j-2$. If $b_j'=0,$ then we omit that summand. We define a new representation (possibly 0) via
\[
\sigma=\pi_\Moe((\psi^-)_>,\underline{\zeta}^-,\underline{l}^-,\underline{\eta}^-)=:\pi_\Moe((add_j^\Moe)^{-1}(\MM)),
\]
where if $i\neq j,$ we set
\begin{align*}
    \underline{\zeta}^-(\rho_i,a_i,b_i'):=\underline{\zeta}(\rho_i,a_i,b_i),\ \ 
    \underline{l}^-(\rho_i,a_i,b_i'):=\underline{l}(\rho_i,a_i,b_i),\ \ 
    \underline{\eta}^-(\rho_i,a_i,b_i'):=\underline{\eta}(\rho_i,a_i,b_i).
\end{align*}
If $i=j,$ we define $\underline{\zeta}^-,\underline{l}^-,\underline{\eta}^-$ based on the following cases. We note that $\underline{\zeta}(\rho_j,a_j, b_j')$ is determined by the sign of the difference $a_j-b_j'$, except possibly in the first case.
\begin{itemize}
    \item \textbf{Case 1:} Suppose $a_j\leq b_j-2$. In this case, we set $\underline{\eta}^-(\rho_j,a_j,b_j')=\underline{\eta}(\rho_j,a_j,b_j)$ and $\underline{l}^-(\rho_j,a_j,b_j')=\underline{l}(\rho_j,a_j,b_j).$ When $a_j=b_j-2,$ we also set $\underline{\zeta}^-(\rho_j,a_j, b_j')=1.$
    \item \textbf{Case 2:} Suppose $a_j\geq b_j$. In this case, we set $\underline{\eta}^-(\rho_j,a_j,b_j')=\underline{\eta}(\rho_j,a_j,b_j)$ and $\underline{l}^-(\rho_j,a_j,b_j')=\underline{l}(\rho_j,a_j,b_j)-1.$ 
    \item \textbf{Case 3:} Suppose $a_j=b_j-1$ and $\underline{\eta}(\rho_j,a_j,b_j)=+1$. In this case, we set $\underline{\eta}^-(\rho_j,a_j,b_j')=-\underline{\eta}(\rho_j,a_j,b_j)$ and $\underline{l}^-(\rho_j,a_j,b_j')=\underline{l}(\rho_j,a_j,b_j)$, {except if $a_j$ is even and $\underline{l}(\rho_j,a_j,b_j)=\frac{a_j}{2},$ in which case we set $\underline{\eta}^-(\rho_j,a_j,b_j')=1$ and $\underline{l}^-(\rho_j,a_j,b_j')=\underline{l}(\rho_j,a_j,b_j)-1$.}  
\end{itemize}
Therefore, the cases where the $\underline{\zeta}^-,\underline{l}^-,\underline{\eta}^-$ are not defined are:
\begin{itemize}
    \item \textbf{ Bad Case 1:} $a_j \geq b_j$ and $\underline{l}(\rho_j,a_j,b_j)=0$.
    \item \textbf{ Bad Case 2:} $a_j=b_j-1$ and $\underline{\eta}(\rho_j,a_j,b_j)=-1$.
\end{itemize}

\begin{remark}
    In \cite[\S4.1]{Moe08}, M{\oe}glin did not specify $\underline{l}^-(\rho_j,a_j,b_j')$ for the exceptional case in Case 3. However, the inequality $\underline{l}^-(\rho_j,a_j,b_j') \leq \half{\min (a_j, b_j')}$ together with the requirement $\underline{l}^-(\rho_j,a_j,b_j')-\underline{l}(\rho_j,a_j,b_j) \in \{-1,0\}$ implies that $\underline{l}^-(\rho_j,a_j,b_j')=\underline{l}(\rho_j,a_j,b_j)-1$. 
\end{remark}

M{\oe}glin relates $\pi_\Moe(\MM)$ and $\pi_\Moe((add_j^{\Moe})^{-1}(\MM))$ when they are both nonzero.

\begin{thm}[{\cite[Theorem A]{Moe11b}}]\label{thm Moeglin + embed}
    Suppose that $\pi=\pi_\Moe(\psi_>,\underline{\zeta},\underline{l},\underline{\eta})=\pi_{\Moe}(\MM)\neq 0$. Fix $j$ such that $\pi^-:=\pi_\Moe((add_j^{\Moe})^{-1}(\MM))$ is well-defined and nonzero. Then
    \[
    \pi\hookrightarrow \Delta_{\rho_j} [B_j,-A_j] \rtimes \pi^-,\]
    where $A_j= \half{a_j+b_j}-1, B_j=\half{a_j-b_j}$.
\end{thm}

\subsection{\texorpdfstring{Verification of $\pi_{\theta}(\EE)=\pi_{\com}(\EE)$}{}}

In this section, we show that $\pi_\theta(\EE)=\pi_\com(\EE)$ (Theorem \ref{thm theta=com}) for $\EE\in\Rep.$ Consequently, we obtain an intrinsic combinatorial algorithm to compute $\pi_\theta(\EE)$ (Algorithm \ref{algo pi_com}).

We identify the operator $add^{-1}$ (Definition \ref{def shift and add})  with the operator defined by M{\oe}glin (see \S\ref{subsec Moeglin operators}).

\begin{thm}\label{thm add = +} Let $\MM\in\Rep_\Moe^{(P')}(\OO(W_n))$ and assume that $(add_{j}^{\Moe})^{-1}$ is well-defined on $\MM$. Furthermore, we assume that the order $>$ also satisfies that 
{
    \begin{itemize}
        \item[(i)] if $a_j\leq b_j-1$, then for any $i\in I_{\rho_j}^-$ with $|B_j| = |B_i|,$ we have $j\leq i$; and
        \item[(ii)] if $a_j\geq b_j,$ then for any $i\in I_{\rho_j}^+$ with $B_j = B_i,$ we have $j\geq i.$
    \end{itemize}
    In addition, we require that if $i\in I_\rho$ and $a_i=b_i$, then $\underline{\zeta}(\rho,a_i,b_i)=+1$.
    }
    By Definition \ref{defn Moeglin to Atobe dictionary}, write
    \[
    \pi=\pi_\Moe(\psi_>,\underline{\zeta},\underline{l},\underline{\eta})=\pi_\Moe(\MM)=\pi_\theta(\EE_\MM).
    \]
    Then $\pi_{\Moe}((add_{j}^{\Moe})^{-1}(\MM))=\pi_\theta(add_j^{-1}(\EE_\MM))$ (possibly vanishing).
\end{thm}

\begin{proof}
Let $\MM^-:=(add_{j}^{\Moe})^{-1}(\MM)$. It is possible that $\MM^-$ may not satisfy $(P')$ (Definition \ref{def Moeglin parameter}(4)), in which case we fix a specific admissible order $>_R$ below, and compute a M{\oe}glin parameter $\MM^-_R$ that satisfies $(P')$ and $\pi_{\Moe}(\MM^-)= \pi_{\Moe}(\MM^-_{R})$ via \cite[Theorem 6.3]{Xu21b}. By Theorem \ref{cor Moeglin to Atobe dictionary orthog}, it suffices to show that 
\[add_j^{-1}(\EE_\MM)= \EE_{\MM^-_{R}}.\]

We give more notations. Write $\MM=(\psi_>,\underline{\zeta},\underline{l},\underline{\eta}).$
We recall the definition of 
\[\MM^-=((\psi^-)_>,\underline{\zeta}^-,\underline{l}^-,\underline{\eta}^-).\] 
If $i \neq j,$ we have
\begin{align*}
    \underline{\zeta}^-(\rho_i,a_i,b_i')=\underline{\zeta}(\rho_i,a_i,b_i),\ \ 
    \underline{l}^-(\rho_i,a_i,b_i')=\underline{l}(\rho_i,a_i,b_i),\ \ 
    \underline{\eta}^-(\rho_i,a_i,b_i')=\underline{\eta}(\rho_i,a_i,b_i).
\end{align*}
For the index $j,$ we define $\underline{\zeta}^-,\underline{l}^-,\underline{\eta}^-$ based on the following cases.
\begin{itemize}
    \item \textbf{Case 1:} Suppose $a_j\leq b_j-2$. In this case, we have $\underline{\eta}^-(\rho_j,a_j,b_j')=\underline{\eta}(\rho_j,a_j,b_j)$ and $\underline{l}^-(\rho_j,a_j,b_j')=\underline{l}(\rho_j,a_j,b_j).$
    \item \textbf{Case 2:} Suppose $a_j\geq b_j$. In this case, we have $\underline{\eta}^-(\rho_j,a_j,b_j')=\underline{\eta}(\rho_j,a_j,b_j)$ and $\underline{l}^-(\rho_j,a_j,b_j')=\underline{l}(\rho_j,a_j,b_j)-1.$
    \item \textbf{Case 3:} Suppose $a_j=b_j-1$ and $\underline{\eta}(\rho_j,a_j,b_j)=+1$. In this case, we have $\underline{\eta}^-(\rho_j,a_j,b_j')=-\underline{\eta}(\rho_j,a_j,b_j)$ and $\underline{l}^-(\rho_j,a_j,b_j')=\underline{l}(\rho_j,a_j,b_j)$,  {except if $a_j$ is even and $\underline{l}(\rho_j,a_j,b_j)=\frac{a_j}{2},$ in which case we set $\underline{\eta}^-(\rho_j,a_j,b_j')=1$ and $\underline{l}^-(\rho_j,a_j,b_j')=\underline{l}(\rho_j,a_j,b_j)-1$.}
\end{itemize}
The proof proceeds in the above cases by computing the extended multi-segment $\EE_{\MM_R^-}$ attached to $\MM_R^-$ using Definition \ref{defn Moeglin to Atobe dictionary}. Before we begin, we explicate the information in $\EE_\MM.$ 

Recall that we have assumed that the admissible order of $\MM$ satisfies $(P')$. We let $I_\rho^\pm=\{i\in I_\rho \ | \ \zeta_i=\pm1\}.$ We write
\[ \EE_{\MM}= \EE_{\MM,<0} + \EE_{\MM, \geq 0},\]
where
\begin{align*}
  \EE_{\MM,<0}=  \cup_\rho\{([A_i,B_i]_{\rho},\hat{l}_i,\hat{\eta}_i)\}_{i\in (I_\rho^-,>')},\ \ \EE_{\MM,\geq 0}=\cup_\rho \{([A_i,B_i]_{\rho},l_i,\eta_i)\}_{i\in (I_\rho^+,>)}.
\end{align*}
Here, $>'$ is the order on $I_\rho^-$ obtained by reversing the order $>.$  We recall the numbers $l_i, \eta_i$ and $\widehat{l}_i, \widehat{\eta}_i$ explicitly. For $i\in I_\rho^+,$ we have 
    \begin{align*}
        l_i= \underline{l}(\rho,a_i,b_i),\ \ \ 
        \eta_i=\underline{\eta}(\rho,a_i,b_i).
    \end{align*}
For $i\in I_\rho^-,$ we set  \begin{align*}
\hat{l}_i=\begin{cases}
\underline{l}(\rho,a_i,b_i)-B_i  & \mathrm{if} \, B_i\in\mathbb{Z},\\
 \underline{l}(\rho,a_i,b_i)-B_i+\frac{1}{2}(-1)^{\alpha_{i}}\underline{\eta}(\rho,a_i,b_i)  & \mathrm{if} \, B_i\not\in\mathbb{Z},
\end{cases}
\end{align*}
and
\begin{align*}
\hat{\eta}_i=\begin{cases}
(-1)^{\alpha_i+\beta_i}\underline{\eta}(\rho,a_i,b_i)  & \mathrm{if} \, B_i\in\mathbb{Z},\\
 (-1)^{\alpha_i+\beta_i+1}\underline{\eta}(\rho,a_i,b_i)  & \mathrm{if} \, B_i\not\in\mathbb{Z},
\end{cases}
\end{align*}
where 
\begin{align*}
    \alpha_{i}=\sum_{\substack{k\in I_\rho^-,\  k<i}}b_k, \ \ 
    \beta_{i}=\sum_{\substack{k\in I_\rho^- ,\  k>i}}a_k.
\end{align*}
 In this formula, if $B_i\not\in\mathbb{Z}$ and {$\underline{l}(\rho,a_i,b_i)=\frac{a_i}{2}$}, then we {replace $\underline{\eta}(\rho,a_i,b_i)$ in the above formulas by $(-1)^{\alpha_i+1}$.} In the rest of the proof, we write $\alpha_i^{-}, \beta_i^{-}, \hat{l}_i^-, \hat{\eta}_i^-$ for these ingredients for $\MM^-$ (or $\MM^-_{R}$ if $\MM^-$ does not satisfy $(P')$). We also write $\rho=\rho_j$ for short.

We now proceed in a case-by-case analysis of $\MM^-.$
\vspace{0.3 cm}

\noindent\textbf{Case 1.}  $a_j\leq b_j-2$. 
{We first assume that $a_j<b_j-2$.}
In this case, $\EE_{\MM, \geq 0}= \EE_{\MM^-, \geq 0}$, and we compare $\EE_{\MM, < 0}$ and $\EE_{\MM^-, < 0}$. Note that $\MM^-$ satisfies $(P')$ already by our assumption on the admissible order. Now the ingredients to define $(\hat{l}_i, \hat{\eta}_i)_{i \in (I_{\rho}^-, >')}$ for $\EE_{\MM,<0}$ and $\EE_{\MM^-,<0}$ are identical except that $\alpha_i^-= \alpha_i -2$ for $i>j$. However, this does not affect the formula. Note that 
\[ B_j':= \frac{a_j- b_j'}{2} =\frac{a_j-b_j}{2}+1= B_j+1.\]
As a consequence, we have 
\[ (\hat{l}_i^-, \hat{\eta}_i^-)= \begin{cases}
    (\hat{l}_i, \hat{\eta}_i) & \text{ if }i \neq j,\\
    (\hat{l}_j + B_j- B_j', \hat{\eta}_j)= (\hat{l}_j -1, \hat{\eta}_j) & \text{ if }i=j.
\end{cases}\]
Thus, we obtain 
\[ \EE_{\MM^-}= add_j^{-1}(\EE_{\MM, <0})+ \EE_{\MM, \geq 0}= add_j^{-1}(\EE_{\MM}).\]

{
Suppose now that $a_j=b_j-2$, i.e., $B_j=-1$ and $B_j'=0$. In this case, $\MM^-$ does not necessarily satisfy $(P')$ (unless $I_\rho^-=\{j\}$). Note that by our assumptions on $>$, we have that $j$ is minimal in $I_\rho$. We consider the following admissible order $>_{R}$ defined by
\begin{itemize}
    \item for any $i^-\in I_{\rho}^- \setminus \{j\}$ and $i^+ \in I_{\rho}^+ $, we have
\[ i^- <_R j <_R i^+,\]
\item and if $i_1, i_2 \in I_{\rho} \setminus \{j\}$, we have $i_1 <_R i_2$ if and only if $i_1< i_2$.
\end{itemize}
 Now we apply \cite[Theorem 6.3]{Xu21b} repeatedly to compute a M{\oe}glin parameter
\[ \MM_{R}^-= ((\psi^-)_{>_R},\underline{\zeta}^-_R,\underline{l}^-_R,\underline{\eta}^-_R) \]
such that $\pi_{\Moe}(\MM_R^-)= \pi_{\Moe}(\MM^-)$. The formula gives
\begin{align*}
    (\underline{l}^-_R(\rho, a_i,b_i'), \underline{\eta}^-_{R}(\rho, a_i,b_i'))= \begin{cases}
        (\underline{l}(\rho, a_i,b_i), \underline{\eta}(\rho, a_i,b_i)) & \text{ if }i \in I_{\rho}^+,\\
         (\underline{l}(\rho, a_i,b_i), (-1)^{b_j-2}\underline{\eta}(\rho, a_i,b_i)) & \text{ if }i \in I_{\rho}^- \setminus \{j\},\\
          (\underline{l}(\rho, a_j,b_j), (-1)^{\gamma}\underline{\eta}^-(\rho, a_j,b_j)) & \text{ if }i=j,
    \end{cases}
\end{align*}
where 
\[ \gamma= \sum_{i \in I_{\rho}^- \setminus \{j\}} A_i -|B_i|+1= \sum_{i \in I_{\rho}^- \setminus \{j\}} A_i+ B_i +1 = \sum_{i \in I_{\rho}^- \setminus \{j\}} a_i= \beta_j. \]

Now we compare $\EE_{\MM,<0}$ and $\EE_{\MM^-_{R},<0}$. For $i \in I_{\rho}^- \setminus \{j\}$, we have 
\[ \alpha^-_i= \alpha_i - b_j,\ \beta_i^-=\beta_i.\]
In particular, we have
\[ \half{1} (-1)^{\alpha_i^-} \underline{\eta}^-_{R}(\rho, a_i, b_i )=\half{1} (-1)^{\alpha_i} \underline{\eta}(\rho, a_i, b_i ),\]
and similarly $\hat{\eta}_i^-= \hat{\eta}_i$. Therefore, we have 
$(\hat{l}^-_i, \hat{\eta}^-_i)=(\hat{l}_i, \hat{\eta}_i)$ for all $i \in I_{\rho}^- \setminus \{j\}$.

Now, we focus on the $j$-th extended segment. In $\EE_{\MM^-_{R}}$, we have just computed that it is
\[ ([A_j-1, B_j+1]_{\rho}, l(\rho, a_j,b_j), (-1)^{\beta_j}\eta^-(\rho, a_j,b_j) ). \]
On the other hand, in $\EE_{\MM}$, it is 
\[ ([A_j, B_j]_{\rho},\hat{l}_j, \hat{\eta}_j   ),\]
where (write $l_j=\underline{l}(\rho, a_j,b_j)$, $\eta_j=\underline{\eta}(\rho, a_j,b_j)$ for short)
\begin{align*}
\hat{l}_j=l_{j}-B_j=l_j+1,
\end{align*}
since $B_j=-1$, and
\begin{align*}
    \hat{\eta}_j=(-1)^{\alpha_j+\beta_j}\eta_{j} 
=(-1)^{\beta_j}\eta_{j} 
=\underline{\eta}^-_{R}(\rho, a_j, b_j ).
\end{align*}
In conclusion, we have
\begin{align*}
    \EE_{\MM^-_{R}}= (\EE_{\MM_R^-,<0} + \{([A_j-1,B_j']_{\rho},\hat{l}_j-1, \hat{\eta}_j)\} )+ \EE_{\MM,\geq 0}= add_{j}^{-1}(\EE_{\MM,<0})+\EE_{\MM,\geq 0}= add_j^{-1}(\EE_{\MM}).
\end{align*}
This completes the verification of this case.
}

\vspace{0.3 cm}
\noindent\textbf{Case 2.} $a_j\geq b_j$. 

In this case, we still have $a_j \geq b_j'=b_j-2$, and $\MM^-$ still satisfies $(P')$ by the assumption on the admissible order. Therefore, 
\[ \EE_{\MM^-}=  \EE_{\MM,<0} + add_j^{-1}(\EE_{\MM, \geq 0}) = add_j^{-1}(\EE_{\MM}). \]

\vspace{0.3 cm}
\noindent\textbf{Case 3.} $a_j=b_j-1$ and $\underline{\eta}(\rho_j,a_j,b_j)=+1$.

This is the case that $\MM^-$ does not necessarily satisfy $(P')$. Note that $A_j \not\in \Z$ in this case. By our assumption, $j$ is the minimal element in $I_{\rho}^-$ (and $I_{\rho}$). We consider the following admissible order $>_{R}$ defined by
\begin{itemize}
    \item for any $i^-\in I_{\rho}^- \setminus \{j\}$ and $i^+ \in I_{\rho}^+ $, we have
\[ i^- <_R j <_R i^+,\]
\item and if $i_1, i_2 \in I_{\rho} \setminus \{j\}$, we have $i_1 <_R i_2$ if and only if $i_1< i_2$.
\end{itemize}
 Now we apply \cite[Theorem 6.3]{Xu21b} repeatedly to compute a M{\oe}glin parameter
\[ \MM_{R}^-= ((\psi^-)_{>_R},\underline{\zeta}^-_R,\underline{l}^-_R,\underline{\eta}^-_R) \]
such that $\pi_{\Moe}(\MM_R^-)= \pi_{\Moe}(\MM^-)$. The formula gives
\begin{align*}
    (\underline{l}^-_R(\rho, a_i,b_i'), \underline{\eta}^-_{R}(\rho, a_i,b_i'))= \begin{cases}
        (\underline{l}(\rho, a_i,b_i), \underline{\eta}(\rho, a_i,b_i)) & \text{ if }i \in I_{\rho}^+,\\
         (\underline{l}(\rho, a_i,b_i), (-1)^{b_j-2}\underline{\eta}(\rho, a_i,b_i)) & \text{ if }i \in I_{\rho}^- \setminus \{j\},\\
          (\underline{l}(\rho, a_j,b_j)-\delta, (-1)^{\gamma}\underline{\eta}^-(\rho, a_j,b_j')) & \text{ if }i=j,
    \end{cases}
\end{align*}
where $\delta=0$ if $\underline{l}(\rho,a_j,b_j)<\frac{a_j}{2}$ and $\delta=1$ otherwise, and
\[ \gamma= \sum_{i \in I_{\rho}^- \setminus \{j\}} A_i -|B_i|+1= \sum_{i \in I_{\rho}^- \setminus \{j\}} A_i+ B_i +1 = \sum_{i \in I_{\rho}^- \setminus \{j\}} a_i= \beta_j. \]

Now we compare $\EE_{\MM,<0}$ and $\EE_{\MM^-_{R},<0}$. For $i \in I_{\rho}^- \setminus \{j\}$, we have 
\[ \alpha^-_i= \alpha_i - b_j,\ \beta_i^-=\beta_i.\]
In particular, we have
\[ \half{1} (-1)^{\alpha_i^-} \underline{\eta}^-_{R}(\rho, a_i, b_i )=\half{1} (-1)^{\alpha_i} \underline{\eta}(\rho, a_i, b_i ),\]
and similarly $\hat{\eta}_i^-= \hat{\eta}_i$. Therefore, we have 
$(\hat{l}^-_i, \hat{\eta}^-_i)=(\hat{l}_i, \hat{\eta}_i)$ for all $i \in I_{\rho}^- \setminus \{j\}$.

Now, we focus on the $j$-th extended segment. In $\EE_{\MM^-_{R}}$, we have just computed that it is
\[ ([A_j-1, B_j+1]_{\rho}, \underline{l}(\rho, a_j,b_j)-\delta, (-1)^{\beta_j}\underline{\eta}^-(\rho, a_j,b_j') ). \]
On the other hand, in $\EE_{\MM}$, it is 
\[ ([A_j, B_j]_{\rho},\hat{l}_j, \hat{\eta}_j   ),\]
where  (write $l_j=\underline{l}(\rho, a_j,b_j)$, $\eta_j=\underline{\eta}(\rho, a_j,b_j)$ for short) if $l_j<\frac{a_j}{2}$
\begin{align*}
\hat{l}_j&=l_{j}-B_j+\frac{1}{2}(-1)^{\alpha_j}\eta_{j} 
= l_{j}-B_j+\frac{1}{2}\eta_{j}=  l_{j}+\frac{1}{2}+\frac{1}{2}=l_j+1, \\
\hat{\eta}_j&=(-1)^{\alpha_j+\beta_j+1}\eta_{j} 
=(-1)^{\beta_j+1}\eta_{j} 
=(-1)^{\beta_j}{\underline{\eta}}^-(\rho, a_j,b_j{'})
\end{align*}
since $B_j=-\half{1}$ and $\eta_j=1$ by the assumption of this case. If $l_j=\frac{a_j}{2}$, then the convention set in Definition \ref{dual segment}(3) is that $\eta_j$ is replaced by $(-1)^{\alpha_j+1}$. Therefore,
\begin{align*}
\hat{l}_j&=l_{j}-B_j+\frac{1}{2}(-1)^{\alpha_j}(-1)^{\alpha_j+1} 
= l_{j}-B_j-\frac{1}{2}=  l_{j}+\frac{1}{2}-\frac{1}{2}=l_j, \\
\hat{\eta}_j&=(-1)^{\alpha_j+\beta_j+1}(-1)^{\alpha_j+1} 
=(-1)^{\beta_j},
\end{align*}
where we recall that $\underline{\eta}^-(\rho,a_j,b_j')=1$ by definition in this case.
In conclusion, we have
\begin{align*}
    \EE_{\MM^-_{R}}= (\EE_{\MM_R^-,<0} + \{([A_j-1,B_j']_{\rho},\hat{l}_j-1, \hat{\eta}_j)\} )+ \EE_{\MM,\geq 0}= add_{j}^{-1}(\EE_{\MM,<0})+\EE_{\MM,\geq 0}= add_j^{-1}(\EE_{\MM}).
\end{align*}
This completes the verification of this case and the proof of the theorem.
\end{proof}

\begin{remark}
    Theorem \ref{thm add = +} also holds for symplectic groups (using $\pi_\Ato$ instead of $\pi_\theta$) with exactly the same proof, except we replace references to Theorem \ref{cor Moeglin to Atobe dictionary orthog} with its symplectic analogue, Theorem \ref{thm Sp}(4). Theorem \ref{thm add = +} also holds for split odd special orthogonal groups by using \cite[Theorem 6.6]{Ato22a}.
\end{remark}

As an application of Theorem \ref{thm add = +}, we obtain an extended multi-segment analogue of Theorem \ref{thm Moeglin + embed}.

\begin{thm}\label{thm add embed}
    Let $\EE\in \Rep^{(P')}(\OO(W_n))$. Write
\[
\EE= \cup_{\rho'} \{([A_i,B_i]_{\rho'}, l_i, \eta_i)\}_{i \in (I_{\rho'}, >)}.\]
Fix an index $j\in (I_{\rho},>)$ such that $b_j\geq 2$ and if $B_i=B_j$ for some $i\in I_\rho$, then $j\geq i$.
    Furthermore, we assume that $\pi_\theta(add_j^{-1}(\EE))\neq 0.$
Then
\[
\pi_\theta(\EE)\hookrightarrow \Delta_{\rho_j}[B_j,-A_j]\rtimes \pi_\theta(add_j^{-1}(\EE)).
\]
\end{thm}

\begin{proof}
    Since $\EE\in \Rep^{(P')}$, there exists an $\MM\in \Moe^{(P')}$ such that $\EE=\EE_\MM$ by Remark \ref{rmk Moe Ato dictionary}. Moreover, $\MM$ satisfies the condition on the admissible order of Theorem \ref{thm add = +} by our assumption on the index $j$. Therefore, we have 
    \[ \pi_{\theta}( add_j^{-1}(\EE))= \pi_{\Moe}((add_{j}^{\Moe})^{-1}(\MM)). \]
    Then Theorem \ref{thm Moeglin + embed} implies that 
    \[ \pi_{\theta}(\EE)=\pi_{\Moe}(\MM) \hookrightarrow \Delta_{\rho_j}[B_j,-A_j] \rtimes \pi_{\Moe}((add_{j}^{\Moe})^{-1}(\MM))=\Delta_{\rho_j}[B_j,-A_j] \rtimes \pi_{\theta}( add_j^{-1}(\EE)). \]
    This completes the proof of the theorem.
\end{proof}

Next, we give a sufficient condition for $\pi_\theta(add_j^{-1}(\EE))\neq 0.$

\begin{prop}\label{prop max triangle}
Consider an extended multi-segment $\EE \in \Rep$ of $G_n=\OO(W_n)$ with
\[\EE^{|max|}= \cup_{\rho'} \{([A_i,B_i]_{\rho'},l_i,\eta_i)\}_{i \in (I_{\rho'},>)} \in \Rep^{(P')}.\]
  Fix a $\rho$ such that $b_{\rho} := \max \{A_i-B_i+1\ | \ i \in I_{\rho}\} >1 $. Let $ j = \min\{ i \in I_{\rho}\ | \ A_i-B_i+1=b_{\rho} \}$ and assume that $j= \max\{ i \in I_{\rho} \ | \ B_i=B_j \}$ by applying row exchanges if necessary. Let $(\EE^{|max|})^{-}:=add_j^{-1}(\EE^{|max|})$, which satisfies $(P')$ by the assumptions. Then, \[\pi_\theta( (\EE^{|max|})^{-}) \neq 0.\]
\end{prop}

\begin{proof}
    This is a direct consequence of Theorems \ref{thm E max}(2) and \ref{thm non-vanishing}. 
\end{proof}

Comparing with Algorithm \ref{algo pi_com}, we obtain a combinatorial algorithm to compute $\pi_\theta(\EE)$ without using the  theta correspondence. 

\begin{thm}\label{thm theta=com}
Let $\EE\in\Rep$ be an extended multi-segment of $G_n=\OO(W_n).$ Then \[\pi_\theta(\EE)=\pi_\com(\EE).\]
\end{thm}

\begin{proof}
We apply induction on the number of Steinberg representations in the non-tempered portion of the $L$-data of $\pi_{\com}(\EE)$. First, suppose that $\pi_{\com}(\EE)$ is tempered. Take the tempered $\MM\in \Moe^{(P')}$ such that $\EE^{|max|}=\EE_{\MM}$. Then Theorems \ref{cor Moeglin to Atobe dictionary orthog}, \ref{thm O} imply that
\[ \pi_{\theta}(\EE)= \pi_{\theta}( \EE^{|max|})= \pi_{\Moe}(\MM).\]
From Equation \eqref{eqn Moe tempered}, we have $\pi_{\Moe}(\MM)= \pi_{\com}(\EE).$

Now suppose that $\pi_{\com}(\EE)$ is not tempered. In this case, $\psi_{\EE^{|max|}}$ is not tempered, and we let $(\EE^{|max|})^{-}$ be defined as in Proposition \ref{prop max triangle} above. By Remark \ref{rmk EE Sp L data}(1), we have an injection
\begin{align}\label{eq proof of pi com}
    \pi_{\com}(\EE) \hookrightarrow \Delta_{\rho}[B_j, -A_j] \rtimes \pi_{\com}((\EE^{|max|})^{-} ),
\end{align}
which is compatible with Langlands classification. In particular, the right hand side of \eqref{eq proof of pi com} has a unique irreducible subrepresentation. Also, Proposition \ref{prop max triangle} and Theorem \ref{thm add embed} above give
\[ \pi_{\theta}(\EE) \hookrightarrow \Delta_{\rho}[B_j, -A_j] \rtimes \pi_{\theta}((\EE^{|max|})^{-} ). \]
The induction hypothesis implies that $ \pi_{\theta}((\EE^{|max|})^{-} )=\pi_{\com}((\EE^{|max|})^{-} )$. Therefore, we have $\pi_{\com}(\EE)= \pi_{\theta}(\EE)$ since they are both the unique irreducible subrepresentation of the right hand side of \eqref{eq proof of pi com}. This completes the proof of the theorem.
\end{proof}

The above theorem is the analogue of Theorem \ref{thm Sp}(7).

\begin{remark}
    In this article, we are primarily concerned with the theta lift from even orthogonal groups to symplectic groups. However, the argument also works in the other direction. That is, $\EE$ be an extended multi-segment of $G_n=\Sp(W_n)$ and $\alpha$ be a positive odd integer. We set $\EE_{W,V}$ to be the extended multi-segment obtained by replacing each $\rho$ with $\chi_W\chi_V\inv\rho$ and set \[
    \EE_\alpha=\left\{\left(\left[\frac{\alpha-1}{2},-\frac{\alpha-1}{2}\right]_{\chi_W}, \frac{\alpha-1}{2},\epsilon_{H_m} \right)\right\}+ \EE_{W,V}.
    \]
    We remark that the added extended segment is inserted as the first extended segment in the admissible order, i.e. corresponds to the minimal element in the total ordered set for $\chi_W$. The extended segments in $\EE_{W,V}$ are ordered similarly to $\EE.$ Here $\epsilon_{H_m}\in\{\pm1\}$ depends on the target tower of this theta lift. 
    Then $\EE_{\alpha}$ is an extended multi-segment of $H_m=\OO(V_m)$. Furthermore, let $\alpha_0=2^{2n}+1.$ If $\pi_\theta(\EE_{\alpha_0})\neq0$, then there exists a unique $\pi_\EE\in\Pi_{\psi_\EE}$ such that $\theta_{-\alpha}(\pi_\EE)=\pi_\theta(\EE_{\alpha})$ for any odd $\alpha\geq\alpha_0.$ In this case, we may define $\pi_\theta(\EE)=\pi_\EE.$ We obtain similarly that $\pi_\theta(\EE)=\pi_\com(\EE).$
\end{remark}

\section{\texorpdfstring{Distinguished members of $\Psi(\pi)$: $\psi^{max}(\pi)$ and $\psi^{min}(\pi)$}{}}\label{sec distinguished members}

In this section, we recall several orderings on $\Psi(G_n)$ and relate them with $\geq_O.$ The results in the even orthogonal case follow the same argument as in the symplectic case where analogous results were proved in \cite{HLL22}. Instead, we remark where we make use of the analogous results for even orthogonal groups obtained in this article. For brevity, we omit the proofs.

\subsection{Partition orderings}
Let $N=2n$ or $2n+1$ when $G_n=\OO(W_n)$ or $\Sp(W_n)$, respectively. There is a one-to-one correspondence between the nilpotent orbits of $\widehat{G}_n(\BC)$ and a certain subset of partitions of $N$ (see \cite[\S 5.1]{CM93}). Recall that these are nilpotent orbits of $\OO_{2n}(\BC)$, not $\SO_{2n}(\BC)$, when $G_n=\OO(W_n)$. For any local Arthur parameter of $G_n$,
$$\psi:W_F \times \SL_2^{D}(\mathbb{C}) \times \SL_2^A(\mathbb{C}) \to {}^LG_n^\circ,$$
we attach the nilpotent orbit of $\widehat{G}_n(\BC)$ containing the element 
\[d(\psi|_{\SL_2^{A}(\mathbb{C})}) \left( \begin{pmatrix}
0 &1 \\0 &0
\end{pmatrix} \right).\]
We let $\underline{p}^A(\psi)$ denote the corresponding partition. We explicate this partition now. Let 
\[ \psi= \bigoplus_{i=1}^r \rho_i \otimes S_{a_i} \otimes S_{b_i}, \]
and $d_i= \dim (\rho_i)$. Then we have \[\underline{p}^A(\psi)=[ b_1^{a_1d_1},\dots , b_r^{a_rd_r} ].\] The partition $\underline{p}^A(\psi)$ is a crucial ingredient in Jiang's Conjecture on upper bounds for wavefront sets of representations in the local Arthur packet $\Pi_\psi$ (\cite[Conjecture 4.2]{Jia14} and \cite[Conjecture 1.6]{LS22}).

Note that if $\psi$ is tempered, then $\underline{p}^A(\psi)=[1^{N}]$ which is the minimal partition of $N$ under the dominance order for partitions. We consider the following definition (see also \cite[Definition 11.5]{HLL22}).

\begin{defn}\label{A ordering}
We define a preorder $\geq_{A}$ on $\Psi(G_n)$ by $\psi_1 \geq_A \psi_2$ if $ \underline{p}^A(\psi_1) \leq \underline{p}^A(\psi_2)$ under the dominance order.
\end{defn}

As mentioned in \cite[\S11]{HLL22}, $\geq_{A}$ is only a preorder but not a partial order since it is possible that $\underline{p}^A(\psi_1)=\underline{p}^A(\psi_2)$ but $\psi_1 \neq \psi_2$. The following theorem was proved for symplectic and split odd special orthogonal groups in \cite[Theorem 11.6]{HLL22}. The same proof gives the even orthogonal case as well.

\begin{thm}\label{thm Jiang's partition} \
\begin{enumerate}
    \item If $T$ is a raising operator applicable on $\psi \in \Psi(G_n)$, then 
    \[T(\psi) \gneq_A \psi.\]
    In particular, if $\psi \geq_{O} \psi'$, then $\psi \geq_A \psi'$.
    \item Let $\pi$ be a representation of $G_n$ of Arthur type. The distinguished members $\psi^{max}(\pi)$ and $ \psi^{min}(\pi)$ are the unique elements in $\Psi(\pi)$ satisfying the following inequality
    \[ \psi^{max}(\pi) \geq_A \psi \geq_A \psi^{min}(\pi),\]
    for any $\psi \in \Psi(\pi).$
\end{enumerate}
\end{thm}

On the other hand, we can  consider the restriction to the Deligne-$\SL_2(\mathbb{C})$ instead of the Arthur-$\SL_2(\mathbb{C})$. In this case, we define
\[\underline{p}^D(\psi):=[a_1^{d_1b_1},\dots, a_r^{d_rb_r}],\]
which is the partition corresponding to the nilpotent orbit of $\widehat{G}_n(\BC)$ containing the element
\[ d(\psi|_{\SL_2^{D}(\mathbb{C})})\left(\begin{pmatrix}
0&1\\0&0
\end{pmatrix}\right).\]
It is immediate that $\underline{p}^D(\psi)=\underline{p}^A(\widehat{\psi})$, (recall that $\widehat{\psi}(w,x,y):=\psi(w,y,x)$). We define the next preorder (see also \cite[Definition 4.4]{HLLZ25}).

\begin{defn}\label{def Dordering}
We define a preorder $\geq_{D}$ on $\Psi(G_n)$ by $\psi_1 \geq_{D} \psi_2$ if $\underline{p}^D(\psi_1)\geq \underline{p}^D(\psi_2)$.
\end{defn}

The partial order $\geq_O$ is reversed by taking the dual. For symplectic and split odd special orthogonal groups, this was proved in \cite[Lemma 8.1]{HLLZ25}. The same proof holds for even orthogonal groups.

\begin{lemma}\label{lem order reversing}
If $T$ is a raising operator applicable on $\psi \in \Psi(G_n)$, then there exists another raising operator $T'$ applicable on $ \widehat{T(\psi)}$ such that
\[ \widehat{\psi}= T'(\widehat{T(\psi)} ). \]
In particular, for any $\pi \in \Pi(G_n)$ of Arthur type, the involution $\psi\mapsto\widehat{\psi}$ on $\Psi(\pi)$ is order reversing with respect to $\geq_{O}$.
\end{lemma}

We briefly digress to say that this lemma, along with Theorems \ref{thm psi max and min} and \ref{thm in L-packet} imply that local Arthur packets of quasi-split even orthogonal groups cannot fully contain each other. The proof of the following theorem is the same as in \cite[Theorems 8.2, 8.3]{HLLZ25} (which is the analogous statement for split odd special orthogonal and symplectic groups) except that we use the analogous results for quasi-split even orthogonal groups mentioned above.

\begin{thm}\label{thm noncontainment}
    Let $G_n=\OO(W_n)$ be quasi-split and $\psi_1, \psi_2 \in \Psi(G_n)$. If $\Pi_{\psi_1} \supseteq \Pi_{\psi_2}$, then $\psi_1= \psi_2$.
\end{thm}

\begin{remark}
    The proof uses the fact that $\Pi_{\phi_{\psi_2}}\cap\Pi_{\psi_1}\neq \emptyset$ and $\Pi_{\phi_{\hat{\psi}_2}}\cap\Pi_{\hat{\psi}_1}\neq \emptyset$ implies that $\psi_1=\psi_2$. If $G_n$ is quasi-split, then the $L$-packets $\Pi_{\phi_{\psi_2}}$ and $\Pi_{\phi_{\hat{\psi}_2}}$ are non-empty, so these conditions follow from $\Pi_{\psi_2} \subseteq \Pi_{\psi_1}$. However, this is not guaranteed if $G_n$ is not quasi-split.
\end{remark}

We return to our discussion on the orderings.
Combining Lemma \ref{lem order reversing} with Theorem \ref{thm Jiang's partition} gives an analogous result for the preorder $\geq_D.$ For symplectic and split odd special orthogonal groups, this was proved in \cite[Theorem 4.5]{HLLZ25}. Again, the same proof holds for even orthogonal groups.

\begin{thm}\label{thm D order} \
\begin{enumerate}
    \item If $T$ is a raising operator applicable on $\psi \in \Psi(G_n)$, then 
    \[T(\psi) \gneq_D \psi.\]
    In particular, if $\psi \geq_{O} \psi'$, then $\psi \geq_D \psi'$.
    \item Let $\pi$ be a representation of $G_n$ of Arthur type. The distinguished members $\psi^{max}(\pi)$ and $ \psi^{min}(\pi)$ are the unique elements in $\Psi(\pi)$ satisfying the following inequality
    \[ \psi^{max}(\pi) \geq_D \psi \geq_D \psi^{min}(\pi),\]
    for any $\psi \in \Psi(\pi).$
\end{enumerate}
\end{thm}

\subsection{\texorpdfstring{Normalized intertwining operator ordering $\geq_N$}{}}
Next, we recall an ordering which arises from considering the vanishing of certain normalized intertwining operators.
Let $\pi$ be a representation of $G_n$ of Arthur type. 
Denote
$$\sigma:=\St(\rho',a_0)=\Delta_{\rho'}[(a_0-1)/2,-(a_0-1)/2],$$ 
and consider the usual (non-normalized) intertwining operator
$$
M(s,\pi,\sigma): \St(\rho',a_0) | \cdot | ^s \rtimes\pi \rightarrow\St(\rho',a_0) | \cdot | ^{-s} \rtimes\pi,
$$
We further consider the Langlands-Shahidi normalized intertwining operator given by
\[N^{LS}(s,\pi,\sigma):= M(s,\pi,\sigma) r(s,\pi,\sigma)^{-1},\]
where
\begin{align}\label{eq LS normalization}
    r(s,\pi, \sigma):=\frac{L(s, \sigma \times \pi)}{L(s+1,\sigma \times \pi)} \times \frac{L(\sigma,2s,\bigwedge^2)}{L(\sigma, 2s+1, \bigwedge^2)}.
\end{align}
Note that since $\widehat{G}_n$ is orthogonal, we only consider the $L$-parameters for $\bigwedge^2$ (unlike \cite{HLL22} where $\mathrm{Sym}^2$ is needed for $\SO_{2n+1}(F)$). 

For any $L$-parameter $\phi$ of $G_n$, we also denote $r(s,\phi, \sigma):=r(s,\pi,\sigma)$, where $\pi\in \Pi_{\phi}$. We remark that there is a more general definition of the Langlands-Shahidi normalization (\cite[p. 150]{Sha10}) in terms of the $L$- and $\varepsilon$-factors. Also, since the $\varepsilon$-factors do not provide any poles or zeros we have omitted them in the above definition for simplicity as in \cite{Moe10}.

We recall another normalization, known as the Arthur normalization, which is compatible with the twisted endoscopic character identities (\cite{Moe08, Moe10, Moe11b, Moe12, Art13}). Arthur used these normalized intertwining operators crucially in his proof of the existence of local Arthur packets (see  \cite[\S 2.3 and \S 2.4]{Art13}). Again, following \cite{Moe10}, since the $\varepsilon$-factors do not provide any poles or zeros, we continue to omit them in the definition below.  Recall that $\sigma=\St(\rho',a_0).$
For $\psi\in\Psi(G_n)$, we define $r(s,\psi,\sigma):= r(s, \phi_{\psi},\sigma)$. 
For each $\psi \in \Psi(\pi)$, the Arthur normalized intertwining operator
is defined as
\[ N_{\psi}(s,\pi,\sigma):=  M(s, \pi, \sigma)r(s, \psi, \sigma)^{-1}.\]
Note that if $\pi \in \Pi_{\phi_{\psi}}$, then the Arthur normalized intertwining operator is the same as the Langlands-Shahidi normalized intertwining operator. 
M{\oe}glin showed that the Arthur normalized intertwining operators are holomorphic in the right half plane when $G_n$ is quasi-split.
\begin{thm}[{\cite[Theorem 3.2]{Moe10}}]\label{thm Moeglin holo}
When $G_n$ is quasi-split, $N_\psi(s,\pi,\sigma)$ is holomorphic for any real $s\geq \half{1}.$
\end{thm}
M{\oe}glin also showed that the image is either an irreducible representation or zero and calculated its image (\cite[Theorem A]{Moe11b}).

Let $f$ denote a meromorphic function. We write $\ord_{s=s_0} f(s)$ to denote the order of vanishing at $s_0$ of $f.$ That is, $\ord_{s=s_0} f(s)=m$ where $m$ is an integer such that $(s-s_0)^{-m} f(s)\neq 0$ in a neighborhood of $s_0.$ Note that $\ord_{s=s_0}f(s)<0$ if $f$ has a pole at $s_0$ and $\ord_{s=s_0}f(s)>0$ if $f$ has a zero at $s_0.$ Let $\pi$ be a representation of $G_n$ of Arthur type. Theorem \ref{thm Moeglin holo} states that for any $\psi'\in\Psi(\pi)$ and $s_0\geq \half{1}$, $\ord_{s=s_0} N_{\psi'}(s,\pi, \sigma)\geq 0.$ We define an ordering on $\Psi(\pi)$ using the ordering of vanishing as follows.

\begin{defn} \label{M ordering}\ 
\begin{enumerate}
    \item [(1)] We define an ordering $\geq_{N}$ on the set of $L$-parameters of $G_n$ by $\phi_1 \geq_{N} \phi_2$ if for any $\sigma= \St(\rho',a_0)$ and $s_0 \in \R_{\geq \half{1}}$, the following inequality holds
$$
 \ord_{s=s_0} r(s,\phi_1, \sigma)
 \geq 
 \ord_{s=s_0} r(s,\phi_2, \sigma).
$$
In particular, $\phi_1 >_{N} \phi_2$ if $N^{LS}(s,\phi_1,\sigma)$ has less vanishing than $N^{LS}(s,\phi_2,\sigma)$ for $s \geq \half{1}$.

\item [(2)] Let $\pi$ be a representation of $G_n$ of Arthur type. We define an ordering $\geq_{N}$ on $\Psi(\pi)= \{ \psi \ | \ \pi \in \Pi_{\psi} \}$ by $\psi_1 \geq_{N} \psi_2$ if $\phi_{\psi_1} \geq_{N} \phi_{\psi_{2}}$. 
\end{enumerate}
\end{defn}
Explicit computation of the order of poles and zeros of $r(s,\psi,\sigma)$ is possible via \cite{JPSS83} (see \cite{Moe10}).
We have the following theorem.

\begin{thm}\label{thm N ordering}
    Let $\pi\in\Pi(G_n)$ be of Arthur type. Then the following hold.
    \begin{enumerate}
        \item The ordering $\geq_N$ is a partial order on $\Psi(\pi).$
        \item If $T$ is a raising operator applicable on $\psi \in \Psi(G_n)$, then 
    \[T(\psi) \gneq_N \psi.\]
    In particular, if $\psi \geq_{O} \psi'$, then $\psi \geq_N \psi'$.
    \item The distinguished members $\psi^{max}(\pi)$ and $ \psi^{min}(\pi)$ are the unique elements in $\Psi(\pi)$ satisfying the following inequality
    \[ \psi^{max}(\pi) \geq_N \psi \geq_N \psi^{min}(\pi),\]
    for any $\psi \in \Psi(\pi).$
    \item If $T$ is a raising operator applicable on $\psi$, then for any $\sigma=\St(\rho',a_0)$, the quotient
\[ \frac{r(s,\phi_{T(\psi)},\sigma)}{r(s,\phi_{\psi},\sigma)}\]
has finitely many zeros on $\R_{\geq \half{1}}$.
\item For any $\psi \in \Psi(\pi)$, we have $\phi_{\pi} \geq_{N} \phi_{\psi}$, where $\phi_{\pi}$ is the $L$-parameter of $\pi$. Moreover, if $ \pi \not\in \Pi_{\phi_{\psi}}$, then $\phi_{\pi} \gneq_{N} \phi_{\psi}$.
\item If $\phi_\pi\neq\phi_\psi$ for any $\psi\in\Psi(G_n)$, then for any $\psi \in \Psi(\pi)$, there exists a $\sigma$ and $s_0 \in \R_{\geq{\half{1}}}$ such that $\ord_{s=s_0} N_{\psi}(s_0,\pi,\sigma)> \ord_{s=s_0} N^{LS}(s_0,\pi,\sigma)$.
    \end{enumerate}
\end{thm}

When $G_n=\Sp(W_n),$ the results are contained in \cite[\S11.2]{HLL22}. When $G_n=\OO(W_n),$ the same proofs work; however, we remark that in the proof of Part (1) (which follows the argument of \cite[Proposition 11.12]{HLL22}), we use Lemma \ref{lem $L$-data} in place of \cite[Lemma 4.7]{HLL22}. Furthermore, in the proof of Part (5) (which follows the argument of \cite[Theorem 11.15]{HLL22}), we use Proposition \ref{prop max triangle} instead of \cite[Proposition 10.11]{HLL22}.

\section{Closure ordering conjecture}\label{sec closure order conj}

In this section, we establish the closure ordering conjecture (see Conjecture \ref{conj closure}) for $\OO(W_n)$ (see Theorems \ref{thm O implies C} and \ref{thm phi pi > phi psi}). The results for even orthogonal groups follow the same argument as in the symplectic case where analogous results were proved in \cite{HLLZ25} (using the analogous results for even orthogonal groups obtained in this article). As in the previous section, we omit the proofs for brevity. Throughout this section, we assume that $G_n=\OO(W_n)$.

Recall that we set $\widehat{G_n}(\BC):=\OO_{2n}(\BC)$.
To an $L$-parameter $\phi$ of $G_n$, we associate a homomorphism $\lambda_\phi:W_F\rightarrow{}^LG_n $ via
$\lambda_{\phi}(w):=\phi(w,d_w)$, where $d_w=\begin{pmatrix}
    |w|^\half{1} & 0 \\ 0 & |w|^\half{-1}
\end{pmatrix}.$ The homomorphism $\lambda_\phi$ is an infinitesimal parameter of $G_n^\circ$ (see \cite[\S 4.1]{CFMMX22}). Conversely, for each infinitesimal parameter $\lambda$ of $G_n^\circ$, we consider
\[ \Phi_{\lambda}(G_n):=\{\phi\in\Phi(G_n) \ | \ \lambda_{\phi}=\lambda\},\ \ \Phi_{\lambda}(G_n^{\circ}):=\{\phi\in\Phi(G_n^{\circ}) \ | \ \lambda_{\phi}=\lambda\}. \]
Recall that here $\Phi(G_n)$ is the quotient of the set of $L$-parameters of $G_n^\circ$, denoted $\Phi(G_n^{\circ})$, by outer conjugation.

We consider the set
\begin{align*}
H_{\lambda}&:= \{g \in \widehat{G_n}(\BC) \ | \ \lambda(w)g= g\lambda(w), \forall w \in W_F \}.
\end{align*}
We note that this definition of $H_\lambda$ agrees with that of \cite{Vog93, CFMMX22} when $G_n$ is connected. We also consider the set
\[
H_{\lambda}^\circ:= \{g \in \widehat{G_n^\circ}(\BC) \ | \ \lambda(w)g= g\lambda(w), \forall w \in W_F \}.
\]
When $G_n$ is connected, we have $H_\lambda=H_\lambda^\circ.$

The \emph{Vogan variety} $V_{\lambda}$ for an infinitesimal parameter $\lambda$ of $G_n$ is defined by
\[ V_{\lambda}:= \{ x \in \mathrm{Lie}(\widehat{G^\circ_n})(\BC) \ | \ \mathrm{Ad}(\lambda(w))x= |w|x, \forall w \in W_F\}. \]
 The Vogan variety $V_{\lambda}$ admits actions of the groups $H_\lambda^\circ$ and $H_\lambda$ via conjugation. Furthermore, this action has finitely many orbits (\cite[Proposition 5.6]{CFMMX22}). 

We continue by defining the closure ordering for the connected group $G_n^\circ.$
For each $\phi \in \Phi_{\lambda}(G_n^\circ)$, we have that the element
\[ X_{\phi}:= d (\phi|_{\SL_2(\BC)}) \left(  \begin{pmatrix}
0&1\\0&0
\end{pmatrix}\right)\]
lies in $V_{\lambda}$. Let $C_{\phi}'$ denote the $H_{\lambda}^\circ$-orbit of $X_{\phi}$. Consequently, there is a map
\begin{align*}
     \Phi_{\lambda}(G_n^\circ)& \to V_{\lambda}/H_{\lambda}^\circ,\\
     \phi& \mapsto C_{\phi}'.
\end{align*}
In fact, this map is a bijection (see \cite[Proposition 4.2]{CFMMX22}). Thus, the geometric structure of $V_{\lambda}$ equips $\Phi_{\lambda}(G_n^\circ)$ with a partial order $\geq_C$ defined by the closure relation.
\begin{defn}\label{def C ordering connected}
For any infinitesimal character $\lambda$ of $G_n^\circ$, the finite set $\Phi_{\lambda}(G_n^\circ)$ is equipped with a partial ordering $\geq_C$ defined by $\phi_{1} \geq_C \phi_2$ if $\overline{C_{\phi_1}'} \supseteq C_{\phi_2}'$.
\end{defn}

The closure ordering is defined in the disconnected case as well. Recall that an $L$-parameter of $G_n$ is simply an $L$-parameter of $G_n^\circ,$ except the equivalence is by $\widehat{G_n}$-conjugation. Thus, the bijection 
\[
\Phi_{\lambda}(G_n^\circ)\to V_{\lambda}/H_{\lambda}^\circ
\]
induces a bijection 
\[
\Phi_{\lambda}(G_n)\to V_{\lambda}/H_{\lambda}.
\]
See \cite[Lemma 2.1]{HL26} for details.
Given an $L$-parameter $\phi\in\Phi(G_n),$ we let $C_\phi$ denote the corresponding $H_\lambda$-orbit (which is the $H_\lambda$-orbit of $X_\phi$). This induces the closure ordering on $G_n$.
\begin{defn}\label{def C ordering}
For any infinitesimal character $\lambda$ of $G_n$, the finite set $\Phi_{\lambda}(G_n)$ is equipped with a partial ordering $\geq_C$ defined by $\phi_{1} \geq_C \phi_2$ if $\overline{C_{\phi_1}} \supseteq C_{\phi_2}$.
\end{defn}

Fix an infinitesimal character $\lambda$ of $G_n$ and consider the set
\begin{align}\label{eq psi lambda}
     \Psi_{\lambda}(G_n):=\{\psi \in \Psi^+(G_n)\ | \ \lambda_{\phi_{\psi}}=\lambda\},
\end{align}
which can be viewed as a subset of $\Phi_{\lambda}(G_n)$ via the injection $\psi\mapsto\phi_\psi$. By M{\oe}glin's construction of $\Pi_\psi$ (Theorem \ref{thm Moeglin construction}), all representations in a fixed $\Pi_{\psi}$ have the same extended cuspidal support (\cite[Proposition 4.1]{Moe09b}). This implies that, for any $\pi \in \Pi_{\psi}$, we have $\lambda_{\phi_{\pi}}=\lambda_{\phi_{\psi}}$, and for a fixed $\pi\in \Pi(G_n)$, we have  $\Psi(\pi) \subseteq \Psi_{\lambda_{\phi_{\pi}}}(G_n)$. Consequently, we may restrict the partial order $\geq_C$ from $\Phi_{\lambda}(G_n)$ to $\Psi_{\lambda}(G_n)$ (or $\Psi(\pi)$).
\begin{defn}\label{def C ordering Arthur}
For any infinitesimal character $\lambda$ of $G_n$, we define a partial order $\geq_{C}$ on the set $\Psi_{\lambda}(G_n)$ by $\psi_1 \geq_C \psi_2$ if $\phi_{\psi_1} \geq_C \phi_{\psi_2}$.
\end{defn}

We recall \cite[Conjecture 1.4]{HLLZ25} (which is a refinement of the closure ordering conjecture in \cite[Conjecture 3.1]{Xu24}).

\begin{conj}[{\cite[Conjecture 1.4]{HLLZ25}}]\label{conj closure}
Let $\mathrm{G}$ be a connected reductive group defined over a non-Archimedean local field $F$. Assume that there is a local Arthur packets theory for $G=\mathrm{G}(F)$ as conjectured in \cite[Conjecture 6.1]{Art89}. Let $\pi$ be an irreducible admissible representation of $G$ which is of Arthur type. Then, there are unique elements $\psi^{max}(\pi)$ and $\psi^{min}(\pi)$ in $\Psi(\pi)$ such that
\begin{align*}
  \phi_{\pi} \geq_C   \phi_{\psi^{max}(\pi)} \geq_C \phi_\psi \geq_C \phi_{\psi^{min}(\pi)},
\end{align*} 
for any local Arthur parameter $\psi \in \Psi(\pi)$. 
\end{conj}

Conjecture \ref{conj closure} is verified in \cite[Theorem 1.3]{HLLZ25} for symplectic and split odd special orthogonal groups. However, its formulation for even special orthogonal groups suggests an even orthogonal analogue which we verify below (see Theorems \ref{thm O implies C} and \ref{thm phi pi > phi psi}). 

First, we see that the partial order $\geq_O$  on $\Psi(G_n)$ implies the closure ordering $\geq_C.$ For $\Sp_{2n}(F),$ the claim is precisely \cite[Theorem 4.5]{HLLZ25}. The same proof generalizes to our situation and we omit it.
\begin{thm}\label{thm O implies C} \
\begin{enumerate}
    \item If $T$ is a raising operator applicable on $\psi \in \Psi(G_n)$, then 
    \[T(\psi) \gneq_C \psi.\]
    In particular, if $\psi \geq_{O} \psi'$, then $\psi \geq_C \psi'$.
    \item Let $\pi$ be a representation of $G_n$ of Arthur type. The distinguished members $\psi^{max}(\pi)$ and $ \psi^{min}(\pi)$ are the unique elements in $\Psi(\pi)$ satisfying the following inequality
    \[ \psi^{max}(\pi) \geq_C \psi \geq_C \psi^{min}(\pi),\]
    for any $\psi \in \Psi(\pi).$
\end{enumerate}
\end{thm}

Consequently, towards the closure ordering conjecture for $G_n=\OO(W_n)$, it remains to show that $\phi_\pi\geq_C\phi_\psi$ for any $\psi\in\Psi(\pi).$ When $G_n=\Sp_{2n}(F)$, the result is \cite[Theorem 6.4]{HLLZ25} and the same proof works in the  even orthogonal case. Indeed, the key step in the proof of \cite[Theorem 6.4]{HLLZ25} is to use \cite[Proposition 6.3]{HLLZ25}; however, we have already established its analogue in Theorem \ref{thm add embed} and Proposition \ref{prop max triangle} for even orthogonal groups. Consequently, we confirm the closure ordering conjecture holds for $G_n=\OO(W_n)$.

\begin{thm}\label{thm phi pi > phi psi}
    For any $\pi\in\Pi(G_n)$ of Arthur type, we have
\[ \phi_{\pi} \geq_C \phi_{\psi^{max}(\pi)}.\]
\end{thm}

\section{Enhanced Shahidi Conjecture}\label{sec enhanced Shahidi}
In this section, we prove the enhanced Shahidi conjecture (Theorem \ref{thm enhanced Shahidi}) for $\OO(W_n^+)$. We begin by recalling the conjecture for quasi-split connected reductive groups.

\begin{conj}[{\cite[Conjecture 1.5]{LS22}}, Enhanced Shahidi Conjecture]\label{conj enhanced Shahidi}
Let $\mathrm{G}$ be a connected reductive group which is quasi-split over $F$. Assume that there is a local Arthur packets theory for $G=\mathrm{G}(F)$ as conjectured in \cite[Conjecture 6.1]{Art89}. Then, a local Arthur parameter $\psi \in \Psi(G)$ is tempered if and only if $ \Pi_{\psi}$ has a generic member.
\end{conj}
For split symplectic or odd special orthogonal groups, the result is already known via 
\cite[Theorem 1.2]{HLL24} or
\cite[Theorem 7.10]{HLLZ25}. Its analogue for the disconnected group $\OO(W_n^+)$ requires pinning down precisely what is meant by a generic representation of $\OO(W_n^+)$, which we briefly recall from \cite[\S 2.2]{AG17b}.

We have fixed a Whittaker datum $\mathfrak{w}$ of the quasi-split group $G_n^{\circ}=\SO(W_n^+)$, which is a $G_n^{\circ}$-conjugacy class of a pair $(B, \mu)$ where $B$ is a Borel subgroup of $G_n^{\circ}$ with unipotent radical $U$, and $\mu: U \to \BC^{\times}$ is a character that is non-trivial on all simple roots. Recall that we fixed an element $c \in \OO(W_n^+) \setminus \SO(W_n^+)$ of order two that normalizes $B$, which also normalizes $U$. Thus, we may consider the disconnected group $\widetilde{U}:= \langle U, c\rangle$. The character $\mu$ has two extensions $ \mu^{\pm}: \widetilde{U} \to \BC^{\times}$, where we label it so that $\mu^{\pm}(c)= \pm 1$. We say a representation $\pi$ of $G_n$ is $\mathfrak{w}$-generic (resp. $\mathfrak{w}^{\pm}$-generic) if 
\[ \dim(\Hom_U(\pi, \mu)) >0 \ \ (\text{resp. } \dim(\Hom_{\widetilde{U}}(\pi, \mu^{\pm}))>0 ) .\]
Note that $\pi$ is $\mathfrak{w}^{\pm}$-generic if and only if $\pi \otimes \det$ is $\mathfrak{w}^{\mp}$-generic. Also, $\pi$ is generic if and only if $\pi$ is $\mathfrak{w}^{\delta}$-generic for some $\delta \in \{+,- \}$.

Now we prove the enhanced Shahidi conjecture for $\OO(W_n^+)$. 

\begin{thm}\label{thm enhanced Shahidi}
    Assume that $G_n=\OO(W_n^+)$.
    A local Arthur parameter $\psi \in \Psi(G_n)$ is tempered if and only if $ \Pi_{\psi}$ has a generic member.
\end{thm}
{
\begin{proof}
   If $\psi$ is tempered, then $\Pi_\psi$ has a $\mfr{w}^+$-generic member $\pi(\psi, \textrm{triv})$ by \cite[Proposition 10.4]{CZ21}.
    Thus it suffices to show the converse direction. In fact, we show the stronger statement that if $\pi$ is generic and of Arthur type, then $\Psi(\pi)$ is a singleton consisting of a tempered local Arthur parameter. Since $\pi$ is generic, we have that $\pi\otimes\det$ is also generic. Furthermore, we have that $\Psi(\pi)= \Psi(\pi \otimes \det)$ by Corollary \ref{cor extended multi-segment for det}. Hence, by replacing $\pi$ by $\pi\otimes\det$ if necessary, we may assume that $\pi$ is $\mathfrak{w}^+$-generic.  By Theorem \ref{thm red from nu to gp symp}, we may further assume that $\pi$ is of good parity.

    First, we consider the case that $\pi$ is tempered. In this case, we have $\pi =\pi(\phi, \textrm{triv})$, where $\phi$ is tempered by \cite[Proposition 10.4]{CZ21} again. Write the tempered $L$-parameter as
    \[ \phi= \bigoplus_{\rho} \bigoplus_{i \in I_{\rho}} \rho \otimes S_{2A_i+1}.\]
    We have $\pi= \pi_{\com}(\EE)$ where 
    \[ \EE= \cup_{\rho}\{ ([A_i,A_i]_{\rho},0, 1)\}_{i \in (I_{\rho},>)}.\]
    Then it is clear that there is no raising operator or inverse of raising operator applicable on any member of $[\EE].$ Therefore, Theorem \ref{thm O} implies that $\Psi(\pi)= \{\psi_{\EE}\}$.

    Next, we show that a generic representation of Arthur type and of good parity must be tempered. Equivalently, we show that if $\pi$ is of good parity and Arthur type of the form 
    \[ \pi=L(\Delta_{\rho_1}[x_1,-y_1],\ldots, \Delta_{\rho_f}[x_f,-y_f]; \pi(\phi,\varepsilon)),\]
    with $f \geq 1$, then $\pi$ is not generic. 
    
    We first reduce to the case that $f=1$ by considering the representation
 \[ \pi':= L(\Delta_{\rho_f}[x_f,-y_f]; \pi(\phi,\varepsilon)).\]
Indeed, it follows from Algorithm \ref{algo Arthur type pi^-} that $\pi'$ is of Arthur type. If the case that $f=1$ is verified, then $\pi'$ cannot be generic. Then since
\[ \pi \hookrightarrow \Delta_{\rho_1}[x_1,-y_1]\times \cdots\times  \Delta_{\rho_{f-1}}[x_{f-1},-y_{f-1}]\rtimes \pi', \]
we conclude that $\pi$ is not generic as well.

    Now we deal with the case that $f=1$. Theorem \ref{thm algo for Arthur type} implies that $\Psi(\pi^{\rho_1,-}; \Delta_{\rho_1}[x_1,-y_1],1)$ is non-empty. On the other hand, since $\pi^{\rho_1,-}:=\pi(\phi,\varepsilon)$ is generic, we have showed at the beginning of this proof that $\Psi(\pi^{\rho_1,-})$ is a singleton consisting of the tempered Arthur parameter $\phi$. Thus, it follows from the definition of $\Psi(\pi^{\rho_1,-}; \Delta_{\rho_1}[x_1,-y_1],1)$ that one of the following conditions must hold:
    \begin{itemize}
        \item [(i)] $y_1-x_1=1$. Namely, $\Delta_{\rho_1}[x_1,-y_1]= \Delta_{\rho_1}[x_1,-x_1-1]. $
        \item [(ii)] $y_1-x_1=2$ and $\phi$ contains at least one copy of $\rho \otimes S_{2x_1+3}$.
    \end{itemize}
    In the rest of the proof, we show that in both cases, the standard module $M(\pi)$ must contain a generic tempered representation. This implies that $\pi$ cannot be generic by counting the dimension of the Whittaker functional of $M(\pi)$.

In the following, we will use the formula for highest derivatives in \cite[Theorem C.4.3]{AGIKMS24}, and we refer the reader to \cite[\S C.2]{AGIKMS24} for the notation of derivatives.    To simplify the notation, we write $\rho:=\rho_1$, $x:=x_1$, and let $m_{z}$ denote the multiplicity of $\rho\otimes S_{2z+1}$ in $\phi$, for $z \in \{x,x+1,x+2\}$. Also, we assume that $\varepsilon= \textrm{triv}$ by replacing $\pi$ with $\pi \otimes \det$ if necessary.

For case (i), we consider the following tempered $L$-parameters
\begin{align*}
    \phi_1&:= \phi- (\rho \otimes S_{2x+3})^{\oplus m_{x+1}}+ (\rho \otimes S_{2x+1})^{\oplus m_{x+1}},\\
    \phi_2&:= \phi- (\rho \otimes S_{2x+3})^{\oplus m_{x+1}}+ (\rho \otimes S_{2x+1})^{\oplus (m_{x+1}+2)},\\
    \phi_3&:=\phi+ \rho \otimes S_{2x+3} +\rho \otimes S_{2x+1}.
\end{align*}
We have
\begin{align*}
    D_{\rho\lvert\cdot\rvert^{x+1}}^{(m_{x+1}+1)} (M(\pi))&= D_{\rho\lvert\cdot\rvert^{x+1}}^{(m_{x+1}+1)} (\Delta_{\rho}[x,-x-1] \rtimes \pi(\phi,\textrm{triv}))\\
    &=\Delta_{\rho}[x,-x] \rtimes D_{\rho\lvert\cdot\rvert^{x+1}}^{(m_{x+1})}(\pi(\phi,\textrm{triv}))\\
    &= \Delta_{\rho}[x,-x] \rtimes \pi(\phi_1, \textrm{triv})\\
    &= \pi(\phi_2, \textrm{triv}).
\end{align*}
is a highest derivative. Here the second equality follows from Tadi{\'c}'s formula (see \cite[Theorem C.1.1]{AGIKMS24}), the third equality follows from the highest derivative formula in \cite[Theorem C.4.3]{AGIKMS24}, and the fourth equality follows from the formula for tempered induction (see \cite[Proposition 4.2]{Ato22b} or the proof of \cite[Proposition 2.4.3]{Art13}). On the other hand, the highest derivative formula also gives 
\[  D_{\rho\lvert\cdot\rvert^{x+1}}^{(m_{x+1}+1)} ( \pi(\phi_3, \textrm{triv}))=\pi(\phi_2, \textrm{triv}). \]
    We conclude by \cite[Lemma C.2.2]{AGIKMS24} that the tempered representation $ \pi(\phi_3, \textrm{triv})$ must occur in $M(\pi)$.

For case (ii), we set 
    \begin{align*}
    \phi_1&:= \phi- (\rho \otimes S_{2x+5})^{\oplus m_{x+2}}+ (\rho \otimes S_{2x+3})^{\oplus m_{x+2}},\\
    \phi_2&:= \phi- (\rho \otimes S_{2x+5})^{\oplus m_{x+2}}- (\rho \otimes S_{2x+3})^{\oplus m_{x+1}} + (\rho \otimes S_{2x+1})^{\oplus (m_{x+2}+m_{x+1})},\\
    \phi_3&:=\phi_2+ (\rho\otimes S_{2x+1})^{\oplus 2},\\
    \phi_4&= \phi+ \rho \otimes S_{2x+5} +\rho \otimes S_{2x+1}.
\end{align*}
A similar argument shows that the following is a composition of highest derivatives.
\begin{align*}
    D_{\rho\lvert\cdot\rvert^{x+1}}^{(m_{x+1}+m_{x+2}+1)}\circ  D_{\rho\lvert\cdot\rvert^{x+2}}^{(m_{x+2}+1)} (M(\pi))&= D_{\rho\lvert\cdot\rvert^{x+1}}^{(m_{x+1}+m_{x+2}+1)}\circ  D_{\rho\lvert\cdot\rvert^{x+2}}^{(m_{x+2}+1)} (\Delta_{\rho}[x,-x-2] \rtimes \pi(\phi,\textrm{triv}))\\
    &=  D_{\rho\lvert\cdot\rvert^{x+1}}^{(m_{x+1}+m_{x+2}+1)}( \Delta_{\rho}[x,-x-1] \rtimes D_{\rho\lvert\cdot\rvert^{x+2}}^{(m_{x+2})} (\pi(\phi, \textrm{triv}))  )\\
     &= D_{\rho\lvert\cdot\rvert^{x+1}}^{(m_{x+1}+m_{x+2}+1)}( \Delta_{\rho}[x,-x-1] \rtimes \pi(\phi_1, \textrm{triv})  )\\
     &= \Delta_{\rho}[x,-x] \rtimes D_{\rho\lvert\cdot\rvert^{x+1}}^{(m_{x+1}+m_{x+2})} (\pi(\phi_1, \textrm{triv}))  )\\
     &= \Delta_{\rho}[x,-x] \rtimes \pi(\phi_2, \textrm{triv})  \\
     &=  \pi(\phi_3, \textrm{triv}). 
\end{align*}
    Then since 
    \[D_{\rho\lvert\cdot\rvert^{x+1}}^{(m_{x+1}+m_{x+2}+1)}\circ  D_{\rho\lvert\cdot\rvert^{x+2}}^{(m_{x+2}+1)} (\pi(\phi_4,\textrm{triv}))= \pi(\phi_3, \textrm{triv}),\]
    we conclude that $\pi(\phi_4, \textrm{triv})$ must occur in $M(\pi)$. This completes the proof of the theorem.
\end{proof}

}

\begin{remark}
The enhanced Shahidi conjecture also holds for $\Psi^+(G_n)$: a local Arthur parameter $\psi\in \Psi^+(G_n)$ is generic, meaning that it is trivial on the Arthur-$\SL_2(\BC)$, if and only if $\Pi_\psi$ contains a generic member. This is a direct consequence of Theorems \ref{thm red from nu to gp symp} and \ref{thm enhanced Shahidi}. This is the analogue of a known result for split symplectic and odd special orthogonal groups (\cite[Theorem 4.7]{HLL24}).
\end{remark}

\section{Weak local Arthur packet}\label{sec weak local Arthur packet}
In this section, we prove the conjecture on weak local Arthur packets in \cite{CMO24} for $\OO(W_n)$. Moreover, we give a framework to verify this conjecture for general groups, assuming some desiderata for local Arthur packets.

\subsection{The conjecture and the relevant results}
Throughout the section, we let $G$ be a connected reductive algebraic group defined over $F$, which is a pure inner form of a connected split group. 

Let $\Pi_{\unip}(G)\subseteq \Pi(G)$ denote the subset of representations of unipotent cuspidal support defined in \cite{Lus95}.  
Kazhdan and Lusztig (\cite{KL87, Lus95, Lus02}) defined an enhanced local Langlands correspondence on unipotent representations of $G$, which we refer the reader to \cite[Theorem 4.1.1]{CMO25} for details:
\begin{align}\label{eq Lusztig LLC}
    \LLC^{\Lus}: \Pi_{\unip}(G) \to \Phi_{unr}^{\mfr{e}}(G),\ \ \ \pi \mapsto( \phi_{\pi}, \varepsilon_{\pi})
\end{align}
where $\Phi_{unr}^{\mfr{e}}(G)$ is the set of pairs $(\phi,\varepsilon)$, where $\phi$ is an unramified $L$-parameter of $G$, i.e., $\phi$ is trivial on the inertia subgroup, and $\varepsilon\in \mathrm{Irr}(\pi_0(A_\phi))$ where $A_\phi=\pi_0(Z_{\widehat{G}(\BC)}(\mathrm{im}(\phi)))$. We denote the $L$-packet for unramified $L$-parameter $\phi$ by $\Pi_{\phi}(G)=\{\pi \in \Pi_{\unip}(G)\ | \ \phi_{\pi}=\phi \}$ as usual. We set 
\[\Pi_{\lambda_{\phi}}(G):= \{ \pi \in \Pi_{\unip}(G) \ | \ \lambda_{\phi_{\pi}}=\lambda_{\phi}\}.\]

Let $\Phi_{unr,\R}(G)$ be the set of unramified $L$-parameters such that $\lambda_{\phi}(\Frob)$ is hyperbolic and set $\Pi_{\unip,\R}(G):= \bigsqcup_{\phi \in \Phi_{unr,\R}(G)} \Pi_{\phi}(G).$ 
The main result of \cite{CMO25} computes the geometric wavefront set $\WF(\pi)$ for $\pi \in \Pi_{\unip,\R}(G)$.

\begin{thm}[{\cite[Theorem 1.4.1]{CMO25}}]\label{thm CMO WF formula}
  Let $G$ be a connected reductive algebraic group defined over $F$, which is an inner form of a connected split group. Assume that $p$ is sufficiently large. Then for $\pi \in \Pi_{\unip,\R}(G)$, we have 
  \[\WF(\pi)=\{d_{BV}(\cO_{\phi_{\mathsf{AZ}({\pi})}})\}.\]
  Here $d_{BV}$ is the Barbasch-Vogan duality (\cite{BV85}).
\end{thm}

Now we define the weak local Arthur packet. 
\begin{defn}\label{def weal LAP}
    Let $\cO$ be any nilpotent orbit of $\widehat{G}(\BC)$. 
    \begin{enumerate}
        \item We let $\phi_{\cO}$ be the tempered $L$-parameter such that $\phi_{\cO}|_{W_F}$ is trivial and $\cO_{\phi_{\cO}}= \cO$. Then we set $\lambda_{\cO}:= \lambda_{\phi_{\cO}}$.
        \item The weak local Arthur packet associated with $\cO$ is defined by
        \[ \Pi_{\cO}^{\weak}(G):= \{ \pi \in \Pi_{\lambda_{\cO}}(G)\ | \ \WF(\pi) =\{d_{BV}(\cO)\} \}.\]
    \end{enumerate}
\end{defn}
\begin{remark}
    In \cite[(3.1.1)]{CMO24}, the weak local Arthur packet $\Pi_{\cO}^{\weak}(G)$ is defined by the collection of $\pi \in \Pi_{\lambda_{\cO}}(G)$ such that $\WF(\pi) \leq d_{BV}(\cO)$, instead of the equality. However, Theorem \ref{thm CMO WF formula} implies that for any $\pi \in \Pi_{\lambda_{\cO}}(G)$, we have $\WF(\pi) \geq d_{BV}(\cO)$, so   
    the two definitions agree.
\end{remark}

Here is the main conjecture we will study in this section.

\begin{conj}\label{conj weak local A-packets}
    Assume that there is a theory of local Arthur packets for $G$. Then $ \Pi_{\cO}^{\weak}(G)$ is a union of local Arthur packets of $G$.
\end{conj}

The conjecture is proved for split $\SO_{2n+1}(F)$ and $\Sp_{2n}(F)$ (when $p \gg 0$) by \cite{LL24} and \cite{GO24} independently. It is also proved for the minimal distinguished nilpotent orbit for $F_4$ in \cite{BL25}, with local Arthur packets replaced by ABV-packets of $L$-parameters of Arthur type, and assuming that Aubert-Zelevinsky involution can be described by Fourier transform.

\subsection{Desiderata for local Arthur packets and ABV packets}

With Lusztig's local Langlands correspondence \eqref{eq Lusztig LLC}, we let $\Pi_{\phi}^{\ABV}(G)$ denote the ABV-packets defined in \cite{CFMMX22}. We recall some properties of these packets.

\begin{prop}\label{prop ABV}
    Let $\phi$ be an unramified $L$-parameter of $G$.
    \begin{enumerate}
        \item [(a)] We have $\Pi_{\phi}(G) \subseteq \Pi_{\phi}^{\ABV}(G) \subseteq \Pi_{\lambda_{\phi}}(G)$.
        \item [(b)] If $\pi \in \Pi_{\phi}^{\ABV}(G)$, then $\phi_{\pi}\geq_C \phi$.
    \end{enumerate}
\end{prop}

For general groups, local Arthur packets are not yet defined. Thus, in the following working hypothesis, we state the desiderata for local Arthur packets, together with one more assumption regarding ABV packets.

\begin{assu}\label{assu Arthur ABV}
    Let $G$ be a connected reductive algebraic group defined over $F$, which is a pure inner form of a connected quasi-split group.
    \begin{enumerate}
        \item [(a)]We assume that for each local Arthur parameter $\psi$ of $G$ whose restriction to the inertia subgroup is trivial, there exists a finite set $\Pi_{\psi}(G)$ with the following properties.
        \begin{enumerate}
            \item [(a1)]We have $\Pi_{\phi_\psi}(G) \subseteq \Pi_{\psi}(G) \subseteq \Pi_{\lambda_{\phi_{\psi}}}(G)$.
            \item [(a2)] If $\pi \in \Pi_{\psi}(G)$, then $\phi_{\pi}\geq_C \phi_{\psi}$.
            \item [(a3)] We have $\Pi_{\widehat{\psi}}(G)=\{\mathsf{AZ}({\pi})\ | \ \pi \in \Pi_{\psi}(G) \}$.
        \end{enumerate}
         \item [(b)]We assume that the ABV-packets for unramified $L$-parameters satisfy the property that
    \[ \Pi_{\widehat{\phi}}^{\ABV}(G)=\{\mathsf{AZ}({\pi})\ | \ \pi \in \Pi_{\phi}^{\ABV}(G)\},\]
    where $\widehat{\phi}$ is the Pyasetskii involution of $\phi$.
    \end{enumerate}
   
\end{assu}

\begin{remark}\label{rmk assum ABV arthur}\ 
    \begin{enumerate}
        \item For split $\Sp_{2n}(F)$, $\SO_{2n+1}(F)$,  and pure inner forms of $\OO_{2n}(F)$, if we define
        \[ \Pi_{\psi}(G):= \{ \pi_{\com}(\EE)\ | \ \psi_{\EE}=\psi\}{\smallsetminus} \{0\},\]
        then Working Hypothesis \ref{assu Arthur ABV}(a) is known (with obvious modification for the disconnected group $\OO_{2n}(F)$). Indeed, for $\Sp_{2n}(F)$, see Theorem \ref{thm Sp} (which also holds for split $\SO_{2n+1}(F)$). For $\OO(W_n)$, (a1) follows from Theorem \ref{thm in L-packet} and the definition of $\pi_{\com}$, (a2) is proved in Theorem \ref{thm phi pi > phi psi}, and (a3) follows from Proposition \ref{prop dual is Aubert dual} and Theorem \ref{thm theta=com}.        
        For pure inner forms of $\SO_{2n+1}(F)$, we expect the same holds. See \S \ref{sec meta}. 
        \item Part (b) of Working Hypothesis \ref{assu Arthur ABV} holds if one can relate the Aubert-Zelevinsky involution on $\Pi_{\unip}(G)$ with Fourier transform on the corresponding perverse sheaves. See \cite[\S 10.3.4]{CFMMX22}.
    \end{enumerate}
\end{remark}

\subsection{The ABV-packets version of weak local Arthur packets conjecture}
In this subsection, we prove the following ABV-packets version of the weak local Arthur packets conjecture under Working Hypothesis \ref{assu Arthur ABV}. We first give some notation.

\begin{defn}
    For each nilpotent orbit $\cO$ of $\widehat{G}(\BC)$, we define
    \[\Phi_{\cO}(G)= \{ \phi \in \Phi_{\lambda_{\cO}}(G)\ | \ d_{BV}(\cO_{\phi})= d_{BV}(\cO) \}.\]
\end{defn}

Let $\widehat{\Pi}_{\phi}(G)=\{\mathsf{AZ}(\pi)\ | \ \pi \in \Pi_{\phi}(G)\}$.

\begin{prop}\label{prop ABV weak LAP}
Assume Working Hypothesis \ref{assu Arthur ABV}(b) holds for $G$ and that $p$ is sufficiently large.
    The weak local Arthur packet $\Pi_{\cO}^{\weak}(G)$ is a union of ABV-packets. More precisely, we have 
    \[  \Pi_{\cO}^{\weak}(G)= \bigcup_{\phi\in \Phi_{\cO}(G)} \widehat{\Pi}_{\phi}(G)=\bigcup_{\phi\in \Phi_{\cO}(G)} \Pi_{\widehat{\phi}}^{\ABV} (G).\]
\end{prop}
\begin{proof}
    The first equation is a direct consequence of Theorem \ref{thm CMO WF formula}. Now we show the second equality. By Proposition \ref{prop ABV}(a) and Working Hypothesis \ref{assu Arthur ABV}(b), we have 
    \[ \widehat{\Pi}_{\phi}(G) \subseteq \{ \mathsf{AZ}(\pi)\ | \ \pi \in \Pi_{\phi}^{\ABV}(G) \}= \Pi_{\widehat{\phi}}^{\ABV}(G).\]
    This shows the left hand side is contained in the right hand side.
    
    Next, we observe that if $\phi_{1}, \phi_2 \in \Phi_{\lambda_{\cO}}(G)$ satisfies that
    $\phi_1 \in \Phi_{\cO}(G)$ and $\phi_2 \geq_{C} \phi_1$, then $\phi_2 \in \Phi_{\cO}(G)$ as well. Indeed, from $\phi_{\cO} \geq_C \phi_2 \geq_{C} \phi_1$, we see that $\cO=\cO_{\phi_{\cO}} \geq \cO_{\phi_2} \geq \cO_{\phi_1}$. Then, taking Barbasch-Vogan dual, which is order reversing, we obtain that 
    \[d_{BV}(\cO)\leq d_{BV}( \cO_{\phi_2}) \leq d_{BV}(\cO_{\phi_1})= d_{BV}(\cO).\]
    Thus, the inequalities above are all equalities. This verifies that $\phi_2 \in \Phi_{\cO}(G)$. Now by this observation, Proposition \ref{prop ABV} implies that for any $\phi \in \Phi_{\cO}(G)$, we have 
    \[ \Pi_{\phi}^{\ABV}(G) \subseteq \bigcup_{\phi'\geq_C \phi} \Pi_{\phi'}(G) \subseteq \bigcup_{\phi'\in \Phi_{\cO}(G)} \Pi_{\phi'}(G). \]
    Taking Aubert-Zelevinsky involution on the above inclusion,   Working Hypothesis \ref{assu Arthur ABV}(b) implies that 
    \[\Pi_{\widehat{\phi}}^{\ABV} (G) \subseteq \bigcup_{\phi'\in \Phi_{\cO}(G)} \widehat{\Pi}_{\phi'}(G). \]
This proves the other direction of the containment for the second equality and completes the proof of this proposition.
\end{proof}

\subsection{Proof of Weak local Arthur packet conjecture for even orthogonal group}
In this subsection, we prove Conjecture \ref{conj weak local A-packets} assuming the following conjecture and Working Hypothesis \ref{assu Arthur ABV}(a) for desiderata of local Arthur packets.

\begin{conj}\label{conj weak LAP Lpar}
    The set $\Phi_{\cO}(G)$ consists of $L$-parameters of Arthur type.
\end{conj}
The conjecture holds for $\Sp_{2n}, \SO_{2n+1}$ and $\OO_{2n}$ (\cite[Proposition 1.6]{LL24}). Also, it holds for the minimal distinguished nilpotent orbit of $\mathrm{F}_4$ (\cite{BL25}). It is a work in progress of the authors to verify the conjecture in general.

\begin{thm}
Let $G$ be a pure inner form of a split group.
    Suppose that there is a theory of local Arthur packets for $G$ satisfying Working Hypothesis \ref{assu Arthur ABV}(a).
    Assume that Conjecture \ref{conj weak LAP Lpar} holds and $p \gg 0$. Then Conjecture \ref{conj weak local A-packets} holds. Moreover, we have 
    \[ \Pi_{\cO}^{\weak}(G)= \bigcup_{ \{ \psi \ | \ \phi_{\psi}\in \Phi_{\cO}(G)\}} \Pi_{\widehat{\psi}}(G). \]
\end{thm}
\begin{proof}
    The proof is identical to the proof of Proposition \ref{prop ABV weak LAP} after the following substitution. For each $\phi \in \Phi_{\cO}(G)$, we take a $\psi$ such that $\phi=\phi_{\psi}$ by Conjecture \ref{conj weak LAP Lpar}. Then we replace $\Pi_{\phi}^{\ABV}(G)$ with $\Pi_{\psi}$ and $\Pi_{\widehat{\phi}}^{\ABV}(G)$ with $\Pi_{\widehat{\psi}}$.
\end{proof}

Though Working Hypothesis \ref{assu Arthur ABV}(a) is stated for connected groups, the analogue for even orthogonal groups holds (see Remark \ref{rmk assum ABV arthur}(a)), and the above proof works. Note that the analogue of Theorem \ref{thm CMO WF formula} holds for pure inner forms of split $\OO_{2n}(F)$ since $\WF(\pi)$ is the same as $\WF(\pi|_{\SO_{2n}(F)})$.
\begin{cor}\label{cor weak LAP orthog}
    Conjecture \ref{conj weak local A-packets} holds for pure inner forms of split $\OO_{2n}(F)$ when $p \gg 0$.
\end{cor}

\section{Unitarity of representations of good parity}\label{sec unitary rep}
In \cite[Conjecture 1.2]{HJLLZ25}, we conjecture for classical groups that a good parity representation is unitary if and only if it is of Arthur type. A refined version for split symplectic and special odd orthogonal groups is proved in \cite[Theorem 3.1]{AM25}: if $\pi$ is unitary, then the good parity part $\pi_{\good}$ (see Definition \ref{defn good parity rep}) is of Arthur type. In this section, we
extend this result to even orthogonal groups $\OO(W_n)$ (Theorem \ref{thm unitary O}).

Recently, we learned from Chen and Zou that they also proved the following theorem including metaplectic groups and non-quasi-split special odd orthogonal groups, using global theta lifting (\cite{CZ26}).

\begin{thm}\label{thm unitary O}
    Let $\pi\in\Pi(\OO(W_n))$ be unitary. Then $\pi_\good$ is of Arthur type.
\end{thm}
\begin{proof}
Note that $(\pi\otimes\det)_\good=\pi_\good\otimes\det.$
    Also $\pi$ is unitary if and only if $\pi\otimes\det$ is unitary. Thus, without loss of generality, we assume that $\pi$ is the representation corresponding to the going-down tower (see \S\ref{subsec theta correspondence}).

    Since $\pi$ corresponds to the going-down tower, it follows from \cite[Theorem 4.3]{AG17a} and \cite[Theorem 6.7]{BH21} that if the non-tempered portion of the $L$-data of $\pi$ is $\{\Delta_{\rho_1}[x_1,y_1],\ldots, \Delta_{\rho_f}[x_f, y_f]\}$, then the non-tempered portion of the $L$-data of $\theta_{-\alpha}(\pi)$ is given by 
    \begin{align*}
       & \{\Delta_{ \chi_W \chi_{V}^{-1}\rho_1}[x_1,y_1],\ldots, \Delta_{\chi_W \chi_{V}^{-1}\rho_f}[x_f, y_f]\}\\
        +& \{\Delta_{\chi_W}[\frac{1-\alpha}{2},\frac{1-\alpha}{2}],\Delta_{\chi_W}[\frac{3-\alpha}{2},\frac{3-\alpha}{2}],\dots,\Delta_{\chi_W}[-1,-1]\}.
    \end{align*} 
Also, if the $L$-parameter of the tempered part of $\pi$ is $\phi_{temp}$, then the  $L$-parameter of the tempered part of $\theta_{-\alpha}(\pi)$ is $\chi_{W}\chi_V^{-1}\phi_{temp}+ \chi_W \otimes S_1$.   
    In particular, we have that $\theta_{-\alpha}(\pi)_\good=\theta_{-\alpha}(\pi_\good).$
    By taking $\alpha\gg 0$, we may assume that we are in stable range and so $\theta_{-\alpha}(\pi)$ is also unitary (\cite[Theorem A]{Li89}). From the symplectic case (\cite[Theorem 3.1]{AM25}), we obtain that $\theta_{-\alpha}(\pi_\good)$ is of Arthur type.

    Next, we show that $\theta_{-\alpha}(\pi_\good)$ lies in $\Pi_\psi$ for some local Arthur parameter $\psi$ which contains $\chi_W\otimes S_1 \otimes S_\alpha$ as a summand with multiplicity one. Indeed, we have $\theta_{-\alpha}(\pi_{\good})= \{\Delta_{\chi_W}[(1-\alpha)/2, (1-\alpha)/2]\}+\theta_{-\alpha}(\pi_{\good})^{\chi_W,-}$ since $\alpha \gg 0$. Therefore, Theorem \ref{thm pi_rho_minus}(c) implies that $\psi^{max}( \theta_{-\alpha}(\pi_{\good}))$ contains exactly one copy of $\chi_W\otimes S_1 \otimes S_\alpha$.    

    Let $\psi'$ be the local Arthur parameter of $\OO(W_m)$ {(for some $m \leq n$)} such that
    \[ \psi= \psi' \otimes \chi_{W} \chi_V^{-1} \oplus \chi_W \otimes S_1 \otimes S_{\alpha}.\]
    Recall from Proposition \ref{prop inverse theta} that there exists a representation $\pi'$ that lies in the $\Pi_{\psi'}(\OO(W_m^+)) \sqcup \Pi_{\psi'}(\OO(W_m^-))$ such that $\theta_{-\alpha}(\pi')= \theta_{-\alpha}(\pi_{\good})$. Then Howe duality for pure inner forms (Lemma \ref{lem Howe duality pure}) implies that $\pi'= \pi_{\good} $. Therefore, $\pi_{\good} \in \Pi_{\psi'}(\OO(W_m))$ is of Arthur type. This completes the proof of the theorem.
\end{proof}

\begin{remark}
    When $W_n\in\mathcal{W}^+$, we have that the converse is true: any representation of Arthur type is unitary. However, our definition of the local Arthur packet for $W_n\in\mathcal{W}^-$ does not immediately imply that representations of Arthur type are unitary. We note that this is still expected.
\end{remark}

Next, we extend the result to special even orthogonal groups. First, we need to explain what $\pi_{\good}$ is since Definition \ref{defn good parity rep} uses the local Langlands correspondence, which we do not have for $\SO(W_n)$. We define $\pi \in \Pi(\SO(W_n))$ to be of good parity if and only if every subquotient of $\mathrm{Ind}_{\SO(W_n)}^{\OO(W_n)} \pi$ is of good parity. This is also equivalent to $\pi$ appears as a subquotient of a restriction of a good parity representation of $\OO(W_n)$. Now we define $\pi_{\good}$ up to outer conjugation as follows: Take any irreducible subquotient $\pi^+$ of $ \mathrm{Ind}_{\SO(W_n)}^{\OO(W_n)} \pi $. Then we define $\pi_{\good}$ to be any choice of the irreducible subquotient of $\textrm{Res}_{\SO(W_n)}^{\OO(W_n)} (\pi^+)_{\good}$. One can give a more refined definition without the ambiguity of the outer conjugation using Jantzen's decomposition for even special orthogonal groups (\cite{JL22}), {the details of which we omit}.

Next, we explain what Arthur type is. Let $G_n=\OO(W_n)$ and $\pi$ be a representation of $G_n^\circ$ (i.e., $G_n^\circ=\SO(W_n)$). We say that $\pi$ is of Arthur type if its $\Sigma_0$-orbit lies in some local Arthur packet of $G_n^\circ$ (see \S\ref{sec Local Arthur packets for orthogonal groups} for notation). Equivalently, $\pi$ is a restriction of a representation of $G_n$ of Arthur type. Note that $\pi$ is of Arthur type if and only if so is its outer conjugation. Therefore, it is well-defined whether $\pi_{\good}$ is of Arthur type or not.

Now we extend the result for special even orthogonal groups.

\begin{thm}\label{thm unitary SO}
    Let $G_n=\OO(W_n)$ and $\pi$ be an irreducible admissible unitary representation of $G_n^\circ$ (i.e., $G_n^\circ=\SO(W_n)$). Then $\pi_\good$ is of Arthur type.
\end{thm}

\begin{proof}
The proof is based on the fact that $\pi$ is unitary if and only if $\mathrm{Ind}_{G_n^\circ}^{G_n}\pi$ is unitary. 

Now we suppose that $\pi$ is unitary. We have two cases based on the following. If $\pi$ is isomorphic to its outer conjugation, then $\mathrm{Ind}_{G_n^\circ}^{G_n}\pi= \pi_1\oplus\pi_2$, where $\pi_1,\pi_2\in\Pi(G_n)$ are unitary and $\pi_1\neq\pi_2.$ Otherwise, $\mathrm{Ind}_{G_n^\circ}^{G_n}\pi\in \Pi(G_n).$ 

Suppose first that $\mathrm{Ind}_{G_n^\circ}^{G_n}\pi= \pi_1\oplus\pi_2$. By definition, $\pi_{\good}$ is any irreducible subquotient of $\textrm{Res}_{G_n^{\circ}}^{G_n}( (\pi_1)_{\good})$. By Theorem \ref{thm unitary O}, $(\pi_1)_{\good}$ is of Arthur type since $\pi_1$ is unitary. Therefore, $ \pi_{\good}$ is of Arthur type by definition.

    Suppose next that $\mathrm{Ind}_{G_n^\circ}^{G_n}\pi\in \Pi(G_n).$
    Then $\mathrm{Ind}_{G_n^\circ}^{G_n}\pi$ is unitary and irreducible. Therefore, by Theorem \ref{thm unitary O}, $(\mathrm{Ind}_{G_n^\circ}^{G_n}\pi)_\good=\mathrm{Ind}_{G_m^\circ}^{G_m}\pi_\good$ is of Arthur type. Therefore $\pi_\good$ must also be Arthur type.
\end{proof}

\section{Metaplectic groups and pure inner forms of odd special orthogonal groups}\label{sec meta}

In this section, we explain that many results in this paper can be extended to metaplectic groups and pure inner forms of odd special orthogonal groups, based on the recent work of Chen (\cite{Che26}).

More precisely, Chen explicated the analogue of M{\oe}glin's construction for local Arthur packets for metaplectic groups, and rephrased them in terms of extended multi-segments as in Definition \ref{def multi-segment}. Under this parametrization, he proved the Adams conjecture from metaplectic groups to split odd special orthogonal groups \cite[Theorem 7.6]{Che26}. Thus, the analogue of Theorem \ref{thm Moeglin Adams} for metaplectic groups holds. With these results, our main result on intersections of local Arthur packet (Theorem \ref{thm O}) can be proved by a similar argument. The results after \S \ref{sec intersections for O(2n)} can also be proved similarly once we have the analogue of M{\oe}glin's result on $(add_{j}^{\Moe})^{-1}$ (Theorem \ref{thm Moeglin + embed}) for metaplectic groups. However, this expected result does not seem to be recorded in the literature. Alternatively, we expect that Theorem \ref{thm add embed} could be proved directly on metaplectic groups following the proof of \cite[Proposition 6.3]{HLLZ25}. Then the results after \S \ref{sec intersections for O(2n)} holds as well.

We remark that if the analogous result of Theorem \ref{thm Moeglin Adams} (the Adams conjecture) holds for the theta lifting from pure inner forms of odd special orthogonal groups to metaplectic groups, then all of the results in this paper can be proved similarly for pure inner forms of odd special orthogonal groups. Note that Theorem \ref{thm Moeglin + embed} for this case is already proved in M{\oe}glin's work.

\appendix
\section{Extended multi-segments of {B}o{\v s}njak and {S}tadler}\label{appendix: Bosnjak Stadler}

In this appendix, we provide a translation of our main results from the parameterization of local Arthur packets using  extended multi-segments (in the sense of Definition \ref{def multi-segment}(2)) to a modified version introduced by Bo{\v s}njak and Stadler in \cite{BS25}.

\subsection{Notation}
Throughout, given an irreducible self-dual supercuspidal representation $\rho$ of some
$\mathrm{GL}_d(F)$ and a segment $[A,B]_\rho$, we write
\[
a=A+B+1,\qquad b=A-B+1.
\]
Recall from Definition \ref{def multi-segment}(2) that an extended multi-segment for $G_n$ is
\[
\EE=\bigcup_\rho\{([A_i,B_i]_\rho,l_i,\eta_i)\}_{i\in(I_\rho,>)},\qquad
0\le l_i\le \tfrac{b_i}{2},\quad \eta_i\in\{\pm1\}/E ,
\]
where $E=\{\pm1\}$ if $2l_i=b_i$ and $E=\{1\}$ otherwise, the order $>$ on $I_\rho$
is also required to be admissible, i.e., it satisfies
\[
\qquad A_i>A_j \text{ and } B_i>B_j\ \Longrightarrow\ i>j,
\]
and the sign condition holds
\begin{equation}\label{eq:signHLL}
\prod_\rho\prod_{i\in I_\rho}(-1)^{\lfloor b_i/2\rfloor+l_i}\eta_i^{\,b_i}=\epsilon_{G_n}.
\end{equation}

In \cite[Definition 3.3]{BS25}, Bo{\v s}njak and Stadler define a modified version of an extended multi-segment, which we call a $\mu$-extended multi-segment 
\[
\mathcal{S}=\bigcup_{\rho}\bigl\{\bigl(([A_i^\rho,B_i^\rho]_\rho,\mu_i^\rho)\bigr)_{i=1}^{n_\rho}\bigr\},
\qquad \mu_i\in\mathbb{Z},\quad \mu_i\equiv b_i \bmod 2,\quad |\mu_i|\le b_i,
\]
with the same admissibility condition, and with the sign condition written as
\begin{equation}\label{eq:signBS}
\sum_{\rho}\sum_{i=1}^{n_\rho}\Bigl(\Bigl\lfloor\frac{\mu_i}{2}\Bigr\rfloor
+\mu_i\sum_{j=1}^{i-1}(b_j-1)\Bigr)\equiv \epsilon_{G_n}' \bmod 2,
\end{equation}
where $\epsilon_{G_n}'=0$ if $\epsilon_{G_n}=1$ and $\epsilon_{G_n}'=1$ otherwise.
The index set is always written as $\{1,\dots,n_\rho\}$ equipped with its natural order.  We adopt this convention from
now on, and drop $\rho$ from the notation whenever a single $\rho$-part is
under discussion.

We note that sometimes the admissible order on $\EE$ is required to satisfy $(P')$, e.g., in the definition of $dual$ (Definition \ref{dual segment}). The condition $(P')$ agrees precisely with what Bo{\v s}njak and Stadler call a very admissible order in \cite[Definition 3.3(6)]{BS25}.

\subsection{The translation map}
Set
\[
\delta_i:=\prod_{j<i}(-1)^{b_j-1}\in\{\pm1\} .
\]
In \cite[Definition 3.4]{BS25}, Bo{\v s}njak and Stadler give the bijection from $\mu$-extended multi-segments to extended multi-segments. Given $\mathcal{S}$, we define $\EE_{\mathcal{S}}=\bigcup_\rho\{([A_i,B_i]_\rho,l_i,\eta_i)\}_{i\in(I_\rho,>)}$ via
\begin{equation}\label{eq:dict}
l_i=\frac{b_i-|\mu_i|}{2},\qquad
\eta_i=\delta_i\,\sgn(\mu_i)
\end{equation}
with the convention $\sgn(0)=1$. Conversely, given an extended multi-segment $\EE$, we define the associated $\mu$-extended multi-segment $\mathcal{S}_\EE=\bigcup_{\rho}\bigl\{\bigl(([A_i,B_i]_\rho,\mu_i)\bigr)\bigr\}_{i\in I_\rho}$ by setting 
\begin{equation}\label{eq:dict2}
\mu_i=\delta_i\,\eta_i\,(b_i-2l_i).
\end{equation}
Technically, Bo{\v s}njak and Stadler only show that the resulting map defines a bijection when $\epsilon_{G_n}=1$; however, it is straightforward to verify that it also defines a bijection when $\epsilon_{G_n}=-1$ (see also Proposition \ref{prop:sign} below).

Recall from Definition \ref{dual segment} that
\[
\alpha_i=\sum_{j<i}a_j,\qquad \beta_i=\sum_{j>i}b_j.
\]
We record some useful identities below.
\begin{lemma}\label{lem:dict}
Let $\EE$ be an extended multi-segment. Fix $\rho$ and let $n=n_\rho$.  Then, with $\mu_i$ as in Equation \eqref{eq:dict2}, the following hold.
\begin{enumerate}
\item We have $|\mu_i|=b_i-2l_i$ and $\sgn(\mu_i)=\delta_i\eta_i$.
\item The identities $B_i+l_i=\dfrac{a_i-|\mu_i|}{2}$ and $A_i-l_i=\dfrac{a_i-2+|\mu_i|}{2}$ hold.
\item For adjacent $k,k+1$,
$\displaystyle\ \epsilon:=(-1)^{A_k-B_k}\eta_k\eta_{k+1}=\sgn(\mu_k\mu_{k+1})$.
\item Recall from Definition \ref{dual segment} that
$\alpha_i=\sum_{j<i}a_j.$ We have $\delta_i\,(-1)^{\alpha_i}=c_i$, where
\[
c_i=\begin{cases}(-1)^{i-1}, & B_j\in\mathbb{Z},\\[2pt]
1, & B_j\notin\mathbb{Z}.\end{cases}
\]
Consequently $(-1)^{\alpha_i}\eta_i=c_i\sgn(\mu_i)$.
\item We have $(-1)^{\lfloor b_i/2\rfloor+l_i}\eta_i^{\,b_i}=(-1)^{\lfloor \mu_i/2\rfloor}\,\delta_i^{\,\mu_i}$.
\end{enumerate}
\end{lemma}

\begin{proof}
Parts (1) and (2) are immediate from Equation \eqref{eq:dict2}.  For Part (3), note that $A_k-B_k=b_k-1$ and
$\delta_{k+1}=\delta_k(-1)^{b_k-1}$, so
$\sgn(\mu_k)\sgn(\mu_{k+1})=\delta_k\delta_{k+1}\eta_k\eta_{k+1}=(-1)^{b_k-1}\eta_k\eta_{k+1}$.
For Part (4), observe that $\delta_i(-1)^{\alpha_i}=(-1)^{\sum_{j<i}(a_j+b_j-1)}$ and
$a_j+b_j-1=2A_j+1$, which is odd when $A_j\in\mathbb{Z}$ and even when
$A_j\in\frac12+\mathbb{Z}$.
Finally Part (5) is checked case-by-case based on $b_i \bmod 2$ and
$\sgn(\mu_i)$, using $\eta_i^{b_i}=\eta_i^{\mu_i}$.
\end{proof}

\begin{remark}
\label{rem:bookkeeping}
Because $\delta_i$ depends on $b_1,\dots,b_{i-1}$, the coordinate $\mu_i$ is not
attached to the $i$-th extended segment alone, rather it depends on the previous rows.  In particular, we note that if a row whose segment has length $b$ is deleted from (or inserted at)
position $p$, then $\mu_i$ is multiplied by $(-1)^{b-1}$ for every $i>p$ in the same
$\rho$-part.

This appears in two instances later. Specifically, deleting rows after applying a $ui$ of type $3'$ (here $b=0$ and so $\mu_i\mapsto-\mu_i$) and the insertion of the rows $([x+1,x]_\rho,1,1)$ in Definition \ref{def criterion for Arthur type} (here $b=2$ and so again $\mu_i\mapsto-\mu_i$ for the rows
after the inserted row).
\end{remark}

Next we record that $\EE$ is an extended multi-segment of $G_n$ if and only if $\mathcal{S}_\EE$ is a $\mu$-extended multi-segment of $G_n$.

\begin{prop}\label{prop:sign}
Under Equation \eqref{eq:dict}, Equation \eqref{eq:signHLL} is 
equivalent to Equation \eqref{eq:signBS}.
\end{prop}

\begin{proof}
    This is a straightforward application of Lemma \ref{lem:dict}(5).
\end{proof}

\subsection{Nonvanishing Conditions}
In this section, we translate the nonvanishing conditions which requires translating the row exchange operator. Recall from Definition \ref{def non-vanishing} that the combinatorial definition of nonvanishing requires verifying several parts. We begin with the translation of $\NV_{*}(\EE)\neq 0$ (see also \cite[Equation (3.5)]{BS25}).

\begin{prop}\label{prop:nvstar}
We have $\NV_*(\EE)\neq 0$ if and only if for any index $i$, $\mathcal{S}_\EE$ satisfies
\begin{equation}\label{eqn:NVstar}
    |\mu_i|\le a_i \quad (B_i\in\mathbb{Z}),\qquad\qquad |\mu_i-1|\le a_i\quad (B_i\notin\mathbb{Z}).
\end{equation}
\end{prop}

\begin{proof}
Recall Equation $(*)$ from Definition \ref{def non-vanishing}(1).
By Lemma \ref{lem:dict}(2), we have $B_i+l_i=(a_i-|\mu_i|)/2$.  If $B_i\in\mathbb{Z}$, Equation $(*)$ is $B_i+l_i\ge0$,
i.e.\ $|\mu_i|\le a_i$.  If $B_i\notin\mathbb{Z}$ then $c_i=1$ by Lemma \ref{lem:dict}(4), so
$\eta_i=(-1)^{\alpha_i+1}$ if and only if $ \mu_i<0$.
Hence Equation $(*)$ gives
$|\mu_i|\le a_i-1$ when $\mu_i<0$ and $|\mu_i|\le a_i+1$ when $\mu_i\ge0$, i.e.,  $|\mu_i-1|\le a_i$.
\end{proof}

For a $\mu$-extended multi-segment $\mathcal{S}$, we write $\NV_{*}(\mathcal{S})\ne0$ if  Equation \eqref{eqn:NVstar} holds for any index $i$.
Next, we recall the translation of $\widetilde{\NV}_{\mathrm{Xu}}(\EE)\ne0$. 
\begin{prop}[{\cite[Definition 3.7]{BS25}}]\label{prop:nvxu}
We have
$\widetilde{\NV}_{\mathrm{Xu}}(\EE)\ne0$ if and only if for any adjacent $k<k+1$, $\mathcal{S}_\EE$ satisfies
\begin{equation}\label{eqn:BSXu}
    |\mu_{k+1}-\mu_k|\ \le\ |A_{k+1}-A_k|+|B_{k+1}-B_k|.
\end{equation}
\end{prop}

For a $\mu$-extended multi-segment $\mathcal{S}$, we write $\widetilde{\NV}_{\mathrm{Xu}}(\mathcal{S})\ne0$ if for any adjacent $k<k+1$, Equation \eqref{eqn:BSXu} holds.
To complete our translation of $\NV(\EE)\neq 0$ into a combinatorial condition on $\mathcal{S}_\EE$ it only remains to translate the row exchange operators. 

\begin{prop}\label{prop:R}
Consider a $\mu$-extended multi-segment $\mathcal{S}=\mathcal{S}_\rho\cup\mathcal{S}^\rho$ and let $k<k+1 \in I_\rho$ be adjacent.  If neither
$[A_k,B_k]\supseteq[A_{k+1},B_{k+1}]$ nor $[A_k,B_k]\subseteq[A_{k+1},B_{k+1}]$
holds, then $R_k(\mathcal{S}):=\mathcal{S}$. 

Otherwise write $\mathcal{S}_\rho=\cup_{i=1}^n\{([A_i,B_i]_\rho,\mu_i)\}$. Then
$R_k(\mathcal{S}):=R_k(\mathcal{S}_\rho)\cup\mathcal{S}^\rho$, where $R_k(\mathcal{S}_\rho)=\cup_{i=1}^n\{([A_i',B_i']_\rho,\mu_i')\}$ is defined by $([A_i',B_i']_\rho,\mu_i')=([A_i,B_i]_\rho,\mu_i)$ if $i\neq k, k+1$ and \[(([A_k',B_k']_\rho,\mu_k'),([A_{k+1}',B_{k+1}']_\rho,\mu_{k+1}'))\] is given by
\[
\begin{cases}
\bigl(([A_{k+1},B_{k+1}],\ \mu_{k+1}),\ ([A_k,B_k],\ 2\mu_{k+1}-\mu_k)\bigr),
& \mathrm{if} \ [A_k,B_k]\supsetneq[A_{k+1},B_{k+1}],\\[4pt]
\bigl(([A_{k+1},B_{k+1}],\ 2\mu_k-\mu_{k+1}),\ ([A_k,B_k],\ \mu_k)\bigr),
& \mathrm{if} \ [A_k,B_k]\subseteq[A_{k+1},B_{k+1}].
\end{cases}
\]

In this situation, we have that $R_k(\EE_{\mathcal{S}})=\EE_{R_k(\mathcal{S})}$.
\end{prop}

\begin{proof}
Without loss of generality, assume that $\mathcal{S}=\mathcal{S}_\rho$ and write $\EE_{\mathcal{S}_\rho}=\cup_{i=1}^n\{([A_i,B_i]_\rho,l_i,\eta_i)\}$ according to Equation \eqref{eq:dict}.
Suppose $[A_k,B_k]\supsetneq[A_{k+1},B_{k+1}]$. We have $m_i:=|\mu_i|=b_i-2l_i$ and
$s_i=\sgn\mu_i$. In $R_k(\EE_\mathcal{S})$ the extended segment at position $k$ is
$([A_{k+1},B_{k+1}],l_{k+1},(-1)^{b_k-1}\eta_{k+1})$ with 
$\delta'_k=\delta_k$. Thus applying Equation \eqref{eq:dict2} gives
$\delta_k(-1)^{b_k-1}\eta_{k+1}m_{k+1}=\delta_{k+1}\eta_{k+1}m_{k+1}=\mu_{k+1}$ as claimed.
The extended segment of $R_k(\EE_\mathcal{S})$ at position $k+1$ has  $\delta'_{k+1}=\delta_k(-1)^{b_{k+1}-1}$,
and, in each of the three sub-cases (a),(b),(c) of Definition \ref{def row exchange}, a direct
computation gives $\mu_k'=-\mu_k+2s_km_{k+1}$ if $\epsilon=1$ (sub-cases (a) and (b)) and $-\mu_k-2s_km_{k+1}$ if $\epsilon=-1$ (sub-case (c)). By Lemma \ref{lem:dict}(3), we obtain that both equal
$2\mu_{k+1}-\mu_k$. This completes the proof in this case. 

The proof in the case that $[A_k,B_k]\subseteq[A_{k+1},B_{k+1}]$ is similar.
\end{proof}

For two $\mu$-extended multi-segments $\mathcal{S}_1$ and $\mathcal{S}_2$, we write $\mathcal{S}_1=_R\mathcal{S}_2$ if there exists $R_{i_1},\dots, R_{i_s}$ such that
$\mathcal{S}_2=(R_{i_1}\circ\cdots\circ R_{i_s})(\mathcal{S}_1)$. We let $[\mathcal{S}]=\{\mathcal{S}' \ | \ \mathcal{S'}=_R\mathcal{S}\}.$

\begin{remark}\label{rem:reflection}
We remark that Bo{\v s}njak and Stadler give a definition of row exchange in \cite[Definition 3.8]{BS25} and shows that it is equivalent to the version given by Atobe (\cite[\S5.2]{Ato22a}) in \cite[Proposition 3.11]{BS25}. However, Definition \ref{def row exchange} may fail to agree with \cite[\S5.2]{Ato22a}. Indeed, we only have that these definitions (for extended multi-segments) agree when $\NV(\EE)\neq 0.$ The same thus holds for our definition of row exchange for $\mu$-extended multi-segments and that of \cite[Definition 3.8]{BS25}.
For example, when
$[A_k,B_k]=[A_{k+1},B_{k+1}]$ the non-vanishing inequality of
Proposition~\ref{prop:nvxu} forces $\mu_k=\mu_{k+1}$ and so the two
formulas agree in this case. 

Furthermore, just as the row exchange of Definition \ref{def row exchange} may fail to be an extended multi-segment, the row exchange in the above proposition may fail to be a $\mu$-extended multi-segment. Again, assuming the nonvanishing conditions, all of these row exchanges are equivalent and so this becomes a non-issue later.
\end{remark}

For a $\mu$-extended multi-segment $\mathcal{S}$, we write $\NV(\mathcal{S})\ne0$ if $\widetilde{\NV}_{\mathrm{Xu}}(\mathcal{S'})\ne0$ and ${\NV}_{*}(\mathcal{S'})\ne0$ for any $\mathcal{S}'=_R\mathcal{S}$. We refer the reader to \cite[Theorem 3.20]{BS25} for a clean equivalent description of $\NV(\mathcal{S})$, without involving row exchange. 

\subsection{The operators}

In this section, we translate various operators on extended multi-segments into the language of $\mu$-extended multi-segments.

Throughout this section we fix $\rho$ and write $\mathcal{S}_\rho$ as a sequence
$\{([A_i,B_i]_\rho,\mu_i)\}_{i=1}^{n}$ as the operators do not affect $\mathcal{S}^\rho$.
We begin with the translation of the shift and add operators.

\begin{prop}\label{prop:shadd}
Let $\mathcal{S}$ be a $\mu$-extended multi-segment. For a fixed $d\in\mathbb{Z}$ and index $j\in I_\rho$, we define 
 $sh^d_j(\mathcal{S})=sh^d_j(\mathcal{S}_\rho)\cup\mathcal{S}^\rho$, where $sh^d_j(\mathcal{S}_\rho)=\cup_{i=1}^\rho\{([A_i',B_i']_\rho,\mu_i')\}$ with \[([A_i',B_i']_\rho,\mu_i')=([A_i,B_i]_\rho,\mu_i), \ \mathrm{if}\  i\neq j \  \mathrm{and} \ ([A_j',B_j']_\rho,\mu_j')=([A_j+d,B_j+d],\ \mu_j).\]

We also define $add^d_j(\mathcal{S})=add^d_j(\mathcal{S}_\rho)\cup\mathcal{S}^\rho$, where $add^d_j(\mathcal{S}_\rho)=\cup_{i=1}^\rho\{([A_i'',B_i'']_\rho,\mu_i'')\}$ with \[([A_i'',B_i'']_\rho,\mu_i'')=([A_i,B_i]_\rho,\mu_i) \ \mathrm{if} \ i\neq j \ \mathrm{and} \ ([A_j'',B_j'']_\rho,\mu_j'')=([A_j+d,B_j-d],\ \mu_j)\] 
except in the degenerate case below. We emphasize two important cases:
\begin{enumerate}
\item $add^{-1}_j$ is defined (i.e., the result is again a BS-extended
multi-segment) if and only if $|\mu_j|\le b_j-2$;
\item the degenerate case of $add^{d}_j$, in which the $j$-th row is
removed, occurs exactly when $b_j-2d=0$, and then necessarily $\mu_j=0$. By
Remark~\ref{rem:bookkeeping} the removal multiplies $\mu_i$ by $-1$ for all $i>j$.
\end{enumerate}

In any case, we have $\EE_{sh^d_j(\mathcal{S})}=sh^d_j(\EE_\mathcal{S})$ and $\EE_{add^d_j(\mathcal{S})}=add^d_j(\EE_\mathcal{S})$
(recall Definition \ref{def shift and add}).
\end{prop}

\begin{proof}

We begin by proving that $\EE_{sh^d_j(\mathcal{S})}=sh^d_j(\EE_\mathcal{S})$ and $\EE_{add^d_j(\mathcal{S})}=add^d_j(\EE_\mathcal{S})$. Let $(\EE_\mathcal{S})_\rho=\cup_{i=1}^n\{([A_i,B_i]_\rho,l_i,\eta_i)\}$. Note that
for $sh^d_j$, the data $b_j$, $l_j$, $\eta_j$ and all $\delta_i$ are unchanged when comparing $\EE_{\mathcal{S}}$ and $sh^d_j(\EE_\mathcal{S})$ from which the claim follows.  For
$add^d_j$, the change from $\EE_{\mathcal{S}}$ and $add^d_j(\EE_\mathcal{S})$ is given by $b_j\mapsto b_j+2d$ and $l_j\mapsto l_j+d$, so $b_j-2l_j$ is unchanged,
$\eta_j$ is unchanged, and $\delta_i$ is unchanged because $b_j$ changes by the
even integer $2d$. These observations imply $\EE_{add^d_j(\mathcal{S})}=add^d_j(\EE_\mathcal{S})$.

Next, we note that Part (1) follows from the observation that $0\le l_j-1$ gives \ $b_j-|\mu_j|\ge2$. Part 2 follows from Remark~\ref{rem:bookkeeping}. This completes the proof of the lemma.
\end{proof}

Next we study the union-intersection operator.

\begin{prop}\label{prop:ui} Let $\mathcal{S}$ be a $\mu$-extended multi-segment and $k<k+1$ be adjacent.  We say $ui_k$ is applicable on $\mathcal{S}$ at $(k,k+1)$ if and only if
\[
A_k<A_{k+1},\qquad B_k<B_{k+1},\qquad
|\mu_{k+1}-\mu_k|=(A_{k+1}-A_k)+(B_{k+1}-B_k) ,
\]
that is, if and only if the non-vanishing inequality of
Equation \eqref{eqn:BSXu} is an \emph{equality}.  In this case, we define $ui_k(\mathcal{S})=ui_k(\mathcal{S}_\rho)\cup\mathcal{S}^\rho$ as follows. 
Set
$\sigma=\sgn(\mu_{k+1}-\mu_k)$ and $c=\sigma\,(A_{k+1}-A_k)$. Then $ui_k(\mathcal{S}_\rho):=\{([A_i',B_i']_\rho,\mu_i')\}_{i=1}^{n},$ where $([A_i',B_i']_\rho,\mu_i')=([A_i,B_i]_\rho,\mu_i)$ if $i\neq k,k+1$ and
\[
\bigl(([A_k',B_k'],\mu_k'),([A_{k+1}',B_{k+1}'],\mu_{k+1}')\bigr):=
\bigl(([A_{k+1},B_k],\ \mu_k-c),\ ([A_k,B_{k+1}],\ \mu_{k+1}-c)\bigr),
\]
except in the following degenerate case: if $\sgn\mu_k=-\sgn\mu_{k+1}$ and $B_{k+1}=A_k+1$, then $\mu_{k+1}-c=0$, and by Remark~\ref{rem:bookkeeping} the deletion of $([A_k, B_k+1]_\rho,0)$ results in
$\mu_i':=-\mu_i$ for all $i>k+1$.

Then $\EE_{ui_k(\mathcal{S})}={ui_k(\EE_\mathcal{S})}$.
\end{prop}

\begin{proof}
We note that Cases 1, 2, 3 of
Definition \ref{def ui} correspond to
\[
\text{Case 1}:\ \sigma=-\sgn\mu_k,\qquad
\text{Case 2}:\ \sigma=\sgn\mu_k,\qquad
\text{Case 3}:\ \sgn\mu_k=-\sgn\mu_{k+1}.
\]
Let $(\EE_\mathcal{S})_\rho=\cup_{i=1}^n\{([A_i,B_i]_\rho,l_i,\eta_i)\}$.

For applicability, Lemma \ref{lem:dict}(2) turns $A_{k+1}-l_{k+1}=A_k-l_k$ (Case 1) into
$|\mu_{k+1}|-|\mu_k|=-\bigl[(A_{k+1}-A_k)+(B_{k+1}-B_k)\bigr]$, turns
$B_{k+1}+l_{k+1}=B_k+l_k$ (Case 2) into
$|\mu_{k+1}|-|\mu_k|=+\bigl[(A_{k+1}-A_k)+(B_{k+1}-B_k)\bigr]$, and turns
$B_{k+1}+l_{k+1}=A_k-l_k+1$ (Case 3) into
$|\mu_k|+|\mu_{k+1}|=(A_{k+1}-A_k)+(B_{k+1}-B_k)$.  Next,  
write $\epsilon=\sgn(\mu_k\mu_{k+1})$ by Lemma \ref{lem:dict}(3). Note that if $\epsilon=1$ then
 $\bigl||\mu_{k+1}|-|\mu_k|\bigr|=|\mu_{k+1}-\mu_k|$, while if $\epsilon=-1$ then
 $|\mu_k|+|\mu_{k+1}|=|\mu_{k+1}-\mu_k|$.
Combining these observations with Lemma \ref{lem:dict}(3)  gives the applicability claim (and identifies the three
cases as stated above).

For the effect, put $\Delta A=A_{k+1}-A_k$, $\Delta B=B_{k+1}-B_k$; then
$b'_k=b_k+\Delta A$, $b'_{k+1}=b_k-\Delta B$, and $\delta'_{k+1}=\delta_{k+1}(-1)^{\Delta A}$.
In Case 1, $\mu'_k=\delta_k\eta_k(b_k+\Delta A-2l_k)=\mu_k+\sgn(\mu_k)\Delta A$
and $\mu'_{k+1}=\delta_{k+1}\eta_{k+1}(|\mu_{k+1}|+\Delta A)=\mu_{k+1}+\sgn(\mu_{k+1})\Delta A$;
since $\sigma=-\sgn\mu_k=-\sgn\mu_{k+1}$ here, this is $\mu-c$ in both rows.  In
Case 2 both sub-cases $b_k-2l_k\gtrless\Delta A$ give
$\mu'_k=\sgn(\mu_k)(|\mu_k|-\Delta A)=\mu_k-\sigma\Delta A$, and similarly for
$\mu'_{k+1}$.  In Case 3, $\mu'_k=\mu_k+\sgn(\mu_k)\Delta A=\mu_k-\sigma\Delta A$,
and the two sub-cases $l_{k+1}\le l_k$, $l_{k+1}>l_k$ (which by the applicability
identity are exactly $|\mu_{k+1}|\ge\Delta A$ and $|\mu_{k+1}|<\Delta A$) both give
$\mu'_{k+1}=\mu_{k+1}-\sigma\Delta A$. The degenerate case follows from Remark~\ref{rem:bookkeeping}. This completes the proof of the proposition.
\end{proof}

The $dual$ operator was already translated in Bo{\v s}njak and Stadler.

\begin{prop}[{\cite[Definition A.4]{BS25}}]\label{prop:dual}
Let $\mathcal{S}=(([A_i,B_i]_\rho,\mu_i))_{i=1}^n$ satisfy $(P')$.  Define $  dual(\mathcal{S})=(([A_i',B_i']_\rho,\mu_i'))_{i=1}^n$ where $[A_i',B_i']_\rho=[A_{n-i+1},B_{n-i+1}]_\rho$ and
\[
{\quad
\mu_i'=
\begin{cases}
(-1)^{\,n-1}\,\mu_{n+1-i}, & B_i\in\mathbb{Z},\\[4pt]
1-\mu_{n+1-i}, & B_i\notin\mathbb{Z} .
\end{cases}\quad}
\]
Then $dual(\EE_\mathcal{S})=\EE_{dual(\mathcal{S})}$.
\end{prop}

We address the final operator next.
\begin{prop}\label{prop:partialdual}
Assume that $\mathcal{S}_\rho=(([A_i,B_i]_\rho,\mu_i))_{i=1}^n$ is half-integral (i.e, $B_k\in\frac12+\mathbb{Z}$), and
let $b_k=A_k-B_k+1$. We define $dual_k^\pm(\mathcal{S}_\rho)=(([A_i',B_i']_\rho,\mu_i'))_{i=1}^n$  as follows.
\begin{enumerate}
\item We say $dual^+_k$ is applicable if and only if
\[
B_k=\frac{1}{2},\qquad \mu_k=-b_k=-(A_k+\frac{1}{2}),\qquad
B_i<\frac{1}{2}\ \text{ for all } i<k .
\]
\item We say $dual^-_k$ is applicable if and only if
\[
B_k=-\frac{1}{2},\qquad \mu_k=b_k=A_k+\frac{3}{2},\qquad
B_j>-\frac{1}{2}\ \text{ for all } j>k .
\]
\item In either case where it is applicable, the effect is $dual_k^\pm(\mathcal{S}_\rho)$ is defined by
\[
([A_i',B_i']_\rho,\mu_i')=
\begin{cases}
( [A_i,B_i]_\rho,2-\mu_i), & i<k,\\
([A_i, \mp \frac{1}{2}]_\rho, 1-\mu_k), & i=k,\\
( [A_i,B_i]_\rho, -\mu_i), & i>k.
\end{cases}
\]
If $dual_k^\pm$ is not applicable, then  we define $dual_k^\pm(\mathcal{S}_\rho)=\mathcal{S}_\rho.$
\end{enumerate}

Then we have $\EE_{dual_k^\pm({\mathcal{S}_\rho})}=dual_k^\pm(\EE_{\mathcal{S}_\rho})$.
\end{prop}

\begin{proof}
Let $\FF=(\EE_\mathcal{S})_\rho=\cup_{i=1}^n\{([A_i,B_i]_\rho,l_i,\eta_i)\}$.
By Lemma \ref{lem:dict}(4), $c_k=1$ in the half-integral case, so the condition
$(-1)^{\alpha_k}\eta_k=-1$ of Definition
\ref{def partial dual} is $\sgn(\mu_k)=-1$, and
$l_k=0$ is $|\mu_k|=b_k$. This gives the applicability condition in Part (1): $dual_k^+$ applies on $\EE_\mathcal{S}$ if and only if $dual_k^+$ applies on $\mathcal{S}$. Applicability for $dual_k^-$, i.e., Part (2), follows by applying $dual$.

For the justification of Part (3), write \[\FF=\FF_1+\{([A_k,\frac12],0,\eta_k)\}+\FF_2\] as in Definition
\ref{def partial dual} (we will use its notation throughout this proof).  By Proposition~\ref{prop:dual}, $  dual(\FF)$ has
$\mu$-values $\mu_i''=1-\mu_{n+1-i}$. Note that
the segment $[A_k,-\frac12]$ has length $b_k+1$.  Passing to
\[\widetilde{\FF_2}+\{([A_k,\frac12],0,(-1)^{\beta_k+1})\}+\widetilde{\FF_1}\]
changes the middle length from $b_k+1$ to $b_k$, hence (see Remark~\ref{rem:bookkeeping}) we
negate the $\mu$-values of all rows after it, i.e., of $\widetilde{\FF_1}$, and
negate the middle $\mu$-value itself.  Applying $dual$ once more,
the rows of $\widetilde{\widetilde{\FF_1}}$ have $\mu$-values
$1-\bigl(-(1-\mu_i)\bigr)=2-\mu_i$ for $i<k$. The middle row of the
output is $([A_k,-\frac12],0,-\eta_k)$, whose corresponding $\mu$-value is
$-\sgn(\mu_k)(b_k+1)=1-\mu_k$.  Finally $\FF_2$ is unchanged, but the length at the position $k$ has changed from
$b_k$ to $b_k+1$, and so it gives $\mu$-values $-\mu_i$ for $i>k$ by Remark \ref{rem:bookkeeping}.
\end{proof}

Recall that we say the operators $ ui^{-1}_{i,j}$,
$  dual\circ ui_{i,j}\circ  dual$ and $  dual^-_k$ on extended multi-segments are called raising operators. Similarly, we call them raising operators on $\mu$-extended multi-segments and say that $\mathcal{S}$ is absolutely maximal if no raising operator is applicable. From Theorem \ref{thm E max} the raising operators on extended multi-segments determine a unique absolutely maximal extended multi-segment (up to row exchanges). Given a $\mu$-extended multi-segment with $\NV(\mathcal{S})\neq 0,$ the above propositions show that there is also a unique (up to row exchange) absolutely maximal $\mu$-extended multi-segment which we denote by $[\mathcal{S}^{|max|}]$.

\subsection{Algorithms} In this section, we translate Algorithms \ref{algo pi_com} and \ref{algo Arthur type pi^-} into the language of $\mu$-extended multi-segments.
We begin by translating Algorithm \ref{algo pi_com}. 

\begin{algo}\label{alg:Ldata} Consider a $\mu$-extended multi-segment $\mathcal{S}$ with $\NV(\mathcal{S})\neq 0$. \\
\textbf{Step 1.} Compute $\mathcal{S}^{|max|}=\bigcup_\rho(([A_i,B_i]_\rho,\mu_i))_i$.  If $\psi_{\mathcal{S}^{|max|}}$ is tempered, then set
\[
\pi_{\com}(\mathcal{S}):=\pi\Bigl(\bigoplus_\rho\bigoplus_{i}\rho\otimes S_{2A_i+1},\ \varepsilon\Bigr),
\qquad \varepsilon(\rho\otimes S_{2A_i+1})=\mu_i .
\]
\noindent \textbf{Step 2.} Otherwise choose $j$ as in Theorem \ref{thm E max}(2): fix $\rho$ with
$b_\rho:=\max_i b_i>1$, let $j=\min\{i: b_i=b_\rho\}$ and, after row exchanges,
$j=\max\{i: B_i=B_j\}$.  Set $(\mathcal{S}^{|max|})^-:=add^{-1}_j(\mathcal{S}^{|max|})$.
Then
\[
\pi_{\com}(\mathcal{S})=\{\Delta_\rho[B_j,-A_j]\}+\pi_{\com}\bigl((\mathcal{S}^{|max|})^-\bigr).
\]
\end{algo}

The above propositions immediately imply that $\pi_\com(\EE_\mathcal{S})=\pi_\com(\mathcal{S}).$ Next we translate the ingredients needed for the translation of Algorithm \ref{algo Arthur type pi^-}. Let $\pi\in\Pi_{gp}(G_n)$ and $\pi=\pi^{\rho,-}+\{(\Delta_{\rho}[x,-y])^r\}$. 

\begin{itemize}
\item (Definition \ref{def EE_rho_minus}) For $\mathcal{S}$ absolutely maximal, ordered so that
$A_i\ge A_j$ if $i>j$ and $B_i=B_j$, with $b_\rho=\max_i b_i>1$,
$j_1=\min\{i:b_i=b_\rho\}$ and \[\{j_1<\dots<j_s\}=\{j:[A_j,B_j]=[A_{j_1},B_{j_1}]\},\]
we set
\[
\mathcal{S}^{\rho,-}=\sum_{k=1}^s add^{-1}_{j_k}(\mathcal{S}).
\]
If $\pi_\com(\mathcal{S})=\pi$, then $\pi_\com(\mathcal{S}^{\rho,-})=\pi^{\rho,-}$ via Theorem \ref{thm pi_rho_minus}. 

\item (Definition \ref{def criterion for Arthur type}(2)) Let $\mathscr{S}(\pi^{\rho,-}; \Delta_{\rho}[x,-y],r)$ denote the set of $\mu$-extended multi-segments $\mathcal{S}$ with $\NV(\mathcal{S})\neq 0$ and equipped with a very admissible order such that $\pi_\com(\mathcal{S})=\pi^{\rho,-}$ and $\psi_{\mathcal{S}}\in \Psi(\pi^{\rho,-}; \Delta_{\rho}[x,-y],r)$. 

\item (Definition \ref{def criterion for Arthur type}(2)(i) case $y-x=1$) Let $\mathcal{S} \in \mathscr{S}(\pi^{\rho,-}; \Delta_{\rho}[x,-y],r)$ and assume that $y-x=1$. Write $\mathcal{S}_\rho=(([A_i,B_i]_\rho,\mu_i))_{(i\in I_\rho, >)}$. We define \[\mathcal{S}^{\rho,+}=(([A_i',B_i']_\rho,\mu_i'))_{(i\in I_\rho\sqcup \{j_1,\dots ,j_r\}, \gg)}\] where $\gg$ is defined by
\[\begin{cases}
j_r \gg j_{r-1} \gg \cdots \gg j_1,\\
\alpha \gg \beta  \Longleftrightarrow \alpha > \beta & \text{for }\alpha, \beta \in I_{\rho},\\
\alpha \gg j_k \Longleftrightarrow {B_\alpha >x} & \text{for }\alpha \in I_{\rho},
\end{cases}
\]
$([A_i',B_i']_\rho,\mu_i')=([A_i,B_i]_\rho,\mu_i)$ for $i<j_1$, $([A_i',B_i']_\rho,\mu_i')=([x+1,x]_\rho, 0)$ if $i=j_{k}$ for some $k=1,2,\dots,r$, and  $([A_i',B_i']_\rho,\mu_i')=([A_i,B_i]_\rho,(-1)^{r}\mu_i)$ for $i>j_k$. We note that each inserted row has $b=2$, so by
Remark~\ref{rem:bookkeeping} the insertion negates $\mu_i$ for every row after the
inserted rows.

\item (Definition \ref{def criterion for Arthur type}(2)(ii) case $y-x>1$).  Let $\mathcal{S} \in \mathscr{S}(\pi^{\rho,-}; \Delta_{\rho}[x,-y],r)$ and assume that $y-x>1$. Change the admissible order if necessary so that there exists $j_1,\ldots, j_r \in I_{\rho}$ such that
\begin{enumerate}
    \item [$\oldbullet$] $j_1 < \cdots < j_r$ are adjacent under the admissible order on $I_{\rho}$,
    \item [$\oldbullet$] $[A_{j_1},B_{j_1}]_{\rho}= \cdots = [A_{j_r},B_{j_r}]_{\rho}=[y-1,x+1]_{\rho}$,
    \item [$\oldbullet$]  $j_1 = \min\{i \in I_{\rho} \ | \ B_i=B_{j_1}\}$,
    \item [$\oldbullet$] $ A_{j_1}-B_{j_1}+3 \geq A_i-B_i+1$ for all $i \in I_{\rho}$ and the equality does not hold for $i<j_1$.
\end{enumerate}

We set
$\mathcal{S}^{\rho,+}=\sum_{k=1}^r add^{1}_{j_k}(\mathcal{S})$. 
\end{itemize}

We now give the translation of Algorithm \ref{algo Arthur type pi^-}, i.e., an algorithm to determine when a representation is of Arthur type using $\mu$-extended multi-segments.

\begin{algo}\label{alg:arthur}
Let $\pi\in\Pi_{gp}(G_n)$.
\begin{itemize}
\item[\textbf{Step 1.}] If $\pi=\pi(\phi,\varepsilon)$ is tempered with
$\phi=\bigoplus_\rho\bigoplus_i\rho\otimes S_{2A_i+1}$, order $I_\rho$ so that
$a_i=2A_i+1$ is non-decreasing and set
\[
\mathcal{S}:=\bigcup_\rho\bigl(([A_i,A_i]_\rho,\ \mu_i)\bigr)_i,\qquad
\mu_i:=\varepsilon(\rho\otimes S_{2A_i+1})\in\{\pm1\},
\]
and output the set $\mathscr{S}(\mathcal{S})$ of $\mu$-extended multi-segments that can be obtained from $\mathcal{S}$ by a composition of row exchanges, $ui_k$, $dual\circ ui_k\circ dual$, $dual_k^\pm$ and their inverses.  The trivial
representation of $G_0$ corresponds to $\mathcal{S}=\emptyset$.
\item[\textbf{Step 2.}] Otherwise write $\pi=\pi^{\rho,-}+\{(\Delta_\rho[x,-y])^r\}$
as in Definition \ref{def pi minus}. Apply the algorithm on $\pi^{\rho,-}$ to construct the set (possibly empty)
\[\{[\mathcal{S}']:\pi_{\com}(\mathcal{S}')\cong\pi^{\rho,-}\}.\]
The representation $\pi$ is of the form $\pi_\com(\mathcal{S})$ for some $\mu$-extended multi-segment $\mathcal{S}$ if and only if the set $\mathscr{S}(\pi^{\rho,-}; \Delta_{\rho}[x,-y],r) $ is non-empty. If this set is non-empty, take any $\mathcal{S}$ in this set and then output the set $\mathscr{S}(\mathcal{S}^{\rho,+})$ obtained by applying all possible compositions of row exchanges, $ui_k$, $dual\circ ui_k\circ dual$, $dual_k^\pm$ and their inverses to $\mathcal{S}^{\rho,+}$.
\end{itemize}
\end{algo}

\end{document}